\documentclass[a4paper,twoside,11pt]{article}
\usepackage[utf8]{inputenc}
\usepackage[english]{babel}
\usepackage[margin=2.5cm]{geometry} 

\usepackage{amssymb, amsfonts, amsthm, amsmath}
\usepackage[hidelinks]{hyperref}
\usepackage{caption} 
\usepackage{pgf,tikz} 
\usetikzlibrary{math,arrows,calc,through}
\usepackage[normalem]{ulem} 

\usepackage[mathscr]{euscript}
\usepackage[percent]{overpic}
\usepackage{graphicx,accents}
\usepackage{caption} 
\usepackage{enumitem} 
\setlist[enumerate, 1]{label=(\roman*)}

\usepackage{color}

\usepackage{mathtools}
\mathtoolsset{showonlyrefs}

\usepackage{thmtools}
\declaretheorem[style=plain,name=Theorem,qed={\tiny$\blacksquare$},numberwithin=section]{theorem}
\declaretheorem[style=plain,name=Corollary,sibling=theorem,qed={\tiny$\blacksquare$}]{corollary}
\declaretheorem[style=definition,name=Definition,sibling=theorem,qed={\tiny$\blacksquare$}]{definition}
\declaretheorem[style=plain,name=Lemma,sibling=theorem,qed={\tiny$\blacksquare$}]{lemma}

\declaretheorem[style=definition,name=Remark,sibling=theorem,qed={\tiny$\blacksquare$}]{remark}

\numberwithin{equation}{section}

\renewcommand{\lor}{{\mathcal{\mathbf{L^3}}}}
\newcommand{\cyc}{{\mathcal{\mathbf{C}}}}
\newcommand{\eucl}{{\mathcal{\mathbf{E^2}}}}
\newcommand{\mob}{\mathcal{M}}

\renewcommand{\sp}{\mathcal{S}}
\newcommand{\osp}{\vec{\mathcal{S}}}
\newcommand{\li}{\mathcal{L}}
\newcommand{\oli}{\vec{\mathcal{L}}}
\newcommand{\ci}{\mathcal{C}}
\newcommand{\oci}{\vec{\mathcal{C}}}
\newcommand{\pl}{\mathcal{P}}
\newcommand{\opl}{\vec{\mathcal{P}}}

\newcommand{\Zb}{\Z_\bullet} 
\newcommand{\Zw}{\Z_\circ} 

\newcommand{\cp}{p} 
\newcommand{\cpf}{p_{\fboxsep=-\fboxrule\sbox0{}\wd0=4pt\ht0=4pt\relax\fbox{\box0}}}
\newcommand{\cpm}{{\odot{p}}}
\newcommand{\cpb}{p_{\bullet}}
\newcommand{\cpw}{p_{\circ}}
\newcommand{\cpmb}{{\odot{p_{\bullet}}}}
\newcommand{\cpmw}{{\odot{p_{\circ}}}}

\newcommand{\lcp}{u} 
\newcommand{\lcpf}{u_{\fboxsep=-\fboxrule\sbox0{}\wd0=4pt\ht0=4pt\relax\fbox{\box0}}}
\newcommand{\lcpm}{{\odot{u}}}
\newcommand{\lcpb}{u_{\bullet}}
\newcommand{\lcpw}{u_{\circ}}
\newcommand{\lcpmb}{{\odot{u_{\bullet}}}}

\newcommand{\iso}{h} 
\newcommand{\isof}{h_{\fboxsep=-\fboxrule\sbox0{}\wd0=4pt\ht0=4pt\relax\fbox{\box0}}}
\newcommand{\isom}{{\odot{h}}}
\newcommand{\isob}{h_{\bullet}}
\newcommand{\isow}{h_{\circ}}
\newcommand{\isomb}{\mathord{\odot}h_{\bullet}}
\newcommand{\isomw}{\mathord{\odot}h_{\circ}}
\newcommand{\isoc}{h_\boxplus}
\newcommand{\cc}{c}
\newcommand{\ccf}{c_{\fboxsep=-\fboxrule\sbox0{}\wd0=.8ex\ht0=.8ex\relax\fbox{\box0}}}
\newcommand{\ccm}{{\odot{c}}}
\newcommand{\ccb}{c_{\bullet}}
\newcommand{\ccw}{c_{\circ}}
\newcommand{\ccmb}{{\odot{c_{\bullet}}}}
\newcommand{\ccmw}{{\odot{c_{\circ}}}}
\newcommand{\ccc}{c_\boxplus}

\newcommand{\smallc}{\mathrm{sc}} 
\renewcommand{\d}{\mathrm d} 
\newcommand{\si}[1]{\texttt{(#1)}} 
\newcommand{\id}{\operatorname{id}}
\newcommand{\miq}[1]{\mathscr{M}(#1)}
\newcommand{\miqb}[1]{\mathscr{M}_\bullet(#1)}
\newcommand{\miqw}[1]{\mathscr{M}_\circ(#1)}

\newcommand{\lorlift}[1]{\ensuremath{#1^{\mathcal{L}}}}
\newcommand{\cyclift}[1]{\ensuremath{#1^{\mathcal{C}}}}

\newcommand{\Z}{\mathbb{Z}}
\newcommand{\R}{\mathbb{R}}
\newcommand{\C}{\mathbb{C}}

\newcommand{\RP}{{\mathbb{R}\mathrm{P}}}

\newcommand{\centersOf}[1]{\ensuremath{\odot(#1)}} 

\newcommand{\RN}[1]{%
	\textup{\uppercase\expandafter{\romannumeral#1}}%
}

\newcommand{\sca}[1]{\left<#1\right>} 
\newcommand{\lorsca}[1]{\left<#1\right>} 
\newcommand{\cycsca}[1]{\left<#1\right>_{2, 2}} 
\newcommand{\mobsca}[1]{\left<#1\right>_{3, 2}} 

\setlist[enumerate]{itemsep=2pt, topsep=0pt}

\title{Discrete Lorentz surfaces and s-embeddings I:\\ isothermic surfaces}
\date{\today}

\author{
	Niklas Christoph Affolter\thanks{TU Berlin, Institute of Mathematics, Straße des 17.~Juni 136, 10623 Berlin, Germany.
	\textit{E-mail addresses}: \texttt{affolter~at~posteo.net, smeenk~at~math.tu-berlin.de}}\ \footnotemark[2] ,
  Felix Dellinger\thanks{TU Wien, Institut of Discrete Mathematics and Geometry, Wiedener Hauptstr. 8-10/104, A-1040 Vienna, Austria.
  \textit{E-mail addresses}: \texttt{felix.dellinger, christian.mueller, denis.polly~at~tuwien.ac.at}} ,\\
	Christian Müller\footnotemark[2] ,
	Denis Polly\footnotemark[2] ,
	Nina Smeenk\footnotemark[1] ,
}

\begin{document}

\maketitle

\begin{abstract}
  S-embeddings were introduced by Chelkak as a tool to study the conformal invariance of the thermodynamic limit of the Ising model. Moreover, Chelkak, Laslier and Russkikh introduced a lift of s-embeddings to Lorentz space, and showed that in the limit the lift converges to a maximal surface. They posed the question whether there are s-embeddings that lift to maximal surfaces already at the discrete level, before taking the limit. This paper is the first in a two paper series, in which we answer that question in the positive.   
  In this paper we introduce a correspondence between s-embeddings (incircular nets) and congruences of touching Lorentz spheres. This geometric interpretation of s-embeddings enables us to apply the tools of discrete differential geometry. We identify a subclass of s-embeddings -- isothermic s-embeddings -- that lift to (discrete) S-isothermic surfaces, which were introduced by Bobenko and Pinkall. S-isothermic surfaces are the key component that will allow us to obtain discrete maximal surfaces in the follow-up paper. Moreover, we show here that the Ising weights of an isothermic s-embedding are in a subvariety. 
\end{abstract}

\newpage

\setcounter{tocdepth}{1}
\tableofcontents

\newpage

\label{sfc:bla}\label{sed:bla}
\label{lfm:bla}\label{len:bla}
\label{ti:bla}
\label{dff:bla}\label{deg:bla}
\label{er:bla}


\section{Introduction}

The goal of this paper is to relate developments in discrete
differential geometry to recent developments in statistical
mechanics. We will therefore quickly introduce both fields and start with discrete differential geometry.

The aim of \emph{discrete differential geometry} is to develop discrete versions of results of differential geometry, such that the results are not just approximating (in a numerical sense) but also \emph{structure preserving}. Ideally, the discretized objects should satisfy discretized equations, exhibit discretized invariant quantities, have the same symmetry groups and should also have a discretized integrable transformation theory. For a general reference we refer to the book~\cite{bsddgbook}, or for a quick introduction to~\cite{bsorganizing07}.

A particularly successful area of discrete differential geometry is the study of \emph{isothermic surfaces}, which are surfaces for which a \emph{conformal} curvature line parametrization exists. A curvature line parametrization is a parametrization that is \emph{conjugate} and \emph{orthogonal}. It turns out one can discretize these notions consistently. In particular, there is a theory of discrete conjugate nets \cite{sauerqnet, dsqnet}, which are quad meshes with planar faces. A special case of discrete conjugate nets are the so-called \emph{circular nets}~\cite{cdscircular, bobenkocircular}, which are discrete curvature line parametrizations, and are defined as quad meshes with circular faces. Furthermore, the special case of discrete conformal curvature line parametrizations known as \emph{discrete isothermic nets} was developed in \cite{bpisosurf, bpisoint}, where the cross-ratio of the four vertices of each face is $-1$. In analogy to the smooth theory, every discrete isothermic net has a discrete \emph{Christoffel dual} which is also a discrete isothermic net.

Another characterization of isothermic surfaces is that their curvature line parametrization is a \emph{K{\oe}nigs net}. A discrete theory for K{\oe}nigs nets was developed in \cite{doliwakoenigs, bskoenigs, doliwabquad}. Hence, one can combine these two discrete notions to obtain an a priori new definition of discrete isothermic nets \cite{bskoenigs}. However, it turns out that the new definition of discrete isothermic nets is actually equivalent to a slightly more general version of the old definition already present in \cite{bpisosurf, bpisoint}, where cross-ratio $-1$ is replaced with factorizing cross-ratio.

Furthermore, a slightly more rigid special case of isothermic nets called \emph{S-isothermic nets} (short for \emph{Schramm isothermic nets}) was introduced in \cite{bpisoint}. S-isothermic nets have proven very practical for the study of \emph{discrete minimal surfaces} \cite{bhsminimal}. S-isothermic nets may be defined as quad meshes with touching incircles. Moreover, every S-isothermic net comes with a sphere per vertex such that adjacent spheres are touching, they are also K{\oe}nigs and they can be subdivided into a discrete isothermic net.

\emph{Minimal surfaces} are surfaces with vanishing mean curvature. Minimal surfaces are also isothermic. In particular, a minimal surface is an isothermic surface such that its Christoffel dual is its Gauß map. The latter characterization may be used to define discrete minimal surfaces as a special case of discrete S-isothermic nets. Later on, it was shown that these discrete minimal surfaces have indeed vanishing discrete mean curvature \cite{lpwywconical, bpwcurvatureparallel, bsddgbook}. 

The theory of discrete isothermic nets was translated to various space forms \cite{bhjrscmc}, and to Lorentz space \cite{kwisolorentz, bhjrweingarten, yasumotomaximal}. This is relevant for us since we actually study an analogue of minimal surfaces in $\R^3$, namely spacelike maximal surfaces in $\R^{2,1}$. We study discrete S-isothermic nets in $\R^{2,1}$, which are defined completely analogously in terms of incidence geometry to discrete S-isothermic nets in $\R^3$, and are also considered in detail in an upcoming publication \cite{bhscmc}.

We have only given a limited exposition of the developments in discrete isothermic and minimal surface theory. Notably, there is another theory of discrete curvature line parametrizations called \emph{conical nets} \cite{lpwywconical}, and its special case of \emph{S-conical nets} \cite{bhsconical, bkhssconical}, as well as the corresponding minimal surfaces. A unifying perspective on discrete minimal surface theory via discrete holomorphic quadratic differentials can be found in \cite{lamcritical, lpquaddiff}. There are also new developments unifying \cite{bsorganizing07} and generalizing \cite{dellingercheckerboard, atbinets} the two discretizations (circular and conical) of curvature line parametrizations. Additionally, there is work on discrete isothermic and discrete minimal surfaces which are \emph{not} represented by discrete principal parametrizations \cite{lamcritical, hsfwparamsurf, pwkrpmsymmetricmeshes, hksycmcongeneralgraphs, hsfsskew, hsfwbonnet}.
Also the so-called associated family in \mbox{\cite{bhsminimal}} is not given in principal parametrizations. Moreover, there is also rich literature on discrete isothermic surfaces with not necessarily vanishing constant mean curvature \cite{bpisoint, hjhpdarboux, bsddgbook, bhjrscmc, bhsconical, bhscmc}.

After this brief introduction into the topics of discrete differential geometry that are relevant for our paper, we will now continue with the related recent developments in statistical mechanics.

We start with the so-called \emph{isoradial graphs} \cite{ksisoradialembeddings}. They were used to prove conformal invariance of the Ising model \cite{smirnovconformalrandomcluster, csising}. They were also used to study the dimer and spanning tree model and the theory of discrete harmonic and discrete holomorphic functions \cite{kenyonisoradial}. The latter is considered to be the \emph{linear} discrete analytic function theory from the perspective of discrete differential geometry \cite{bmsanalytic}. There is a whole avalanche of research using isoradial graphs, and we refer the reader to the survey paper \cite{bdtisosurvey}. More recently, a generalization of isoradial graphs, to so-called \emph{s-embeddings} \cite{chelkaksembeddings,chelkaksgraphs}, was introduced specifically to investigate the Ising model. These s-embeddings are maps from quad graphs to the plane such that the image of each quad has an incircle. As such, they are also generalizations of \emph{incircular nets} \cite{abconfocal, bstincircular}, which were studied in the discrete differential geometry community. Subsequently, a more general type of maps was used to study the dimer model, the so-called \emph{t-embeddings} \cite{amiquel, klrr, clrdimer}. These t-embeddings are also known in the discrete differential geometry community as \emph{(planar) conical nets} \cite{muellerconical}.

It was also shown in \cite{athesis}, that s- and t-embeddings are closely related to discrete conjugate nets and discrete line congruences, at least from an algebraic and discrete integrable point of view. However, what was missing so far is if and how these embeddings relate to the discrete theory of isothermic or minimal surfaces. This question got even more intriguing with recent progress showing that s-embeddings can be lifted to certain discrete surfaces in $\R^{2,1}$ that converge to maximal surfaces in the thermodynamic limit \cite{chelkaksgraphs, clrpembeddings,clrdimer} (note that they call them minimal surfaces, which is a slight misnomer). These lifts have recently also been used in a number of other papers \cite{cimuniversality,mahfoufthesis,mahfoufcrossingestimates, bnrtembaztec, bnrtemblozenge, craztec}. In fact, in \cite{clrpembeddings} they posed the question whether there are some s-embeddings such that their lifts can be interpreted as discrete maximal surfaces, before taking limits. 

This paper is part of a two-paper series (with \cite{admpsmaximal}), the main result of the two papers is to show that the answer is \emph{yes}: there are s-embeddings that are related to discrete maximal surfaces. In this paper we pave the way by introducing a special class of s-embeddings which we call \emph{isothermic s-embeddings}. An isothermic s-embedding is characterized by the fact that there is a second s-embedding with the same incircles. We show that isothermic s-embeddings are exactly those s-embeddings that lift to S-isothermic nets in $\R^{2,1}$.
As a consequence, each isothermic s-embedding comes with a Christoffel dual partner which is also an isothermic s-embedding (but does not have the same incircles).

Moreover, we show that the $X$-variables that define the corresponding Ising model are in a subvariety that we call the \emph{isothermic subvariety}. Thereby, we also show that not every (planar) Ising model can be realized as an isothermic s-embedding.

We also discuss the symmetry under transformations of the objects we introduce. 

Isothermic surfaces are \emph{Möbius invariant}, as are the discrete S-isothermic surfaces. In fact, the lifts of s-embeddings are generally Möbius invariant, although it is unclear what the corresponding smooth statement would be. However, the $X$-variables are Lorentz invariant, but not Möbius invariant. This could imply that they are not the right quantities to describe isothermic s-embeddings. On the other hand, there are certain \emph{conformal $X$-variables} introduced in \cite{amiquelmobius} for t-embeddings that \emph{are} Möbius invariant, so they are an alternative candidate. Even more intriguingly, it turns out that the $X$-variables are more than Lorentz invariant, they are \emph{Laguerre invariant}. Unfortunately, while a theory of discrete Laguerre isothermic surfaces exists \cite{ppyweierstrass}, a theory of discrete Laguerre S-isothermic surfaces is yet to be discovered. Finally, let us note that the lifts of t-embeddings are actually \emph{Lie invariant}, thus one might also be interested in finding new $X$-variables that are Lie invariant as well.

In the follow-up paper \cite{admpsmaximal} we will specialize our results from discrete isothermic to discrete maximal surfaces and s-embeddings. Using the technique of \cite{bhsminimal}, we explain how discrete maximal surfaces are obtained from hyperbolic orthogonal circle patterns. We will also investigate the 1-parameter associated family of a discrete maximal surface, which also provides an associated family of s-embeddings, and show that the $X$-variables are constant in this family.

Let us briefly discuss possible directions of future investigations. As mentioned, there is also the notion of discrete S-conical nets \cite{bhsconical,bkhssconical}, another discretization of isothermic surfaces in curvature line parametrization.
In fact, an S-conical net is -- similar to an S-isothermic net -- also defined by two families of touching spheres with certain additional constraints. Thereby, it is immediately clear from our results that each S-conical net corresponds to a pair of s-embeddings, but we did not (yet) investigate this direction further. Additionally, there are also definitions for discrete constant mean curvature surfaces both in the S-isothermic and in the S-conical setup. Both these classes of surfaces will also correspond to special classes of s-embeddings, another direction of research that appears to be promising.

Moreover, our geometric construction of the lift of isothermic s-embeddings via touching spheres in $\R^{2,1}$ actually works for general s-embeddings and even t-embeddings, not only isothermic s-embeddings. However, in this case we do not know if these lifts correspond to special surfaces or parametrizations, or if there are interesting other special cases. In particular, it would be interesting if the lift of harmonic embeddings \cite{klrr, tutteembedding}, also known as Tutte embeddings, corresponds to special surfaces or if there is a special case that does.

\subsection{Plan of the paper}

In Section~\ref{sec:mainresults} we lay out the main results including the main definitions and theorems. In Section~\ref{sec:lorentzgeometry} we briefly recall the necessary preliminaries of Lorentz geometry and the associated sphere geometries: Möbius, Laguerre and Lie geometry. In Section~\ref{sec:cccp} we explain the key relation between certain sphere congruences -- called contact congruences -- in three-dimensional Lorentz space and circle patterns (and t-embeddings) in two-dimensional Euclidean space. In Section~\ref{sec:cyclift} we explain how contact congruences are understood as a direct geometric interpretation of the origami lift introduced by Chelkak, Laslier and Russkikh -- without the need to consider $\R^{2,2}$. In Section~\ref{sec:laguerrecp} we provide a Laguerre geometric version of circle patterns that we call cycle patterns. In Section~\ref{sec:nullcongruences} we introduce null congruences, a special case of contact congruences that correspond to incircular nets (s-embeddings) as a special case of circle patterns (t-embeddings). In Section~\ref{sec:miquel} we explain Miquel dynamics, which enable us to characterize isothermic congruences in Section~\ref{sec:isothermiccongruences}. In Section~\ref{sec:isothermicnets} we show that isothermic congruences correspond to S-isothermic nets. In Section~\ref{sec:xvariables} we show that for isothermic incircular nets the X-variables, the variables that define the associated Ising model, are in a subvariety. In Section~\ref{sec:transformations} we discuss the transformation behaviour of the objects we introduced.

\subsection*{Acknowledgments}

N.C.~Affolter and N.~Smeenk were supported by the Deutsche Forschungsgemeinschaft (DFG) Collaborative Research Center TRR 109 ``Discretization in Geometry and Dynamics''.  N.C.~Affolter was also supported by the ENS-MHI chair funded by MHI. 
F.~Dellinger and C.~M{\"u}ller gratefully acknowledge the support by the Austrian Science Fund (FWF) through grant I~4868 (SFB-Transregio “Discretization in Geometry and Dynamics”) and project F77 (SFB “Advanced Computational Design”). D.~Polly was supported by the French National Agency for Research (ANR) via the grant ANR-18-CE40-0033 (DIMERS) for an inspiring research visit to Paris in the fall of 2021.

N.C.~Affolter would like to thank Dmitry Chelkak for pushing for the development of a discrete theory of maximal s-embeddings, as well as him and Misha Bashok and Rémy Mahfouf for countless explanations and discussions of the Lorentz lift. We would also like to thank Jan Techter, Cédric Boutillier, Boris Springborn, Carl Lutz, Paul Melotti, Sanjay Ramassamy, Marianna Russkikh. 



\section{Main results} \label{sec:mainresults}

\begin{figure}
	\centering
	\fbox{
	\begin{overpic}[width=\linewidth]{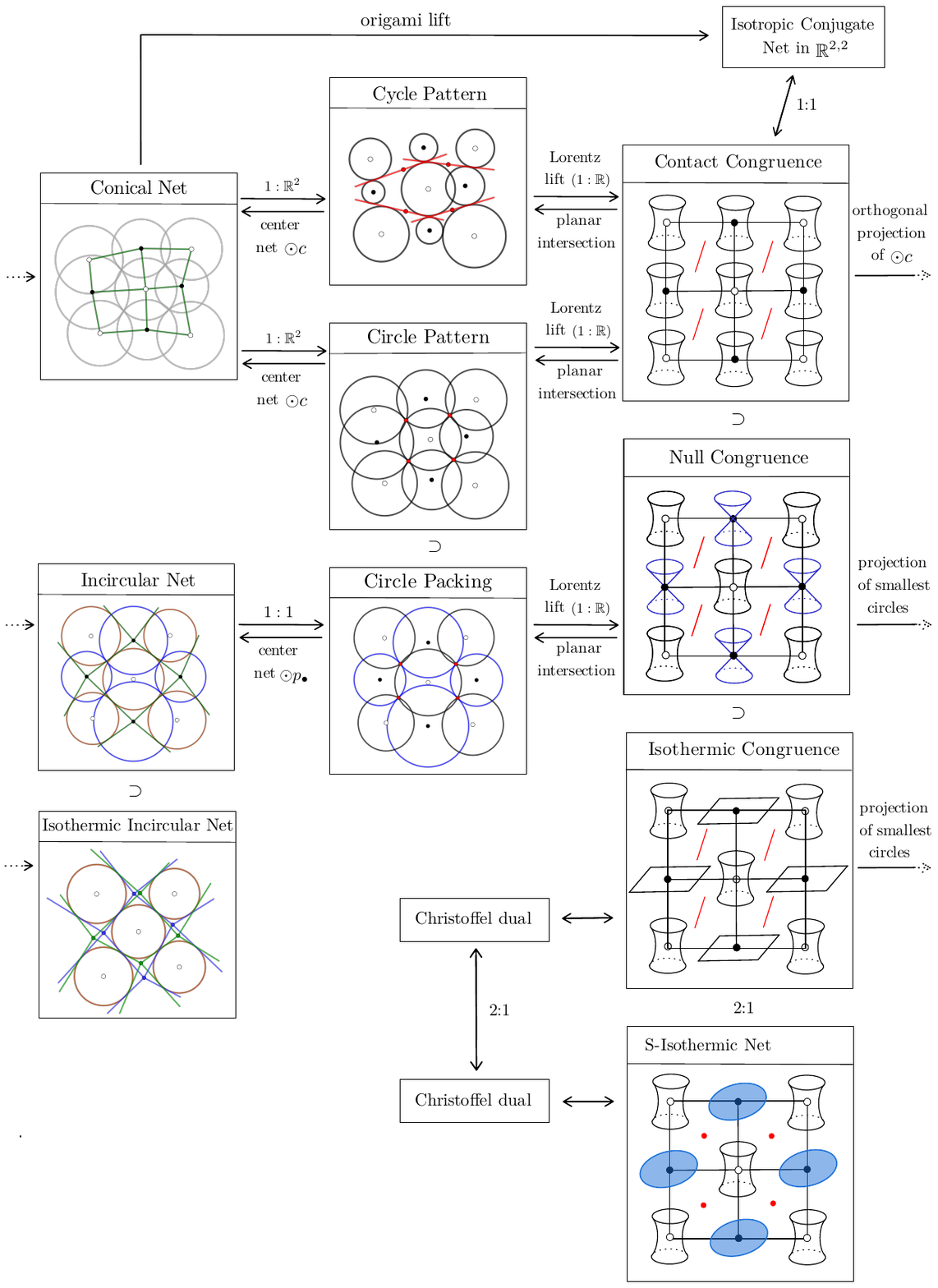}
		\put(34.5, 96.35){\scriptsize \cite{clrpembeddings}}
		\put(58.4, 19.1){\footnotesize \cite{bhsminimal}}
	\end{overpic}}
	\caption{An overview of the relations between our main geometric objects (see Section 2).}
	\label{fig:overview}
\end{figure}

In this paper we consider discrete surfaces in \emph{Lorentz space $\lor = \R^{2,1}$}. Due to the Minkowski metric two different points $P,P' \in \lor$ may have positive, zero or negative squared distance. A quick introduction to Lorentz geometry is given in Section~\ref{sec:lorentzgeometry}. We denote the space of lines by $\li(\lor)$, of planes by $\pl(\lor)$, of circles by $\ci(\lor)$ and of spheres by $\sp(\lor)$. A subscript $+$, $0$ or $-$ refers to \emph{spacelike}, \emph{isotropic}, or \emph{timelike} objects, for example $\sp_+$ is the space of spacelike spheres. An arrow denotes a space of oriented objects, so $\osp$ is the space of \emph{oriented spheres}. Note that we call the spheres in $\sp_0(\lor)$ \emph{null-spheres}. Null-spheres are not oriented, but a null-sphere $S$ is considered to be in oriented contact with an oriented sphere $S'$ if the center of $S$ is in $S'$.

\begin{figure}
	\centering
  \includegraphics[width=.49\linewidth]{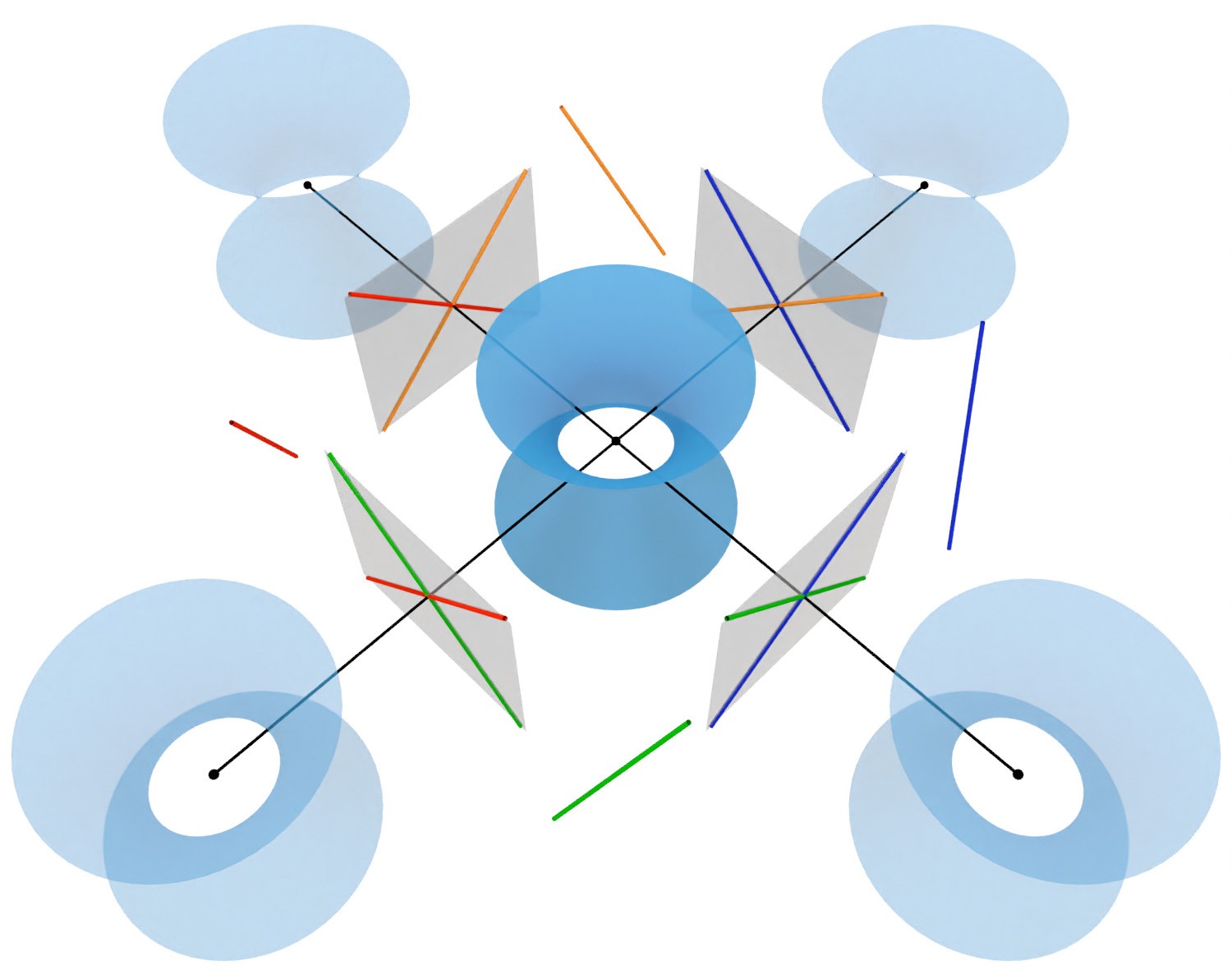}
	\includegraphics[width=.49\linewidth]{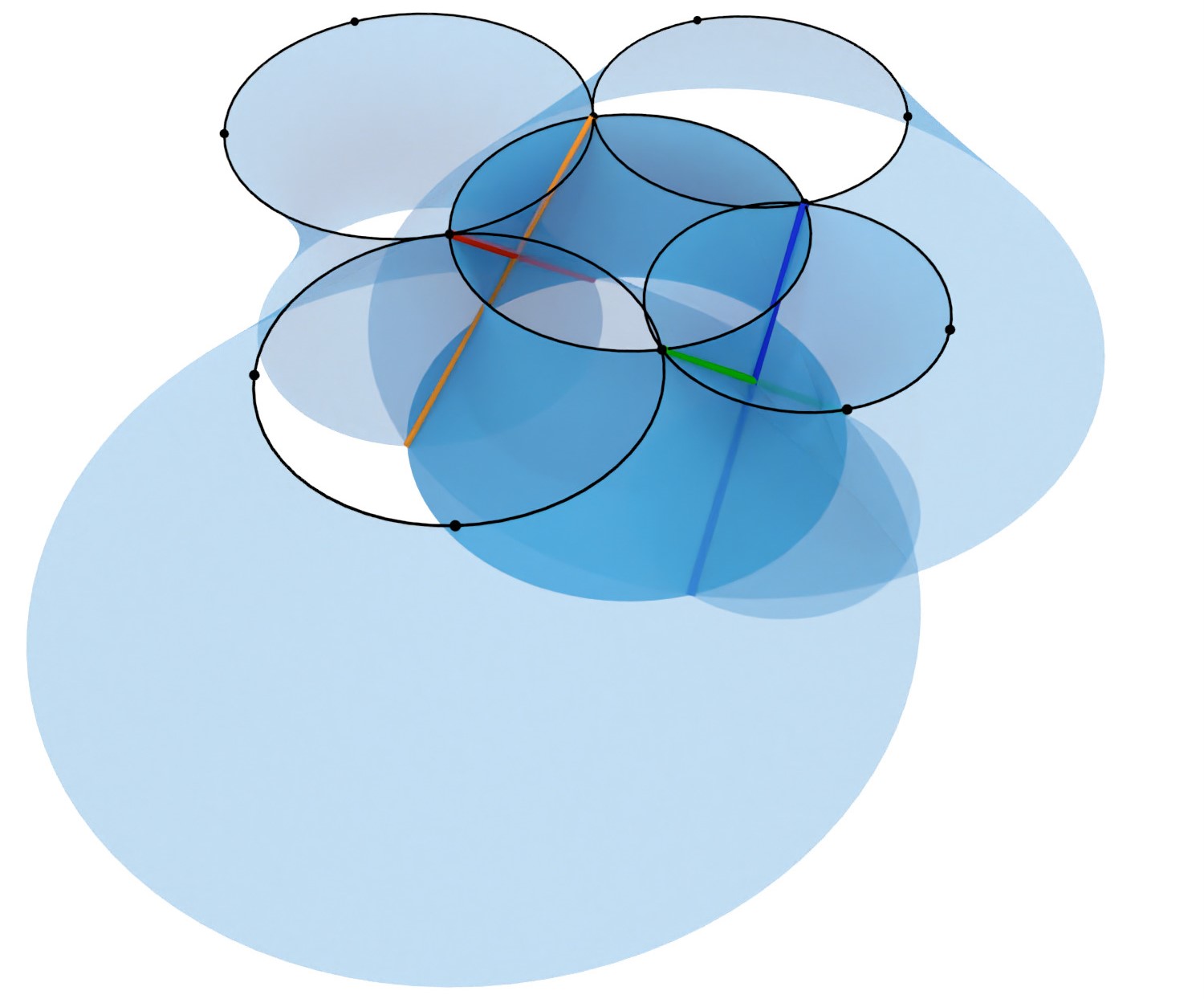}
	\caption{Left: the local combinatorics of a contact congruence. At each vertex there is a timelike sphere and at each face an isotropic line. Adjacent spheres share a tangent plane that contains two isotropic lines. Right: the corresponding geometric picture in Lorentz space $\lor$.}
	\label{fig:cc}
\end{figure}

\begin{definition} \label{def:contactcongruence}
	A \emph{contact congruence} is a pair of maps 
	\begin{align}
		\cc&: \Z^2 \rightarrow \osp_-(\lor) \cup \sp_0(\lor), &
		\ccf&: F(\Z^2) \rightarrow \oli_0(\lor),
	\end{align}
  such that the oriented sphere $\cc(v)$ is in oriented contact with the oriented isotropic line $\ccf(f)$ whenever $v$ and $f$ are incident, see also Figure~\ref{fig:cc} and Figure~\ref{fig:cccp}.
	The \emph{center net} (of a contact congruence $\cc$) is a map $\ccm: \Z^2 \rightarrow \lor$, such that $\ccm(v)$ is the center of the sphere $\cc(v)$. 
\end{definition}

Note that this definition implies that $\cc(v)$ contains $\ccf(f)$ whenever $v$ and $f$ are incident, see Section~\ref{sec:laguerre} for details on oriented contact. Moreover, the definition implies that around each face the four spheres are in pairwise oriented contact (so there are six pairwise contacts).

Let us define another basic type of map that will be useful throughout.

\begin{definition}\label{def:conjugatenet}
	A map $x: \Z^2 \rightarrow \R^n$ is a \emph{conjugate net} \cite{sauerqnet, dsqnet, bsddgbook} if the image of the vertices of each face $f\in F(\Z^2)$ are contained in a plane. We denote this plane by $x(f)$. 
\end{definition}

A conjugate net $x: \Z^2 \rightarrow \lor$ is called \emph{spacelike}, \emph{isotropic} or \emph{timelike} if all its planes are of that type.
In particular, a map $\ccm: \Z^2 \rightarrow \lor$ is a center net of a contact congruence $\cc$ if and only if $\ccm$ is an isotropic conjugate net (see Lemma~\ref{lem:centerNetIsIsoConjugate}).
Thus, we may think of the center net $\ccm$ of a contact congruence $\cc$ as a \emph{discrete surface} in $\lor$, and of $\cc$ as being defined by $\ccm$ and the choice of one oriented isotropic line $\ccf(f_0)$ (in the corresponding plane of the isotropic conjugate net $\ccm$) for some $f_0\in F$.

Let $\eucl \simeq \R^2$ be the Euclidean plane, which we embed into $\lor$ as
\begin{align}
	\eucl = \{x \in \lor \ | \ x_3 = 0  \}.
\end{align}

\begin{definition}\label{def:circlepattern}
	A \emph{circle pattern} is a pair of maps 
	\begin{align}
		\cp&: \Z^2 \rightarrow \ci(\eucl), & \cpf&: F(\Z^2) \rightarrow \eucl,
	\end{align}
	such that the circle $\cp(v)$ contains the point $\cpf(f)$ whenever $v$ and $f$ are incident. The \emph{conical net} (of a circle pattern~$\cp$) is the map $\cpm: \Z^2 \rightarrow \eucl$, where $\cpm(v)$ is the center of $\cp(v)$.
\end{definition}

The predecessor of circle patterns were \emph{circle packings}. Koebe \cite{koebecp} and Thurston \cite{thurstoncp} introduced circle packings as discretizations of conformal maps. A good introduction to circle packings is the monograph by Stephenson \cite{stephensoncp}. The generalization to circle patterns with arbitrary intersection angles was started by Thurston \cite[Section 13.7]{thurstonlecturenotes} and Rivin \cite{rivin}. Conical nets in the plane were studied in the context of discrete differential geometry \cite{muellerconical}, and have recently been related to the \emph{dimer model} from statistical physics \cite{amiquel, klrr, clrdimer}. In that context, conical nets are known as \emph{t-embeddings} \cite{clrdimer} or \emph{Coulomb gauge} \cite{klrr}, with some additional embedding constraints that we do not require. 

Note that for a conical net $x$ there is a two-parameter family of circle patterns $\cp$ with $\cpm = x$. Specifically, $\cp$ is determined by the choice of one point $\cpf(f_0)$ for some $f_0 \in F(\Z^2)$.

\begin{theorem}\label{th:cctocp}
  Every contact congruence $\cc$ defines a circle pattern $\cp$ via the intersection
	\begin{align}
		\cp(v) &= \cc(v) \cap \eucl, & \quad \cpf(f) &= \ccf(f) \cap \eucl,
	\end{align}
	see also Figure~\ref{fig:cccp}.
	The conical net $\cpm$ is the orthogonal projection of the center net $\ccm$.
\end{theorem}

Conversely, for every circle pattern $\cp$ there is a one-parameter family of contact congruences $\cc$ that intersect $\eucl$ in $\cp$, the parameter is for example the (oriented) radius of a sphere $\cc(v_0)$. In this context, we call the contact congruence $\cc$ the \emph{Lorentz lift} $\lorlift\cc$ of the circle pattern $\cp$. Analogously, for every conical net $x$ there is a three-parameter family of contact congruences $\cc$ such that $\cpm$ is the orthogonal projection of $\ccm$, for example by choosing an initial isotropic line $\ccf(f_0)$ for some $f_0 \in F(\Z^2)$.

\begin{figure}
	\centering
	\includegraphics[width=.75\linewidth]{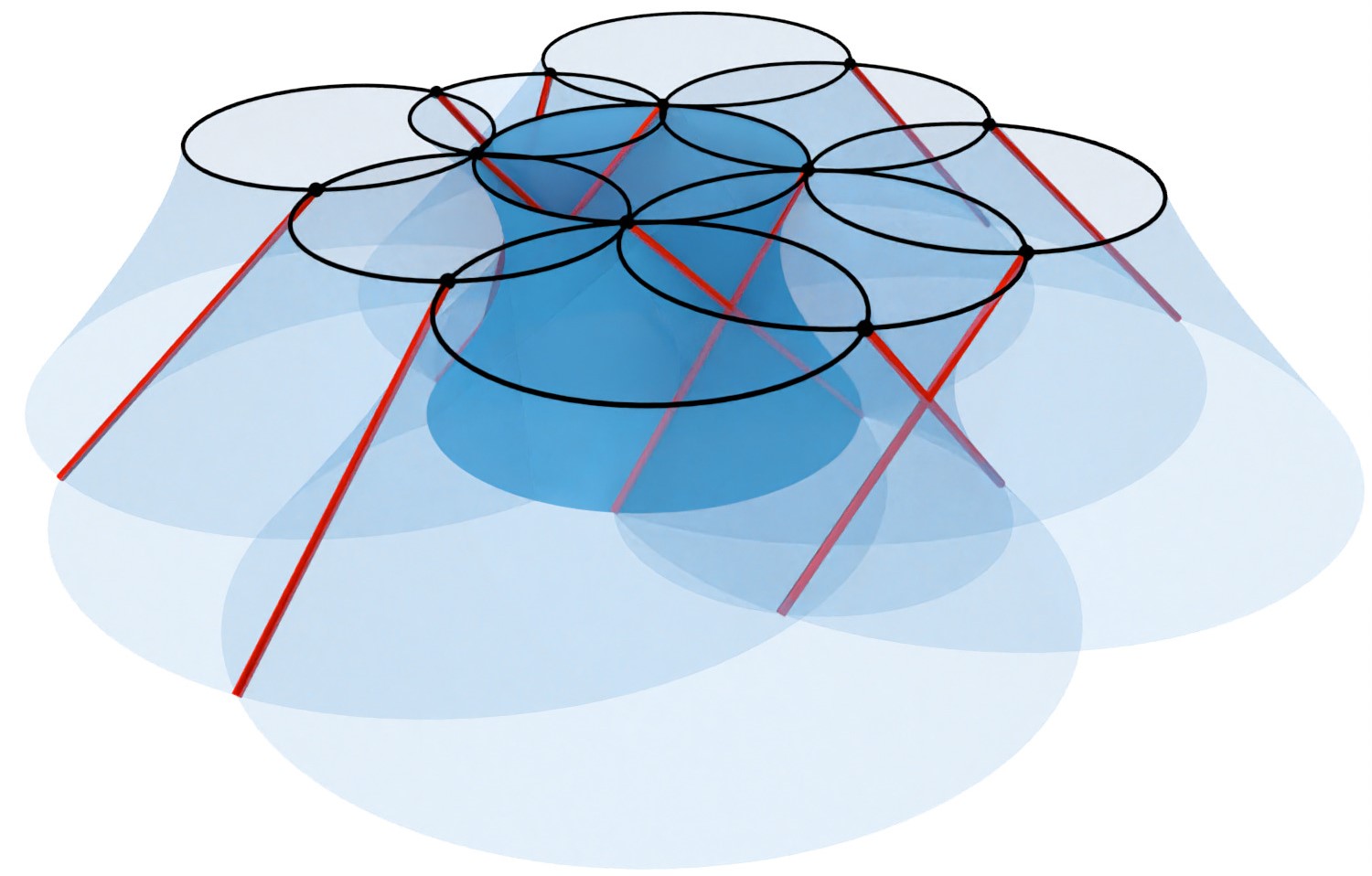}
	\caption{The intersection of a contact congruence with a plane is a circle pattern.}
	\label{fig:cccp}
\end{figure}

\begin{remark}
	We restrict ourselves to contact congruences and circle patterns defined on $\Z^2$ throughout the paper, to keep the exposition focused on the geometry. However, one verifies without difficulty that all the constructions and theorems that involve general contact congruences or circle patterns hold if the combinatorics of $\Z^2$ is replaced by the combinatorics of an arbitrary locally finite, planar and bipartite graph.
\end{remark}

In Laguerre geometry \cite{laguerrelaguerre, blptlaguerre}, there is the so-called cyclographic model $\cyc \simeq  \R^{2,2}$: every point $P \in \cyc$ corresponds to a unique oriented timelike sphere or null-sphere $\xi_{\osp}(P)$ in $\lor$ and vice versa, see Section~\ref{sec:laguerre}. We may therefore consider the cyclographic lift $\cyclift{\cc}$ of a contact congruence $\cc$. Via the cyclographic lift we show in Section~\ref{sec:cyclift} that contact congruences are in a bijective relation with fully isotropic conjugate nets in $\cyc$. Since we may lift a conical net $\cpm$ to a contact congruence $\cc$, we also think of $\cyclift{\cc}$ as a cyclographic lift of $\cpm$.
It turns out that the cyclographic lift of a conical net coincides with the $\R^{2,2}$-lift of a conical net as defined using the origami map by Chelkak, Laslier and Russkikh \cite{clrdimer}. 

In this light, what we have provided so far is an interpretation of the \cite{clrdimer} origami-lift as a discrete surface in $\lor$, by viewing $\R^{2,2}$ as the cyclographic model of Laguerre geometry. Moreover, we interpret the isometries of $\R^{2,2}$ as the Laguerre transformations of $\lor$.

Conical nets have recently received attention in the field of statistical mechanics, in particular in the study of the dimer model (see \cite{kenyondimerintro} for an introduction to the dimer model). It was shown in \cite{amiquel, klrr} that one may associate a dimer model to a conical net, and this approach has been further developed in \cite{clrdimer,clrpembeddings}. The quantities that define the dimer model on a conical net $\cpm$ are the \emph{$X$-variables}
\begin{align}
	X: \Z^2 \rightarrow \R, \quad  v \mapsto \frac{|\cpm(v_1)-\cpm(v)||\cpm(v_3)-\cpm(v)|}{|\cpm(v_2)-\cpm(v)||\cpm(v_4)-\cpm(v)|},
\end{align}
where $v_1, v_2,v_3, v_4$ are the vertices adjacent to $v$.
Borrowing an (apparently unpublished) observation of Chelkak \cite{chelkakliftxvars}, one may show that the $X$-variables are also (the square of) the ratio of two distances of the cyclographic lift $\cyclift{\cpm}$ of a conical net $\cpm$, specifically
\begin{align}
	X(v) = \frac{|\cyclift{p}(v_1)-\cyclift{p}(v_3)|^2_{2,2}}{|\cyclift{p}(v_2)-\cyclift{p}(v_4)|^2_{2,2}}.
\end{align}
In Laguerre geometry, this means that we may read of the $X$-variables from the Lorentz lift $\lorlift{\cp}$ of a conical net as the square of the tangential distance of $\lorlift\cp(v_1)$ to $\lorlift\cp(v_3)$ divided by the square of the tangential distance of $\lorlift\cp(v_2)$ to $\lorlift\cp(v_4)$.

We now turn towards special cases of circle patterns and contact congruences, which will lead us to discrete isothermic surfaces.

For the following, let us consider the bipartition of $\Z^2$ into black vertices $\Zb^2$ and white vertices $\Zw^2$, that is
\begin{align}
	\Zb^2 &= \{z \in \Z^2\ | \ z_1 + z_2 \in 2 \Z  \}, & \Zw^2 &= \{z \in \Z^2\ | \ z_1 + z_2 \in 2 \Z + 1 \}.
\end{align}
Note that $\Zb^2 \simeq \Zw^2 \simeq \Z^2$, therefore we may also speak of black (white) edges, that is edges of $\Zb^2$ ($\Zw^2$). Note that $E(\Zb^2)$ is in direct bijection with $F(\Z^2)$ and with $E(\Zw^2)$. We also see that $F(\Zb^2)$ is in direct bijection with $\Zw^2$, and $F(\Zw^2)$ is in direct bijection with $\Zb^2$.

Let us denote by $\cpb$ and $\cpw$ the two restrictions of a circle pattern $\cp$ to $\Zb^2$ and $\Zw^2$ respectively.

\begin{figure}
	\centering
	\includegraphics[width=.49\textwidth]{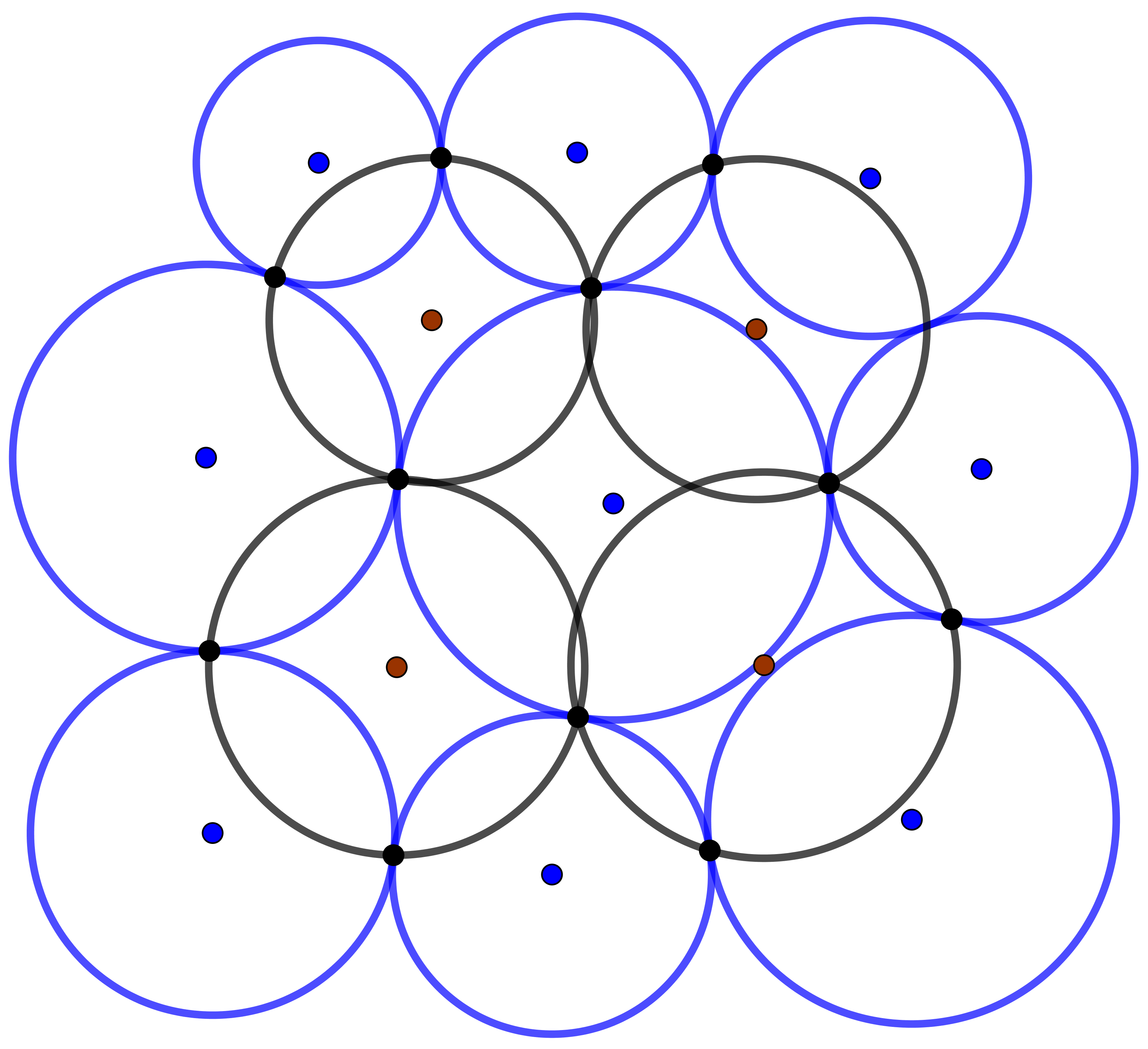}
	\hspace{0mm}
	\includegraphics[width=.49\textwidth]{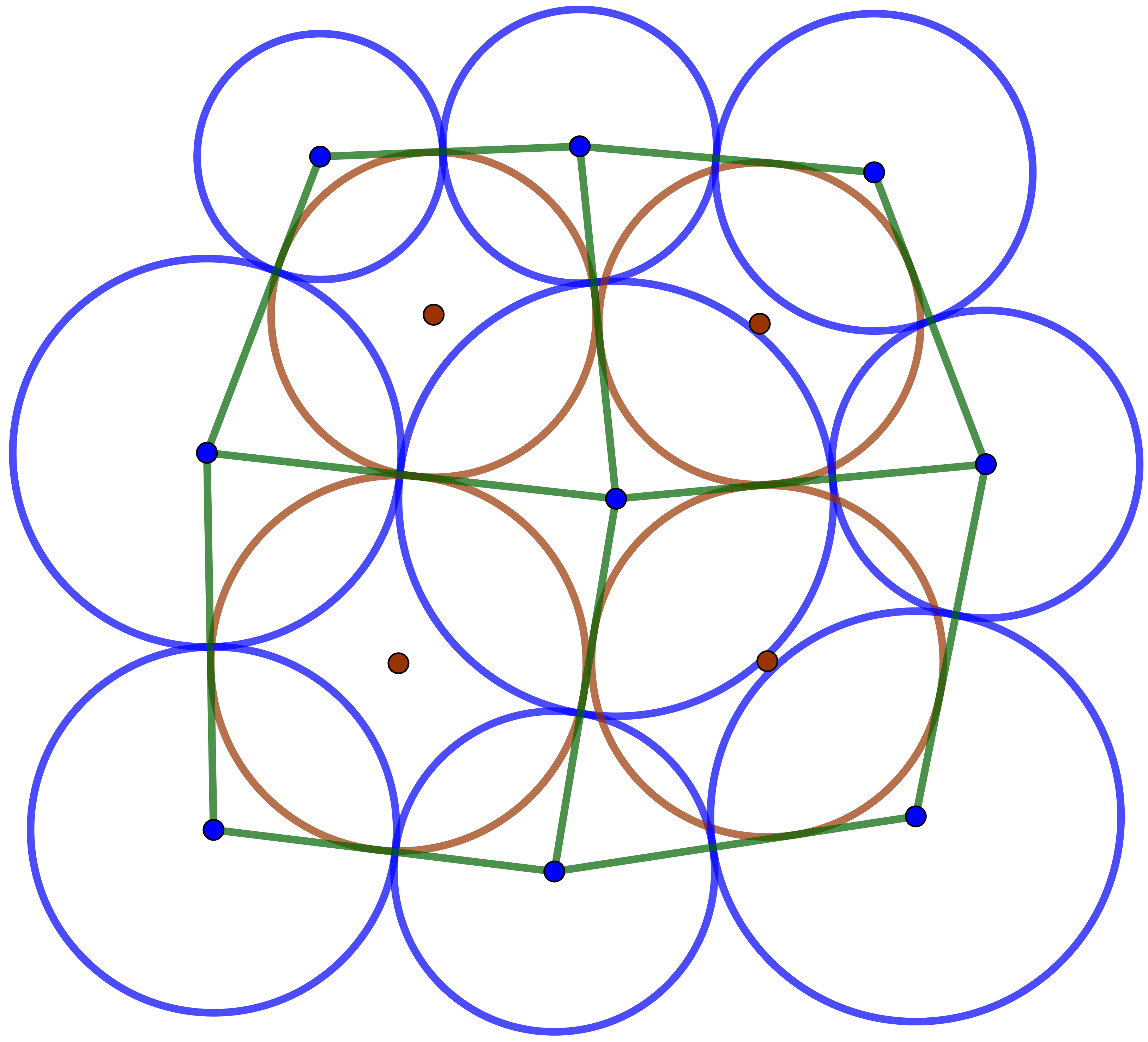}
	\caption{A circle packing $\cp$ (with $\cpb$ in blue and $\cpw$ in black) and the corresponding incircular net $\cpmb$ (green), with incircles (brown).}
	\label{fig:circlepacking}
\end{figure}

\begin{definition}\label{def:circlepacking}
  A \emph{circle packing} $\cp$ is a circle pattern such that for any two black vertices $b,b' \in \Zb^2$ incident with a common face of $\Z^2$ the two circles $\cp(b)$ and $\cp(b')$ are touching (see also Figure~\ref{fig:circlepacking}).
\end{definition}

It is the restriction to $\cpb$ (of what we call a circle packing $p$) that is usually called a circle packing (see \cite{stephensoncp}). However, generically every circle packing $\cpb$ (with $\Z^2$ combinatorics) defines a unique circle pattern $\cp$ that restricts to $\cpb$, due to the \emph{touching coins lemma} \cite{bhsminimal}. We choose to call the whole of $\cp$ a circle packing, so as not to introduce new terms unnecessarily.

Every quad in $\cpmb$ has an incircle and the incircle centers are given by $\cpmw$. Therefore we also call $\cpmb$ an \emph{incircular net}. Note that the incircles of $\cpmb$ are not the circles of $\cpw$ though.

\begin{remark}
  A special case of incircular nets with straight parameter lines was related to confocal conics in \cite{gsfastrhombisch,abconfocal,bstincircular}. In the context of Euclidean Laguerre geometry this resulted in new incidence theorems. We do not study this special case in the following, but it is not hard to see that corresponding contact congruences are \emph{flat}, in the sense that the centers of the spheres are all contained in a plane. Moreover, \cite{gsfastrhombisch} introduced a generalization of incircular nets to $\R^3$ and showed that this discretizes surface parametrizations with a first fundamental form such that the diagonal entries are equal. 
\end{remark}

In statistical mechanics, the notion of an s-embedding \cite{chelkaksembeddings, chelkaksgraphs} was introduced in relation to the Ising model (for introductions to the Ising model see \cite{dcsconformalinvariance, baxterbook}). An s-embedding is an incircular net with additional embeddedness constraints, which we do not require.

Let us denote by $\ccb$ and $\ccw$ the two restrictions of a contact congruence $\cc$ to $\Zb^2$ and $\Zw^2$ respectively.

\begin{definition} \label{def:nullcongruence}
	A \emph{null congruence} is a contact congruence $\cc$ such that 
	\begin{enumerate}
    \item the image of $\ccb$ is in $\sp_0$,
		\item the image of $\ccw$ is in $\osp_-$.
	\end{enumerate}
  In other words, the radii of all black spheres are zero (see also Figure~\ref{fig:nullcongruence}).
\end{definition}

\begin{figure}
	\centering
	\includegraphics[height=.45\textwidth]{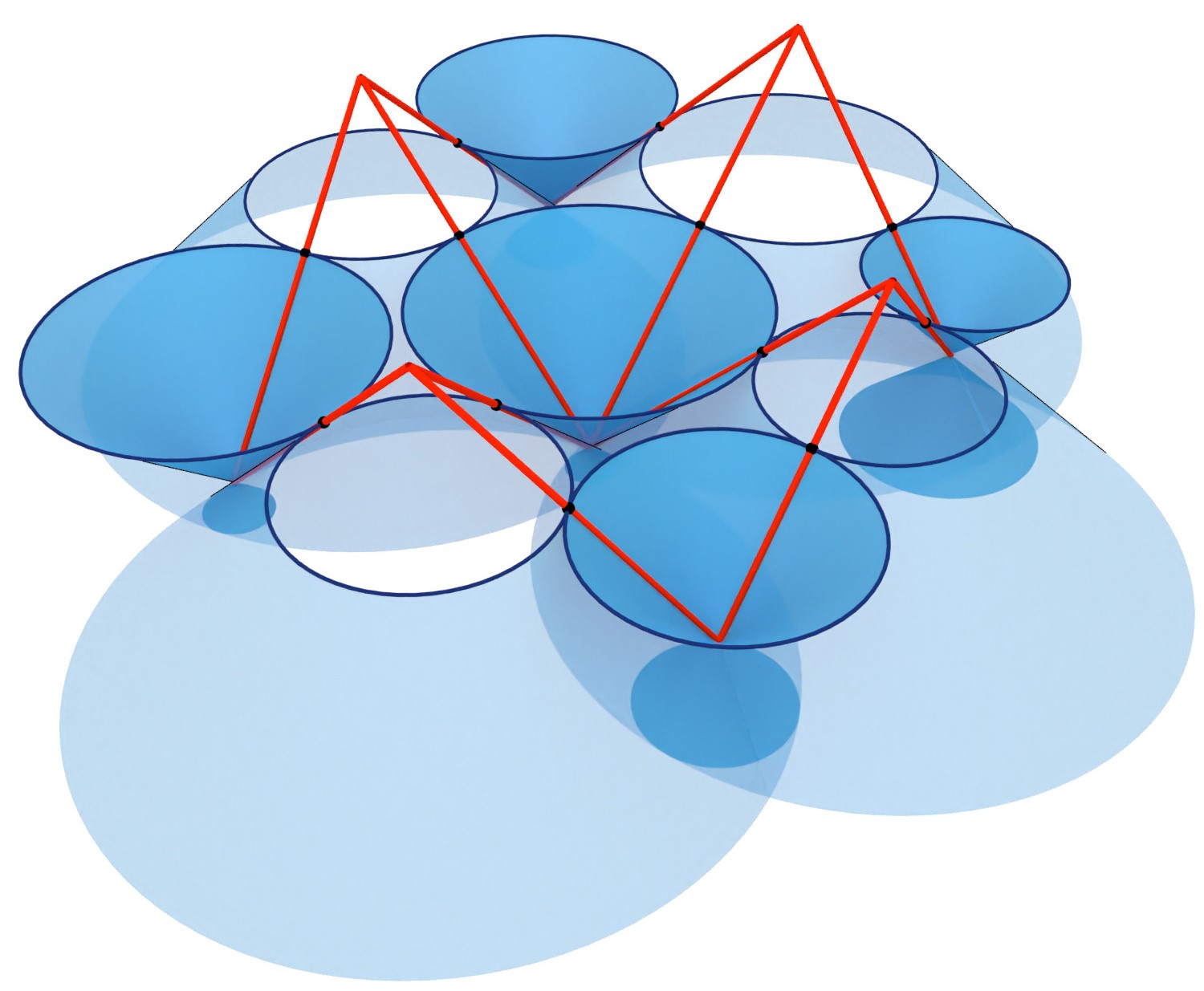}
	\includegraphics[height=.45\textwidth]{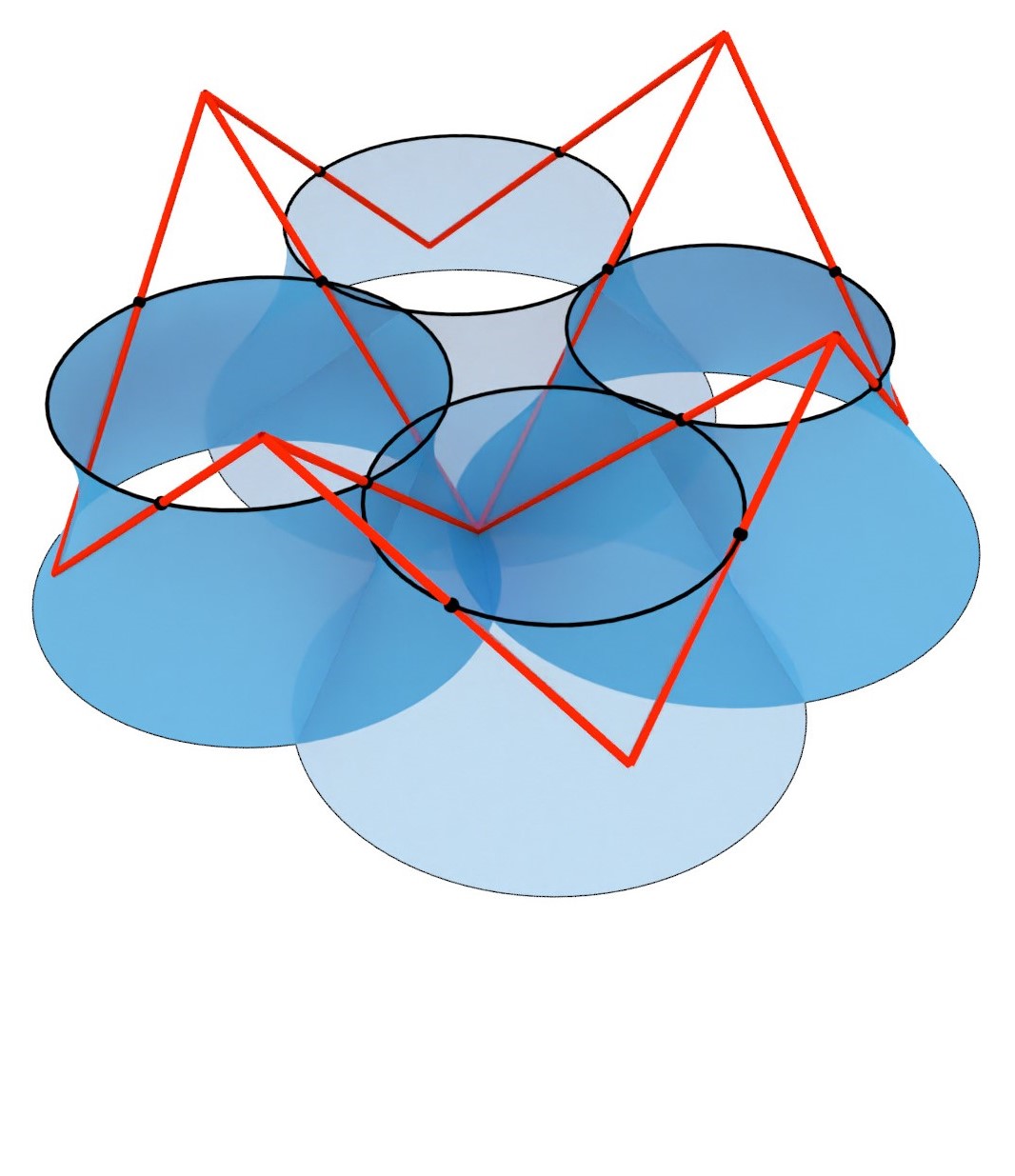}
	\caption{The null-spheres $\ccb$ (left) and time-like spheres $\ccw$ (right) of a null congruence $\cc$ as well as the isotropic lines $\ccf$ (red).}
	\label{fig:nullcongruence}
\end{figure}

\begin{theorem}\label{th:nullcctopacking}
	The circle pattern associated to a null congruence via Theorem~\ref{th:cctocp} is a circle packing.
\end{theorem}
Therefore every null congruence defines an incircular net. Conversely, an incircular net defines a null congruence $\cc$ up to the choice of the orientation of a sphere $\cc(v_0)$ and the third coordinate of $\ccm(v_0)$ for some $v_0 \in \Z^2$. The incircles of an incircular net are actually the projections of the smallest circles of the null congruence, see Section~\ref{sec:laguerrecp} for details.

Let us turn to a special case of null congruences.

\begin{definition}\label{def:isothermiccongruence}
	An \emph{isothermic congruence} is a null congruence $\cc$ such that $\ccmw$ is a spacelike conjugate net.
\end{definition}

To each isothermic congruence $\cc$ one may associate an incircular net $\cpb$ via Theorem~\ref{th:nullcctopacking}. Consider a black vertex $b\in \Zb^2$ and the four adjacent white vertices $w_1,w_2,w_3,w_4 \in \Zw^2$. Two consecutive incircles $\lcp(w_i), \lcp(w_{i+1})$ share a common tangent that contains $\cpb(b)$. In Section~\ref{sec:isothermiccongruences} we show that the four other tangents also intersect in a point $\miq{\cpb(b)}$, and this is a characterizing property of incircular nets that correspond to isothermic congruences. We call these incircular nets \emph{isothermic incircular nets}. The fact that the other tangents intersect at each black vertex of an isothermic incircular net means that isothermic incircular nets actually come in pairs, see Figure~\ref{fig:isothermiccongruence}. In fact, the two so related isothermic incircular nets are related by \emph{Miquel dynamics}, see Section~\ref{sec:miquel}.

Bobenko, Hoffmann and Springborn \cite{bhsminimal} introduced the notion of S-isothermic nets in $\R^3$. We now translate their definition to the Lorentz setup, and then we proceed to show how S-isothermic nets relate to isothermic congruences.

\begin{definition}\label{def:sisothermic}
	A \emph{(spacelike Lorentz) S-isothermic net} $h$ is a collection of maps 
	\begin{align}
		\isow&: \Zw^2 \rightarrow \osp_-(\lor), \\
		\isob&: \Zb^2 \rightarrow \ci_+(\lor), \\
		\isof&: F(\Z^2) \rightarrow \lor,
	\end{align}
  such that for all faces $f = (w,b,w',b') \in F(\Z^2)$ the two spheres $\isow(w)$, $\isow(w')$ are in oriented contact and the two circles $\isob(b)$, $\isob(b')$ intersect the two spheres orthogonally in the point $\isof(f)$.
	The \emph{center net} of an S-isothermic net $\iso$ is the map $\isom: \Z^2 \rightarrow \lor$, such that $\isom(v)$ is the center of $\iso(v)$ for all $v\in \Z^2$.
\end{definition}

\begin{figure}
	\centering
	\includegraphics[width=.45\textwidth]{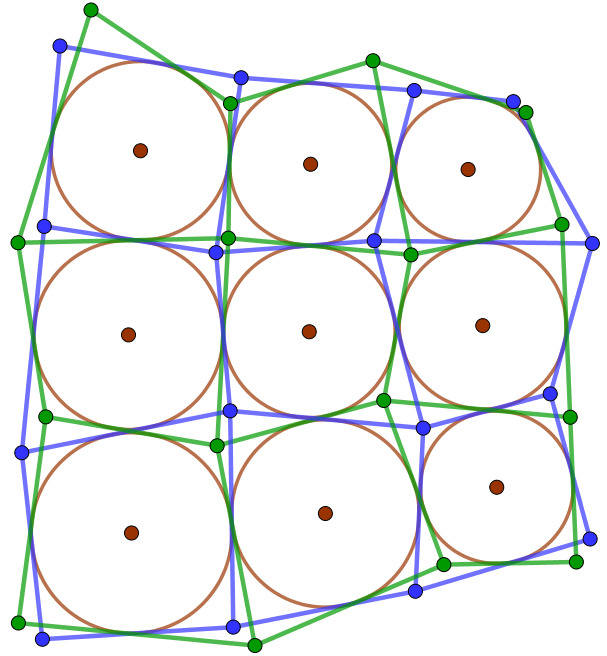}
	\includegraphics[width=.54\textwidth,trim=0 -2cm 0  0,clip]{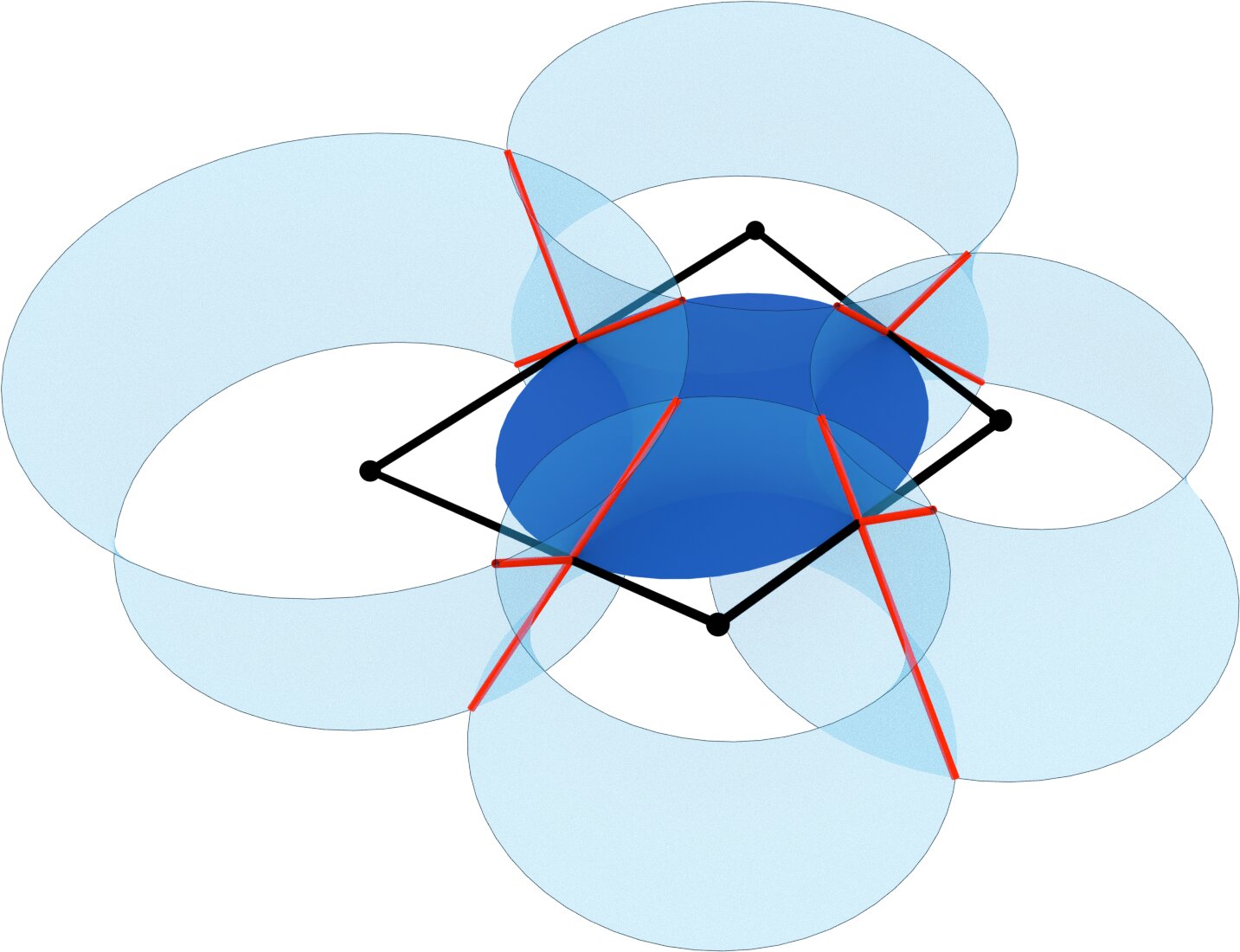}
	\caption{Left: an isothermic incircular net $\cpb$ (blue) is always accompanied by a second incircular net $\miq\cpb$ (green) with the same incircles (brown). Right: an isothermic congruence, the four sphere centers are in a plane, as a result there is a circle orthogonal to the four spheres.}
	\label{fig:isothermiccongruence}
\end{figure}

The following theorem provides the link between incircular nets as used in statistical mechanics and isothermic nets as used in discrete differential geometry.

\begin{theorem}\label{th:isocongruencetoisonet}
	Every isothermic congruence $\cc$ defines a unique S-isothermic net $\iso$ such that $\ccw = \isow$, and vice versa every S-isothermic net defines two isothermic congruences such that $\isow = \ccw$.
\end{theorem}

Essentially, the circles $\isob$ are obtained from the spheres $\ccb$ by intersecting each null-sphere $\ccb(b)$ with the spacelike plane spanned by the adjacent centers of $\ccw$, see Section~\ref{sec:isothermicnets} for details.

As it turns out, Theorem~\ref{th:isocongruencetoisonet} allows us to apply the full machinery developed in \cite{bhsminimal}, to isothermic congruences and therefore also to isothermic incircular nets.

\begin{remark}
  Note that the ``s'' in s-embedding and in S-isothermic were not originally related, although there are some connections now.
\end{remark}

As shown in \cite{bhsminimal}, S-isothermic nets come in pairs via a construction called the \emph{Christoffel dual}. Let us denote by 
\begin{align}
  \isoc = \isom \cup \isof: \Z^2 \simeq (\Z^2 \cup F(\Z^2)) \rightarrow \lor, \label{eq:combinediso}
\end{align}

the combination of center net $\isom$ and contact points $\isof$ of an S-isothermic net $\iso$, defined on $\Z^2 \cup F(\Z^2)$, which we identify with $\Z^2$ for the time being. In this notation, we define the discrete differential
\begin{align}
	\d \isoc(v,v') = \isoc(v') - \isoc(v),
\end{align}
for all adjacent $v,v' \in \Z^2$. We also define the dual differential
\begin{align}
	\d \isoc^*(v,v') = \pm \frac{\d \isoc(v',v)}{|\d \isoc(v',v)|^2},
\end{align}

for all adjacent $v,v'\in \Z^2$, where the sign is $+$ if the edge $(v,v')$ is horizontal or $-$ if the edge is vertical in $\Z^2$. This dual form is \emph{closed}, which means it may be integrated to obtain a map $\isoc^*: \Z^2 \rightarrow \lor$. In fact, it is not hard to see that $\isoc^*$ also defines a unqiue S-isothermic net $\iso^*$, which is called the \emph{Christoffel dual} of $\iso$. Due to the identification in Theorem~\ref{th:isocongruencetoisonet}, this means we also obtain a definition of the Christoffel dual $\cc^*$ of an isothermic congruence $\cc$.

\begin{remark}
	One may generalize the combinatorics of null congruences and incircular nets as follows. Consider a quad-graph $Q$. Now construct a graph $G$ such that
	\begin{align}
	B(G) &= V(Q), & W(G) &= F(Q), & E(G) = \{ (v,f) \in V(Q) \times F(Q)  \ | \ v \mbox{ adjacent to } f \}.
	\end{align}
  We may replace $\Z^2$ by $G$, in this case the incircular net $\cpb$ has the combinatorics of the quad-graph $Q$. In some references (e.g., \cite{klrr,chelkaksgraphs}), the combinatorics are chosen slightly differently, by also adding to the edges the set
	\begin{align}
	\{ (v,v') \in V(Q) \times V(Q) \ | \ v \mbox{ adjacent to } v' \},
	\end{align}
	which also doubles the number of faces. Our statements about null congruences hold for these combinatorics in the same way, except that the orientation of every second sphere in $\ccw$ is reversed. However, our statements about isothermic congruences only apply if the dual graph $Q^*$ is bipartite, since otherwise the sign in the Christoffel dual cannot be consistently alternating.
\end{remark}


\section{Lorentz sphere geometry}
\label{sec:lorentzgeometry}

\subsection{Lorentz geometry}
\label{sec:basiclorentzgeometry}

\begin{figure}
	\centering
	\begin{minipage}{.27\textwidth}
		\includegraphics[width=\textwidth]{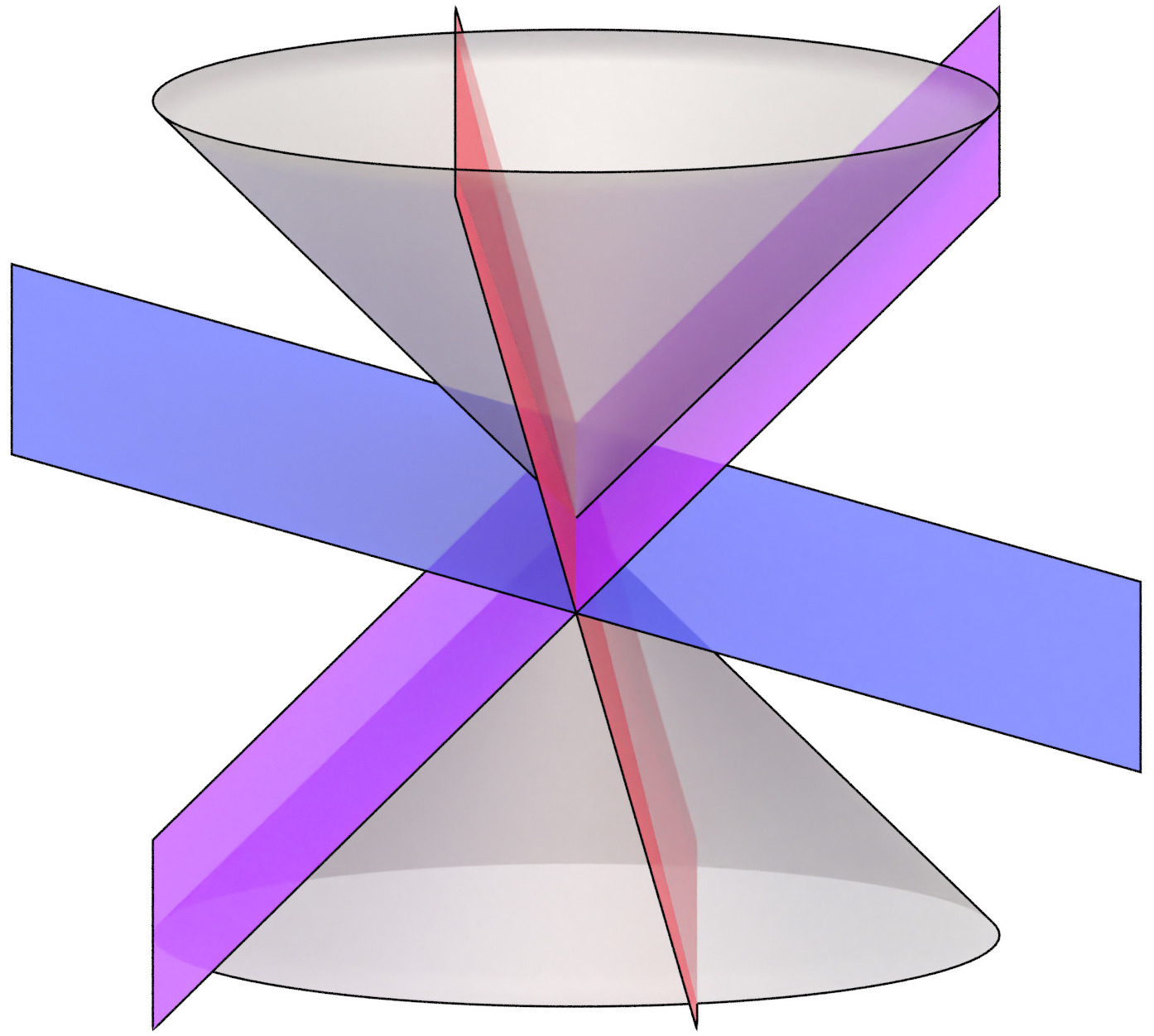}
	\end{minipage}
	\hfill
	\begin{minipage}{.27\textwidth}
		\includegraphics[width=\textwidth]{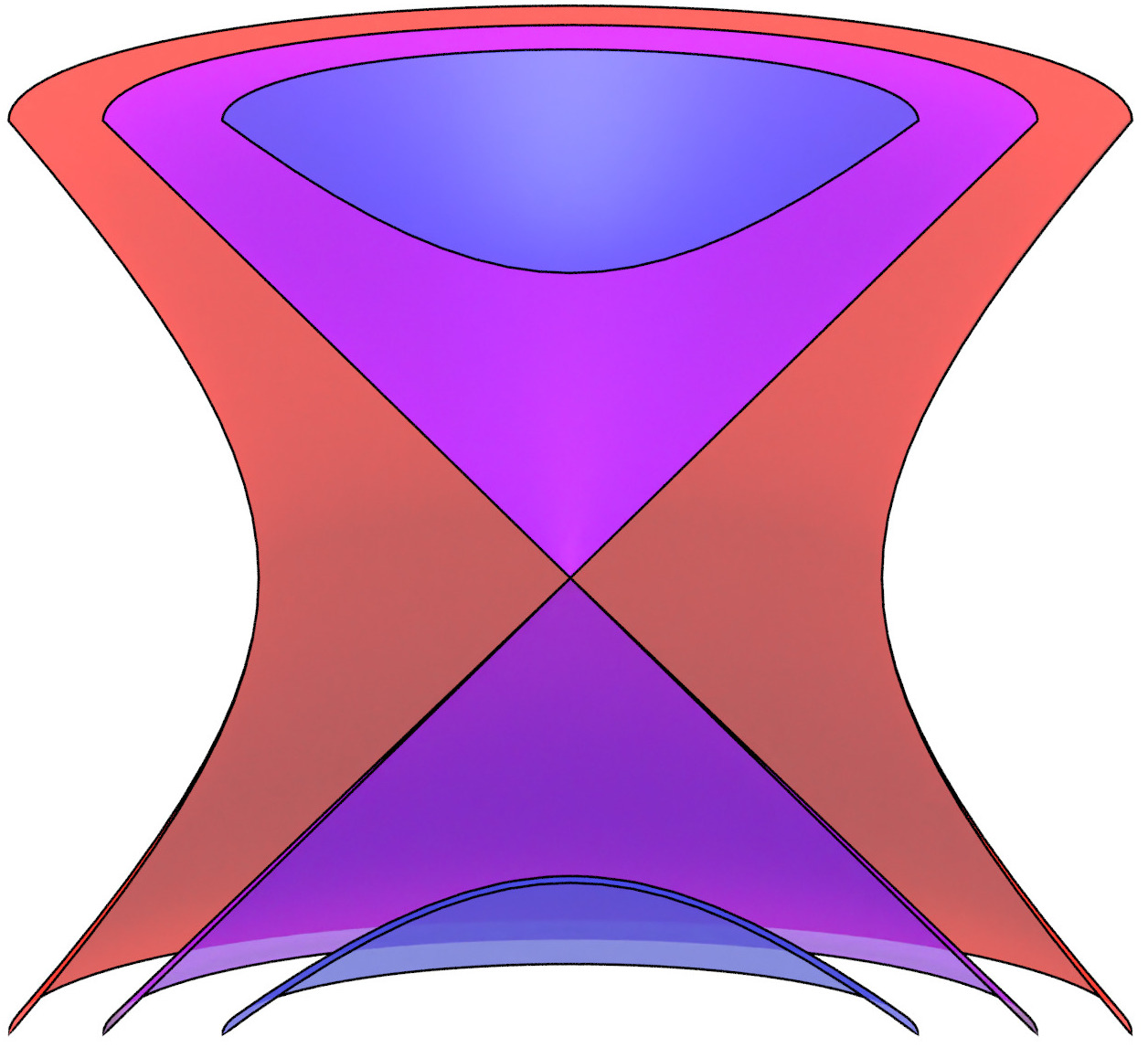}
	\end{minipage}
	\hfill
	\begin{minipage}{.35\textwidth}
		\includegraphics[width=\textwidth]{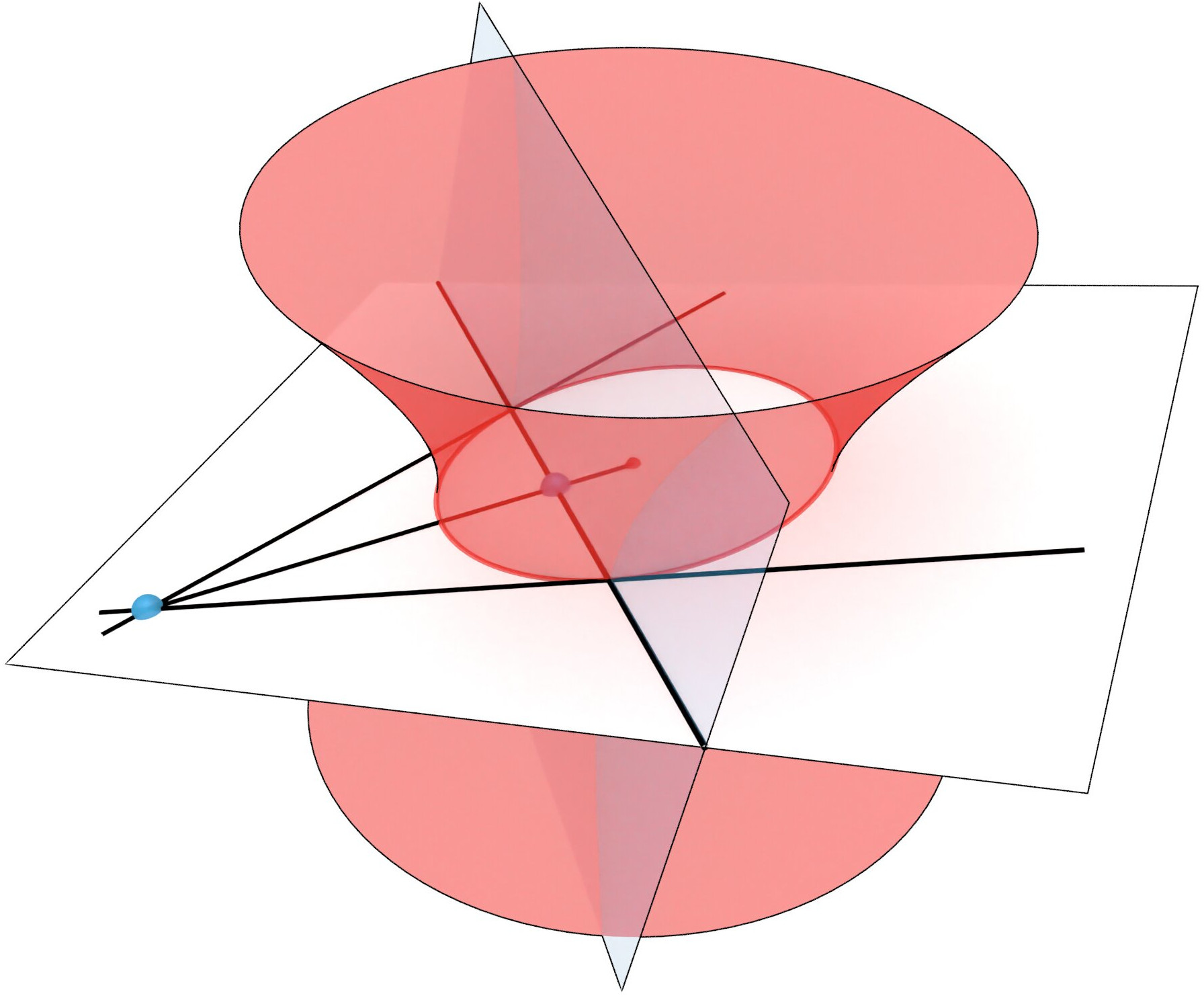}
	\end{minipage}
	\caption{Left: a timelike plane (red), an isotropic plane (violet) and a spacelike plane (blue).  Center: a timelike sphere (red), a null-sphere (violet) and a spacelike sphere (blue). Right: inverting a point about a sphere.}
	\label{fig:lorentzbasics}
\end{figure}

\emph{Lorentz space} $\lor = \R^{2,1}$ is equipped with the non-degenerate bilinear form
\begin{align}
	\label{eq:lor_bilinear_form}
	\langle x, y \rangle  = x_1y_1 + x_2y_2 - x_3y_3,
\end{align} 
which has signature $\texttt{(++-)}$.

For a point $x \in \lor$, the evaluation $\sca{x,x}$ may be positive, zero, or negative. Hence, two points $x, y \in \lor$ may have positive, zero or negative squared distance $\langle x - y, x - y\rangle$. We denote the \emph{space of lines} by $\li(\lor)$ and \emph{spacelike}, \emph{lightlike (isotropic)} and \emph{timelike lines} as $\li_+(\lor), \li_0(\lor)$ and $ \li_-(\lor)$, respectively. The restriction of the bilinear form~\eqref{eq:lor_bilinear_form} to the three types of lines has signature $\texttt{(+)}$, $\texttt{(0)}$ and $\texttt{(-)}$, respectively. In a slight abuse of terminology we will say a spacelike line \emph{has} signature $\si{+}$, an isotropic line \emph{has} signature $\si{0}$ and a timelike line \emph{has} signature $\si{-}$. Similarly, \emph{spacelike}, \emph{isotropic}, and \emph{timelike planes}, $\pl_+(\lor)$, $\pl_0(\lor)$, $\pl_-(\lor)$, have induced signatures $\texttt{(++)}$, $\texttt{(+0)}$ and $\texttt{(+-)}$, see Figure~\ref{fig:lorentzbasics} (left). They can also be distinguished by the signature of their normal vector, in particular they have timelike, isotropic, and spacelike normal vectors respectively. Moreover to each isotropic line there exists a unique isotropic plane containing the line. Isotropic lines and isotropic planes intersect $\eucl$ at a (Euclidean) 45 degree angle.

A sphere $S$ in Lorentz space with center $c$ and squared radius $r^2 \in \R$ is given by the point set
\begin{equation*}
	S := \{x \in \lor \mid \lorsca{x-c, x-c} = r^2\}.
\end{equation*} 
We distinguish \emph{spacelike}, \emph{null-} and \emph{timelike spheres}, $\sp_+(\lor)$, $\sp_0(\lor)$, $\sp_-(\lor)$, corresponding to the cases $r^2 < 0$, $ r^2 = 0$ and $r^2 > 0$. Geometrically they correspond to two-sheeted hyperboloids, cones and one-sheeted hyperboloids, see Figure~\ref{fig:lorentzbasics} (center). Two timelike spheres are said to touch if they have a point $P$ in common and share the tangent plane in that point, see Figure~\ref{fig:touchingspheres}. Since the tangent plane is timelike, the two spheres share two isotropic lines that contain $P$ and are contained in the tangent plane. Furthermore, the line containing the two sphere centers contains $P$ and is orthogonal to the tangent plane. 

Non-empty planar sections of spheres are \emph{(Lorentz) circles} denoted by $\ci(\lor)$. Equivalently, circles are non-empty intersections of two spheres.
Spacelike circles are circles contained in spacelike planes. We denote spacelike Lorentz circles by $\ci_+(\lor)$. All tangent lines to a spacelike Lorentz circle are spacelike. There are also other type of circles, but we do not need them.

\subsection{Lorentz Möbius geometry}
\label{sec:conformallorentzgeometry}

Let $\RP^4$ be the real four-dimensional projective space and let 
\[
\mobsca{x, y}
= x_1y_1 + x_2y_2 + x_3y_3 - x_4y_4 - x_5y_5,
\]
denote a non-degenerate bilinear form on the space of homogeneous coordinates $\R^5$.
We call the quadric 
$$\mob = \{ [x] \in \RP^4 \mid \mobsca{x, x} = 0\},$$
of signature $\texttt{(+++--)}$ the \emph{Lorentz M\"obius quadric} of $\lor$. 
Via stereographic projection, points of $\lor$ can be identified with points on $\mob$. The inverse stereographic projection maps spheres to hyperplanar sections of $\mob$. Therefore, spheres in $\lor$ can be identified with points in $\RP^4$ via polar points of the corresponding hyperplanes. Spacelike spheres correspond to points with $\mobsca{x, x} < 0$, null-spheres to points with $\mobsca{x, x} = 0$ and timelike spheres to points with $\mobsca{x, x} > 0$. Note that each null-sphere is identified with the point on $\mob$ that is its center and apex.

\begin{remark}
  We used the term ``Lorentz Möbius quadric'' since the Lorentz Möbius quadric is the analogon of the Möbius quadric of classical (Euclidean) Möbius geometry. We just replace Euclidean space $\mathbf{E^3} \simeq \R^3$ with Lorentz space $\lor \simeq \R^{2,1}$. Another possible name would be ``conformal Lorentz geometry'', since the Lorentz Möbius transformations preserve the conformal structure of $\lor$, that is intersection angles. Moreover, the Lorentz Möbius quadric is also known as the \emph{Einstein universe} see for example \cite{emngeodesics}), so one could also use the name ``Einstein geometry'' for Lorentz Möbius geometry. Lorentz Möbius geometry is also sometimes called ``pseudo-conformal geometry'', see for example \cite{agpseudoconf}.
\end{remark}

Since we do not use Euclidean Möbius geometry in this article, we simply use the term ``Möbius geometry'' for Lorentz Möbius geometry throughout the paper.

Two spheres $S_1$ and $S_2$ in $\lor$ touch each other if and only if the line spanned by the two corresponding points in $\RP^4$ is tangential to $\mob$. The point of tangency stereographically projects to the point of contact of $S_1$ and $S_2$.
Two spheres $S_1$ and $S_2$ intersect orthogonally if the point in $\RP^4$ corresponding to $S_1$ is in the polar complement of the point corresponding to $S_2$ with respect to the Möbius quadric, and vice versa.

Let us consider a spacelike Lorentz circle $C$. There exists a one-parameter family of spheres containing $C$. The sphere centers lie on the axis of the circle. Lifting the one-parameter family of spheres to the corresponding points in $\RP^4$ yields a line $\ell$ of signature \texttt{(+-)}. This family contains timelike and spacelike spheres and precisely two null-spheres $S_1$, $S_2$, which correspond to the two intersection points of $\ell$ with $\mob$. The polar complement $\ell^\perp$ is a plane of signature \texttt{(++-)}, which intersects $\mob$ in the (inverse) stereographic projection of $C$. Any sphere $S$ represented by a point in $\ell^\perp$ intersects $C$ orthogonally, and contains both $\centersOf{S_1}$ and $\centersOf{S_2}$. 

The group of transformations of (Lorentz) Möbius geometry is $\mathrm{PO}(3,2)$, the group of projective transformations preserving the Lorentz Möbius quadric. In $\lor$, these transformations are generated by the isometries of $\lor$, scalings, and inversions in spheres. The inversion of a point $x\in \lor$ in a sphere $S \subset \lor$ is defined as the intersection of the line $x\vee \centersOf{S}$ with the polar plane of $x$ with respect to $S$, see Figure \ref{fig:lorentzbasics} (right).

\subsection{Timelike Lorentz Laguerre geometry}
\label{sec:laguerre}

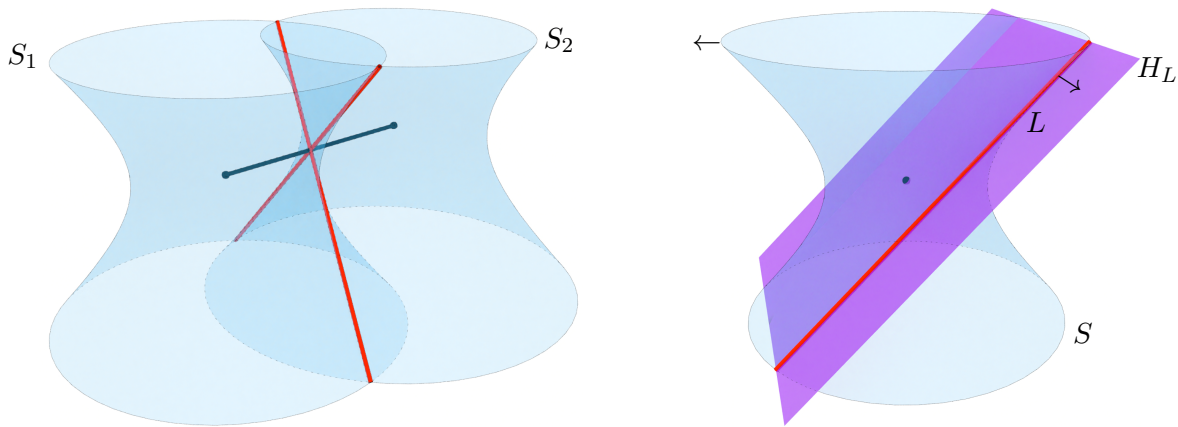
\begin{figure}
	\centering
	\begin{overpic}[width=.45\textwidth]{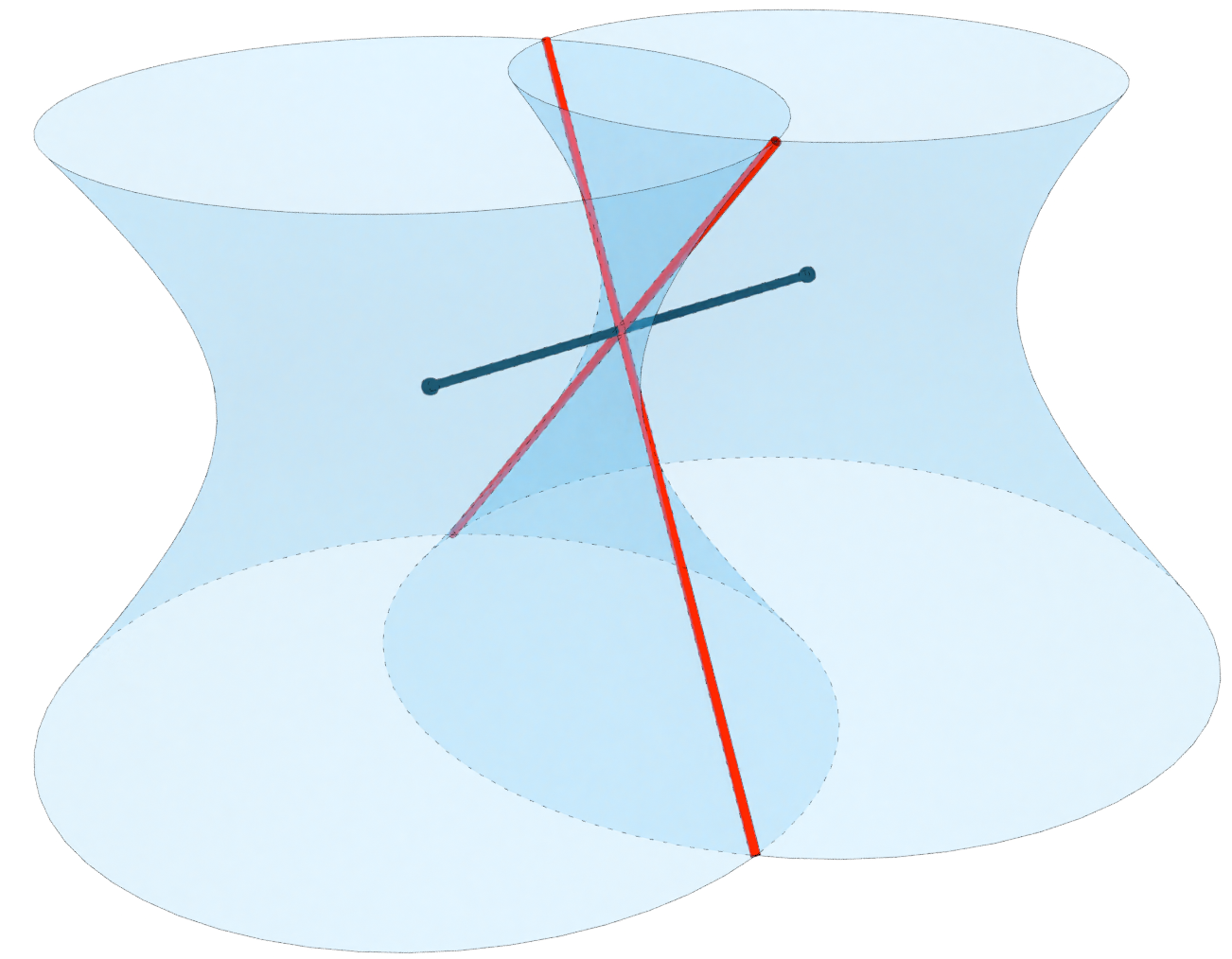}
		\put(-5, 67){$S_1$}
		\put(93, 70){$S_2$}
	\end{overpic}
	\hspace{1.5cm}
	\begin{overpic}[width=.35\textwidth]{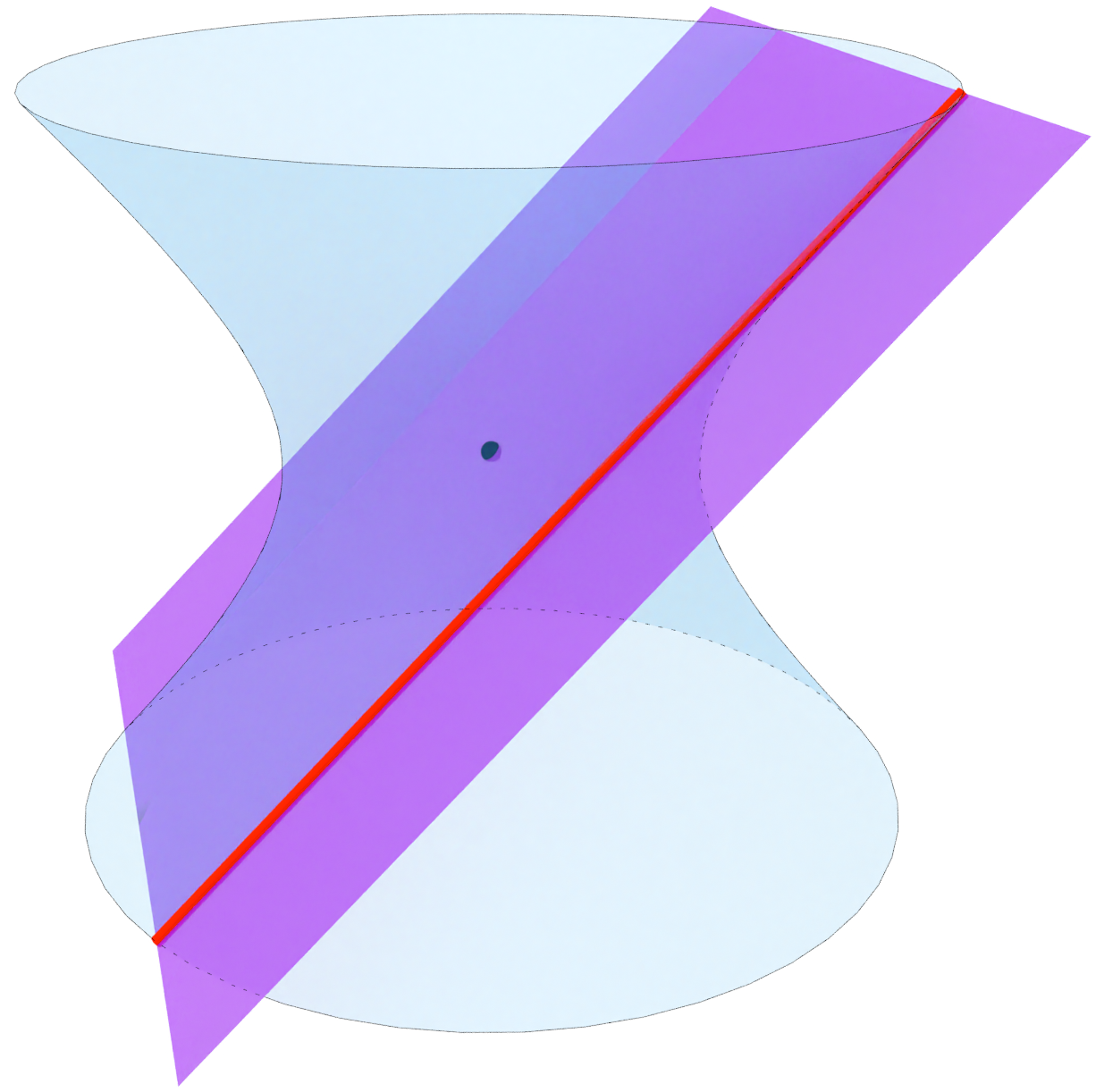}
		\put(-5.5,89.5){$\leftarrow$}
		\put(77,80){ \rotatebox[origin=c]{-33}{$\rightarrow$}}
		
		\put(99,83){$H_L$}
		\put(73,70){$L$}
		\put(84,20){$S$}
	\end{overpic}
	\caption{Left: two touching spheres, their common generators (red) span the tangent plane. Right: an oriented isotropic line (red) and the unique corresponding isotropic plane (violet), as well as an oriented sphere in contact with the oriented isotropic line.}
	\label{fig:touchingspheres}
\end{figure}

The objects of Laguerre geometry \cite{laguerrelaguerre} are oriented planes and spheres, and the incidence relations are given by oriented contact. Lorentz Laguerre geometry is a bit more involved compared to Euclidean Laguerre geometry, since there are planes and spheres of different signature. We only use the Laguerre geometry of timelike and isotropic planes and spheres in $\lor$ in this article. Hence, we only use the term ``Laguerre geometry'' when referring to timelike Lorentz Laguerre geometry. 

\subsubsection{Oriented planes, spheres and isotropic lines}

Consider a spacelike plane $H \in \pl_+(\lor)$. There are two normal vectors $N_\pm \in \lor$, which are orthogonal to $P - P'$ for all $P,P' \in H$ and such that the squared length of $N_\pm$ is $-1$. The two normal vectors are related by $N_+ = - N_-$. Assigning to $H$ one of the two normal vectors corresponds to assigning $H$ an \emph{orientation}. We denote the space of \emph{oriented spacelike planes} by $\opl_+(\lor)$. Analogously, we orient timelike planes except in this case the squared length of the normal vector is $+1$, the space of \emph{oriented timelike planes} is denoted by $\opl_-(\lor)$. Isotropic planes do not come with an orientation in Laguerre geometry.

Spacelike and timelike spheres are assigned an orientation (in the sense of an oriented surface), which induces an orientation on all tangent planes. We denote the sets of oriented spheres by $\osp_+$ and $\osp_-$, respectively. As such, we may think of spheres as being oriented ``inwards'' or ``outwards''. Null-spheres are not assigned an orientation.

An oriented plane $H$ is in \emph{contact} with an oriented sphere $S$ if $H$ is a tangent plane of $S$ with the same orientation. An oriented plane $H$ is in contact with a null-sphere $S$ if the center $\centersOf{S}$ is contained in $H$. An isotropic plane $H$ is in contact with an oriented or null-sphere $S$ if $\centersOf{S}$ is contained in $H$. If two objects are in contact we also say they are \emph{touching}.

Two spheres $S,S'$ are in contact if they share a point $P$ and the same oriented tangent plane at $P$. A null-sphere $S$ is in contact with a sphere $S'$ if $\centersOf{S} \in S'$. 

As discussed in Section~\ref{sec:basiclorentzgeometry} every isotropic line $L$ is contained in a unique isotropic plane $H_L$. The line $L$ divides $H_L$ into two regions, and we give $L$ an orientation by choosing one of the two regions, see also Figure~\ref{fig:touchingspheres}. We denote the set of oriented isotropic lines by $\oli_0(\lor)$.

If a timelike sphere $S$ contains an isotropic line $L$, then $\centersOf{S} \in H_L \setminus L$. The intersection $C = H_L \cap S$ consists of two parallel isotropic lines: $L$ and a second line $L'$, and $L'$ is the reflection of $L$ about $\centersOf{S}$. Moreover, $C$ is a circle in $H_L$ in the sense that every point on $C$ has the same distance to $\centersOf{S}$. If $S$ is an oriented timelike sphere, then there is a natural induced orientation of $C$ and thus of $L$ and $L'$ as well, either towards $\centersOf{S}$ or away from $\centersOf{S}$. If $L$ is an oriented isotropic line, then $S$ is in contact with $L$ if the orientation of $L$ by choosing a region in $H_L$ as above coincides with the induced orientation of $L$ as a subset of $C \subset S$.

Note that the set of all oriented timelike spheres and null-spheres that are in contact with an oriented isotropic line $L$ is a two-dimensional family. Each such sphere $S$ is uniquely defined by the oriented isotropic line $L$ and its center $\centersOf{S} \in H_L$.

\subsubsection{Cyclographic model}

The \emph{(timelike Lorentz) cyclographic model} $\cyc = \R^{2,2}$ is a practical tool in Laguerre geometry, since it models the space of oriented timelike spheres and null-spheres as a vector space. The Laguerre transformations correspond to the isometries of $\cyc$, which are composed of translations and linear transformations which leave the bilinear form
\begin{align}
	\cycsca{x,y} = x_1 y_1 + x_2 y_2 - x_3 y_3 - x_4 y_4
\end{align}
invariant.
We view $\lor$ as a subset of $\cyc$ via
\begin{align}
	\lor = \{ P \in \cyc \mid P_4 = 0 \}.
\end{align}
For $\cyclift{S} \in \cyc$ let the corresponding oriented sphere $S$ in $\lor$ be given by
\begin{align}
	S  = \xi_{\osp}(\cyclift{S}) = \{ Q \in \lor \mid |Q-\cyclift{S}|^2_{2,2} = 0 \}. \label{eq:cyclographicxi}
\end{align}
The orientation of $S$ is understood to be outwards if $\cyclift{S}_4 > 0$ and inwards if $\cyclift{S}_4 < 0$, and if $\cyclift{S}_4 = 0$ then $S$ is a null-sphere. Conversely, we call $\cyclift{S}$ the \emph{cyclographic lift} of $S$.

Let $S,T$ be two oriented spheres such that there is an oriented tangent plane to both $S$ and $T$ which touches $S$ in a point $J$ and $T$ in a point $K$. The \emph{tangential distance} of $S,T$ is the distance of $J,K$ and equals $|\cyclift{S} - \cyclift{T}|_{2,2}$. If $\cyclift{S},\cyclift{T}$ are the two corresponding points in $\cyc$, then the spheres $S, T$ are in oriented contact if and only if $|\cyclift{S} - \cyclift{T}|_{2,2} = 0$.

Each (non-oriented) timelike plane $E$ (signature $\si{+-}$) in $\lor$ corresponds to exactly two isotropic hyperplanes $H_\pm(E)$ in $\cyc$ of signature $\si{+-0}$, such that
\begin{align}
	H_\pm(E) \cap \lor = E.
\end{align}
The two choices correspond to the two orientations of $E$. The cyclographic lift $\cyclift{E}$ of an oriented timelike plane $E$ is the unique hyperplane (of the two $H_\pm(E)$) for which all points $\cyclift{S} \in \cyclift{E}$ correspond to oriented spheres $S$ in oriented contact with $E$.

Moreover, each isotropic plane $E$ in $\lor$ (signature $\si{+0}$) corresponds to exactly one isotropic hyperplane $H(E)$ in $\cyc$ of signature $\si{+-0}$, such that
\begin{align}
  H(E) \cap \lor = E.
\end{align}
The cyclographic lift $\cyclift{E}$ of an isotropic plane $E$ is the unique hyperplane for which all points $\cyclift{S} \in \cyclift{E}$ correspond to spheres $S$ in contact with $E$.

Analogously, each isotropic line $L$ (signature $\si{0}$) corresponds to exactly two fully isotropic planes $E_\pm(L)$ in $\cyc $ of signature $\si{00}$, such that
\begin{align}
	E_\pm(L) \cap \lor = L.
\end{align}
The two choices correspond to the two orientations of $L$. The cyclographic lift $\cyclift{L}$ of $L$ is the unique fully isotropic plane of the two $E_\pm(L)$ such that for all points $\cyclift{S} \in \cyclift{L}$ holds that $S$ is in oriented contact with $L$.

\subsubsection{Timelike Lorentz Laguerre transformations}

The (timelike Lorentz) Laguerre transformations correspond to the isometries of $\cyc$, that is they are the elements of $\mathrm{O}(2,2)$ plus translations. Laguerre transformations preserve the contact of timelike spheres (including null-spheres) with each other. They also preserve the contact of timelike spheres (including null-spheres) with timelike planes (including isotropic planes). However, the image of a null-sphere may be a timelike sphere and vice versa, analogously the image of an isotropic plane may be a timelike plane and vice versa, because the intersection of a signature $\si{+-0}$ plane with $\lor$ may either have signature $\si{+-}$ or $\si{+0}$. Note that scalings of $\cyc$ (Laguerre similarities) also preserve contact but are not Laguerre transformations in the stricter sense. 

\begin{remark}
	It is also possible to represent oriented timelike and isotropic planes as points on a degenerate quadric $\mathcal B$ (the timelike Lorentz Blaschke cylinder) of signature $\si{++--0}$ in $\RP^4$. In that case the Laguerre transformations (including scalings) correspond to projective transformations of $\RP^4$ that preserve $\mathcal B$.
\end{remark}

\subsection{Timelike Lorentz Lie geometry}
\label{sec:lie}

Since we only treat the case of timelike Lorentz Lie geometry and no other Lie geometry, we will generally just write Lie geometry when we mean timelike Lorentz Lie geometry.

The \emph{Lie quadric} is a quadric in $\RP^5$ of signature $\si{+++---}$. In Lie geometry, oriented timelike planes, isotropic planes, oriented timelike spheres, null-spheres and points are treated equivalently.  Each such object is identified with a point on the Lie quadric.

An object $A$ is in contact with an object $B$ if the point corresponding to $A$ is in the polar complement of $B$ with respect to the Lie quadric (and vice versa). Lie transformations are the transformations of $\RP^5$ that preserve the Lie quadric. The Möbius and Laguerre geometries are subgeometries of Lie geometry. Möbius transformations additionally preserve a fixed point inside the Lie quadric (of signature $\si{-})$, while Laguerre transformations additionally preserve a fixed point on the Lie quadric.



\section{Contact congruences and circle patterns} \label{sec:cccp}

\subsection{Conical nets and circle patterns} \label{sec:conicalandcp}

Given a (non-degenerate) map $f: \Z^2 \to \eucl$ and an edge $e=(v,v')\in E(\Z^2)$, we denote the line containing $f(v)$ and $f(v')$ by $\ell(e)$. Let us denote the reflection about the line $\ell(e)$ by $\rho(e)$, we call $\rho({e})$ the \emph{reflection about the edge $e$}.
Generally, the composition of four reflections about four lines intersecting in a point in $\R^2$ is a rotation around that intersection point. 

\begin{lemma} \label{lem:conicalreflectionid}
  A map $\cpm:\Z^2 \to \eucl$ is a conical net if and only if for every $v\in \Z^2$
	\begin{align}
		\prod_{e \ni v} \rho(e) = \mathrm{id}
	\end{align}	
 	is satisfied.
\end{lemma}
\begin{proof}
	See for example \cite{bsddgbook}. 
\end{proof}

Consider the dual graph $(\Z^2)^* \simeq F(\Z^2)$ of $\Z^2$, and let us denote by $E^*(\Z^2)$ the set of dual edges and by $e^*$ the dual edge of an edge $e \in E(\Z^2)$. In what follows we will denote by $\rho(e^*)$ the reflection about the original edge $e$, hence $\rho(e^*) = \rho(e)$.

A \emph{discrete connection} $\Gamma$ is a map from the oriented edges of a graph to some (automorphism) group, such that $\Gamma(v,v')\circ \Gamma(v',v) = \id$ for all edges $(v,v')$. A discrete connection is said to be \emph{flat} if it evaluates to the identity along each closed cycle.

\begin{lemma}\label{lem:flatConnection}
	Let $\cpm:\Z^2 \to \eucl$ be a conical net.
	The map $\rho: E^\ast(\Z^2) \to \operatorname{Iso}(\eucl)$ on 
	(oriented) dual edges is a flat discrete connection.
\end{lemma}
\begin{proof}
  Since reflections are involutions, $\rho$ is a
  discrete connection. Furthermore, $\rho$ is 
	flat if for all $f^* \in F^*(\Z^2)$ holds that
	$$
	\prod_{e^\ast \in f^*} \rho({e^\ast}) = \id,
	$$
	which is obviously equivalent to Lemma~\ref{lem:conicalreflectionid}.
\end{proof}

\begin{remark}
	Conical nets may also be characterized by an angle condition, see \cite{muellerconical}. Moreover, conical nets with some additional embedding constraints are known as t-embeddings, see \cite{clrdimer}.
\end{remark}

Each conical net corresponds to a 
$2$-parameter family of circle patterns \cite{muellerconical}. To see this, we demonstrate how to construct a circle pattern $p$ from the connection $\rho$ of a conical net $\cpm$. Begin by choosing a point $\cpf(f_0) \in \eucl$ assigned to a face $f_0$ containing a vertex $v_0$. For $p$ to be a circle pattern, it is necessary that $\cpf(f) = \rho(f,f')\circ \cpf(f')$ for all adjacent $f,f' \in F(\Z^2)$. This condition is automatically satisfied, because $\rho$ is flat. We may therefore use 
the reflections to determine all of $\cpf$. By construction, for
each vertex $v \in \Z^2$, the points $\cpf(f)$ for faces incident to $v$ lie on a circle. 
Therefore, this construction defines a circle pattern $\cp$ with centers given by $\cpm$. Since we may choose the initial point arbitrarily, there is indeed a 2-parameter family of circle patterns with center net $\cpm$. 


\subsection{Lorentz lift} \label{sec:lorlift}

We now describe the construction of the Lorentz lift $\lorlift\cp$ of a circle pattern $\cp$ to Lorentz space $\lor$, see Section~\ref{sec:lorentzgeometry} for details on Lorentz geometry. Recall that we embedded $\eucl \simeq \R^2$ as the restriction of $\lor \simeq \R^{2,1}$ to
\begin{align}
	\eucl = \{x \in \lor \ | \ x_3 = 0\}.
\end{align}

\begin{figure}
	\centering
		\begin{overpic}[width=.4\textwidth]{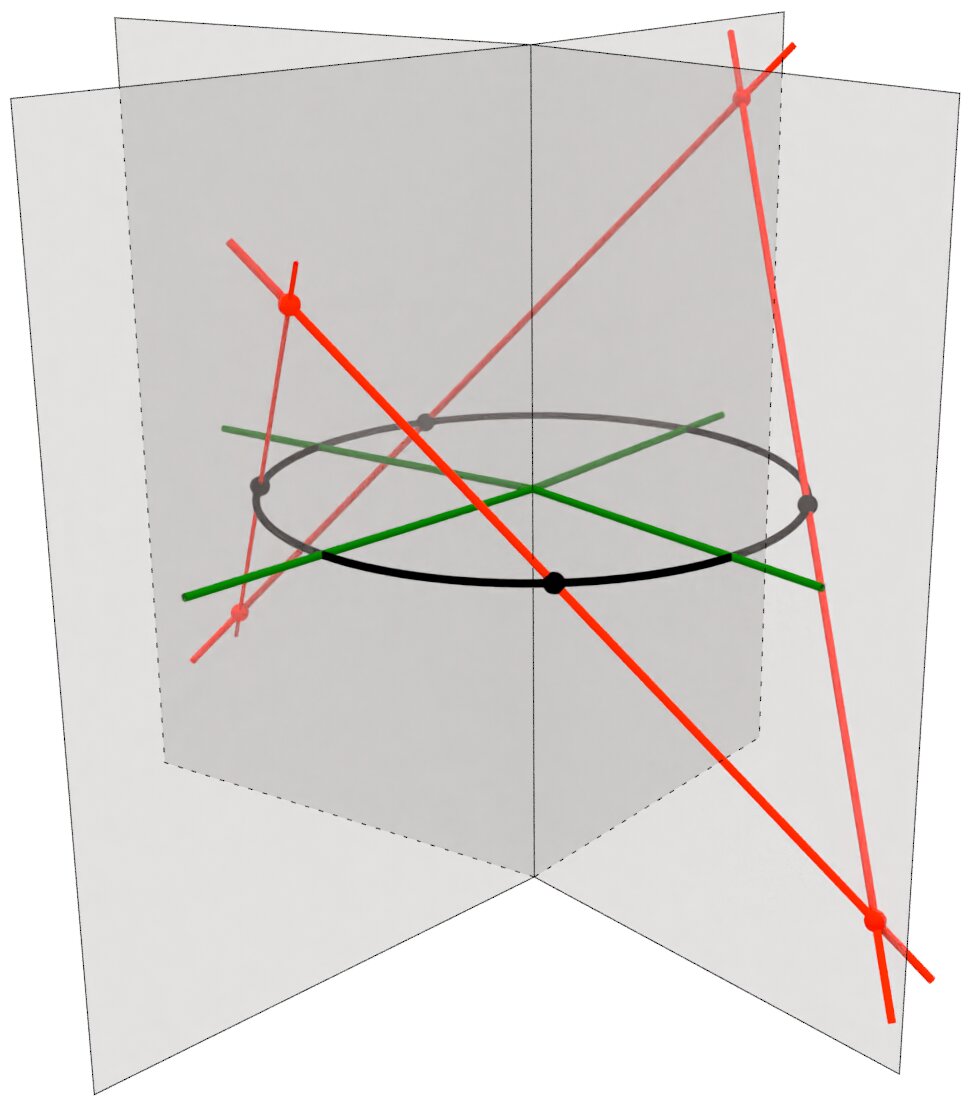}
		\put(59, 53){$\ell$}
		\put(87, 85){$\lorlift\ell$}
		\put(77, 36){$\lorlift\cpf$}
	\end{overpic}
	\hspace{1.5cm}
	\begin{overpic}[width=.4\textwidth]{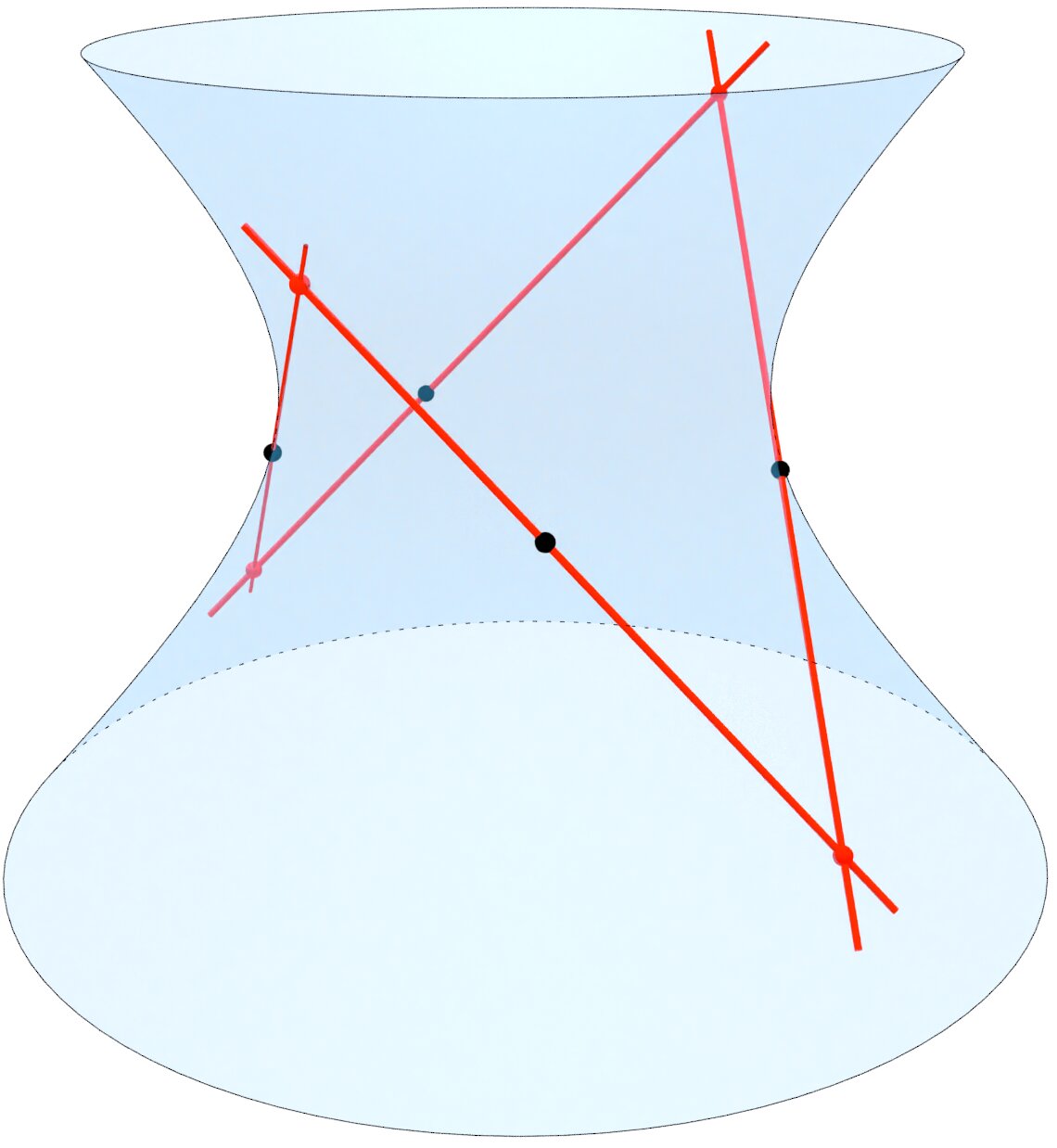}
		\put(85,91){$\lorlift{\cp}(v)$}
	\end{overpic}
	\caption{Left: the vertical planes $\lorlift\ell$ (gray) above the lines $\ell$ (green) of the conical net $\cpm$, and the isotropic lines $\lorlift\cpf$ (red) that are obtained by reflecting about these planes. Right: the unique timelike Lorentz sphere $\lorlift{\cp}(v)$ containing the four isotropic lines.} 
	\label{fig:lorlift}
\end{figure}

For each line $\ell \subset \eucl$ we denote by $\lorlift{\ell}$ the vertical
plane containing $\ell$, that is 
\begin{align}
	\lorlift \ell = \{x \in \lor \ | \ (x_1,x_2) \in \ell\}. \label{eq:vertplanes}
\end{align}
Also recall that for each edge $e = (v,v')\in E(\Z^2)$, we denote by $\ell(e)$ the line containing the two centers $\cpm(v)$ and $\cpm(v')$.
Moreover, we denote the reflection across the plane $\lorlift{\ell}(e)$ by $\lorlift{\rho}(e)$, since the reflection restricts to the reflection $\rho(e)$ about the line $\ell$ in $\eucl$. Since $\rho$ is a flat discrete connection in $\eucl$, it follows immediately that the discrete connection $\lorlift{\rho}$ in $\lor$ is flat as well.

In Section~\ref{sec:conicalandcp} we explained how to construct a circle pattern from a conical net using the reflections $\rho$. In analogous fashion, we now explain how to construct a contact congruence, which we call the \emph{Lorentz lift} $\lorlift{\cp}$, from a circle pattern $p$ and the connection $\rho$.

Fix a face $f_0 \in F(\Z^2)$ and an oriented isotropic line $\lorlift\cpf(f_0) \in \oli_0(\lor)$ (see Section~\ref{sec:laguerre}) through $\cpf(f_0)$. Since $\lorlift{\rho}$ is flat, we may propagate $\lorlift\cpf(f_0)$ by reflections to obtain a well-defined map 
\begin{align}
	\lorlift\cpf: F(\Z^2) \rightarrow \oli_0(\lor),
\end{align}
with the property $\lorlift\cpf (f) =  \lorlift{\rho}(f,f')\circ \cpf(f')$  (see Figure~\ref{fig:lorlift}). Moreover, by construction the intersection $\lorlift\cpf(f) \cap \eucl$ is $\cpf(f)$. Also note that if the face $f$ is adjacent to the face $f'$, then $\lorlift\cpf(f)$ intersects $\lorlift\cpf(f')$ in $\lorlift{\ell}(f,f')$, since $\lorlift{\ell}(f,f')$ is the fixed plane of the reflection $\lorlift\rho(f,f')$.

As a result, around each vertex $v\in \Z^2$ we obtain four oriented isotropic lines that intersect cyclically. There is a unique oriented timelike Lorentz sphere containing these four isotropic lines, and we denote this sphere by $\lorlift\cp(v)$. The only exception is if the four isotropic lines intersect in one common point, in which case they define a unique null-sphere $\lorlift\cp(v)$. In fact, it is not hard to see that $\lorlift\cp(v)$ is also obtained by rotating any one of the isotropic lines around the axis through $\cpm(v)$ orthogonal to $\eucl$. This in turn shows that the circle $\cp(v)$ coincides with the intersection $\lorlift\cp(v) \cap \eucl$. Moreover, the center $\odot {\lorlift\cp(v)}$ is on the aforementioned axis, and therefore the projection of $\odot{\lorlift\cp(v)}$ is $\cpm(v)$.

\begin{theorem}
	Any Lorentz lift $\lorlift{\cp}$ of a circle pattern $\cp$ is a contact congruence (see Definition~\ref{def:contactcongruence}).
\end{theorem}
\begin{proof}
	All that remains to prove is that for adjacent $v,v'\in \Z^2$ the spheres $\lorlift{\cp}(v)$ and $\lorlift{\cp}(v')$ are in oriented contact. Let $f,f'\in F(\Z^2)$ be such that $(f,f') = (v,v')^*$. By construction, the two oriented isotropic lines $\lorlift{\cpf}(f)$ and $\lorlift{\cpf}(f')$ are contained in both $\lorlift{\cp}(v)$ and $\lorlift{\cp}(v')$. Therefore the two spheres are indeed in oriented contact, with point of contact $\lorlift{\cpf}(f) \cap \lorlift{\cpf}(f')$ and tangent plane $\lorlift{\cpf}(f) \vee \lorlift{\cpf}(f')$.
\end{proof}

If we are only given the data of a conical net, and not the circle pattern, the initial isotropic line in the construction of the Lorentz lift may be chosen arbitrarily in $\oli_0(\lor)$. Therefore, for each conical net there is a three-dimensional family of corresponding contact congruences.

Now, let us prove the converse direction, that every contact congruence defines a circle pattern.

\begin{proof}[Proof of Theorem \ref{th:cctocp}]
	Let $\cc$ be a contact congruence, then we define a circle pattern $\cp$ via
	\begin{align}
		\cp(v) &= \cc(v) \cap \eucl, & \quad \cpf(f) &= \ccf(f) \cap \eucl.
	\end{align}
  Clearly, the first intersection yields circles and the other yields points as required. Moreover, for every face $f \in F(\Z^2)$ incident to a vertex $v \in \Z^2$ we have $\ccf(f) \subset \cc(v)$, therefore we also have $\cpf(f) \in \cp(v)$. Moreover, if $\cpm(v)\neq\ccm(v)$, then $\cc(v)$ is rotation-symmetrical around the line $\cpm(v) \vee \ccm(v)$, hence, $\cpm(v)$ is indeed the orthogonal projection of $\ccm(v)$.
\end{proof}

Let us also observe an additional interesting property of the center net $\centersOf{\cc}$ of a contact congruence $\cc$.

\begin{lemma}\label{lem:centerNetIsIsoConjugate}
	Every center net $\ccm$ of a contact congruence $\cc$ is an isotropic conjugate net. Every isotropic conjugate net is the center net of a contact congruence.
\end{lemma}
\begin{proof}
	Let $\cc$ be a contact congruence. Consider a face $f \in F(\Z^2)$ incident to the four vertices $v_1,v_2,v_3,v_4\in \Z^2$. Since the four spheres $\cc(v_1), \cc(v_2), \cc(v_3), \cc(v_4)$ each contain the isotropic line $\ccf(f)$, the four points  $\ccm(v_1), \ccm(v_2), \ccm(v_3), \ccm(v_4)$ are in an isotropic plane (see Section~\ref{sec:laguerre}). Therefore $\ccm$ is an isotropic conjugate net. 
	
	Now, let us assume $\ccm$ is some isotropic conjugate net, how do we construct a corresponding contact congruence $\cc$? Note that for each face $f$ there is a one-parameter family of isotropic lines contained in the isotropic plane of $\ccm$ corresponding to $f$.
	Choose an initial face $f_0$ and an initial oriented isotropic line $\ccf(f_0)$ in the isotropic plane of $\ccm$ corresponding to $f_0$.  Now we claim that $\cc$ is completely determined.
	
	To prove this, consider two adjacent faces $f,f'$, let $(v,v')^* = (f,f')$ and assume $\ccf(f)$ is already determined. Then $\ccf(f')$ needs to intersect $\ccf(f)$ in the line
	\begin{align}
		\ccm(v) \vee \ccm(v'),
	\end{align}
	and the orientation of $\ccf(f)$ with respect to $\ccm(v)$ (away from or towards) needs to agree with the orientation of $\ccf(f')$. Thus $\ccf(f)$ determines $\ccf(f')$. In this way all of $\ccf$ is determined from $\ccm$ and $\ccf(f_0)$. It remains to check that this is well-defined.
	
	There is a unique oriented timelike sphere with center $\ccm(v)$ containing $\ccf(f)$, and a unique oriented timelike sphere with center $\ccm(v)$ containing $\ccf(f')$. But since $\ccf(f)$ and $\ccf(f')$ intersect, this is the same oriented timelike sphere, which is $\cc(v)$. Hence $\cc(v)$ is well-defined, since it is the same for every vertex $v$. 	
	
	Consider a vertex $v$ and the four adjacent faces $f_1,f_2,f_3,f_4 \in F(\Z^2)$. Assume we began with knowing $\ccf(f_1)$ and iterate the construction above to obtain $\ccf(f_2)$, $\ccf(f_3)$ and $\ccf(f_4)$, and then iterate once more to obtain a line $L$ that should be $\ccf(f_1)$. However, all we know at this point is that $L$ is in the intersection of $\cc(v)$ with the isotropic plane corresponding to $f_1$, and there are two such lines.
	
  To see that $L$ coincides with $\ccf(f_1)$ we need an additional argument. For each pair $f_i$, $f_{i+1}$ consider the plane that is orthogonal to $\eucl$ and that contains the intersection of the two isotropic planes corresponding to $f_i$ and $f_{i+1}$ (these planes correspond to those in Equation~\eqref{eq:vertplanes}). Then $\ccf(f_{i+1})$ is actually the reflection of $f_i$ about this plane. All four of these planes intersect in the line $A$ orthogonal to $\eucl$ and pass through $\ccm(v)$. The composition of four such reflections is a rotation about $A$. As a result, $L$ needs to coincide with $\ccf(f_1)$ since the other option is not a rotation of $\ccf(f_1)$. Thus, $\ccf$ is indeed also well-defined which concludes the proof.
\end{proof}

\section{Cyclographic lift and origami-map}\label{sec:cyclift}

Recall that we defined the cyclographic model $\cyc$ for (timelike Lorentz) Laguerre geometry and the cyclographic lifts of oriented timelike spheres and planes as well as isotropic lines in Section~\ref{sec:laguerre}.

\begin{definition}
  Let $\cc$ be a contact congruence. The \emph{cyclographic lift}
  $\cyclift{\cc}$ is the pair of maps 
  \begin{align}
    \cyclift{\cc}: \Z^2 &\rightarrow \cyc , & v &\mapsto \cyclift{(\cc(v))},\nonumber
    \\
    \cyclift{\ccf}: F(\Z^2) &\rightarrow \pl_0(\cyc), &f &\mapsto \cyclift{(\ccf(f))}. 
    \label{eq:cclift}
  \end{align}
\end{definition}

\begin{remark}
	Note that the projection of $\cyclift\ccf(f)$ to $\lor$ consists of the centers of all the oriented spheres containing $\ccf(f)$, and the projection of $\cyclift\cc(v)$ is the center  $\ccm(v)$. 
\end{remark}

Given a contact congruence $\cc$ Equation~\eqref{eq:cclift} determines $\cyclift{\cc}$ uniquely. We only need to verify that the planes $\cyclift{\ccf}$ are fully isotropic.

\begin{lemma}\label{lem:isotropiccyclift}
	The cyclographic lift $\cyclift{\cc}$ of every contact congruence $\cc$ -- as given by Equation~\eqref{eq:cclift} -- is a conjugate net such that for each $f\in F(\Z^2)$ the plane $\cyclift{\ccf}(f)$ is a fully isotropic plane. Conversely, every conjugate net in $\cyc$ with fully isotropic planes is the cyclographic lift of a contact congruence.
\end{lemma}
\begin{proof}
	Let $c$ be a contact congruence. Since the four spheres $\cc(v_1)$, $\cc(v_2)$, $\cc(v_3)$, $\cc(v_4)$ around a face $f$ contain $\ccf(f)$, the cyclographic lifts of the four spheres lie in the fully isotropic plane $\cyclift{(\ccf(f))} = \cyclift{\ccf}(f)$ (as discussed in Section~\ref{sec:laguerre}). 
	
  Conversely, let $\cyclift{\cc}$ be a fully isotropic conjugate net. At each face $f$ the adjacent four vertices $v_1$,$v_2$,$v_3$,$v_4$ are mapped to the spheres $\xi_{\vec{S}}(\cyclift{\cc}(v_i))$ (as in Equation~\ref{eq:cyclographicxi}), which are in oriented contact with the oriented isotropic line $\cyclift{\ccf(f)}\cap \lor$. 
\end{proof}

If a contact congruence $c$ is the Lorentz lift $\lorlift{\cp}$ of some circle pattern $\cp$ or of a conical net $\cpm$, we also consider a cyclographic lift $\cyclift{c}$ to be a cyclographic lift $\cyclift{\cp}$ of $\cp$ or a cyclographic lift $\cyclift{\cpm}$ of $\cpm$.

In the remainder of this section, we compare the cyclographic lift to a construction introduced in \cite{clrdimer}. First, we choose an integral $\mathscr R$ of the connection $\rho$ (see Section~\ref{sec:conicalandcp}), that is a map
\begin{align}
  \mathscr R&: F(\Z^2) \rightarrow  \mathrm{Iso}(\eucl), & \mathscr R(f') &= \rho(f,f') \circ \mathscr R(f) \quad \text{for all}\ (f,f')^* \in E(\Z^2).
\end{align}
Clearly, such a map $\mathscr R$ is unique up to composition with an element of $\mathrm{Iso}(\eucl)$. The ambiguity is usually eliminated by requiring that $\mathscr R$ is the identity for some chosen base face $f_0$. Note that, up to the just mentioned ambiguity, $\mathscr R$ is well-defined due to the flatness of $\rho$, see Lemma~\ref{lem:flatConnection}.

\begin{definition} \label{def:origamimap}
	Let $\cpm$ be a conical net and let $\mathscr R$ be an integral of $\rho$. The \emph{origami map} $o$ is the map
	\begin{align}
		o&: \Z^2 \rightarrow \eucl, & o(v) &\mapsto \mathscr R(f_v)( \cpm(v)),
	\end{align}
	where $f_v$ is any face adjacent to $v$.
\end{definition}
If $f_1,f_2$ are two faces incident to $v$, then $\cpm(v)$ is a fixed point of $\mathscr R(f_1)\mathscr R^{-1}(f_2)$. Hence the choice of $f_v$ in Definition~\ref{def:origamimap} has no influence on $o$.

Given a conical net $\cpm$, the map
\begin{align}
	h: \Z^2 \rightarrow \R^{2,2} \simeq \cyc, \quad v \mapsto (\cpm(v), o(v)),
\end{align}
was defined in \cite{clrdimer}.

\begin{theorem}
	The map $h$ is a cyclographic lift $\cyclift{\cpm}$ of $\cpm$.
\end{theorem}
\begin{proof}
  Chelkak, Laslier and Russkikh have shown in \cite{clrdimer} that $h$ is a fully isotropic conjugate net. Therefore -- due to Lemma~\ref{lem:isotropiccyclift} -- $h$ is a cyclographic lift of some conical net $\cpm'$. Moreover, if $\cp$ is the corresponding circle pattern of a conical net $\cpm$, then $h$ agrees with $\cyclift{p}$ in the first two coordinates, that is, in $\eucl$. As a result, $\cpm' = \cpm$ and $h$ is indeed a cyclographic lift $\cyclift{\cpm}$ of $\cpm$.
\end{proof}

Consequently, one may view the Lorentz lift of a conical net as a geometric interpretation of the definition of $h$ given by \cite{clrdimer}, by viewing $\R^{2,2}$ as the cyclographic model $\cyc$ of $\lor$. 

Note that the origami map is only defined up to a Euclidean isometry. These three degrees of freedom in the construction of the origami map correspond -- in a not so obvious way -- with the choices we made in the construction of $\cyclift{p}$: a two parameter choice of an initial point for $\cpf$ in $\eucl$, and a one parameter choice of an isotropic line through that point in $\lor$.

\section{Cycle patterns} \label{sec:laguerrecp}

Cycle patterns are an analogue of circle patterns where we replace circles in $\eucl$ with oriented circles in $\eucl$ and intersection points with oriented tangents. The word \emph{cycle} was originally used to refer to oriented circles \cite{laguerrelaguerre}. Special cases of cycle patterns appear in the literature \cite{gsfastrhombisch,abconfocal,bstincircular, fairleythesis}. In our case, we show that incircular nets and circle packings correspond to a special case of cycle patterns, which will also be useful later when characterizing isothermic incircular nets.

\begin{figure}
	\centering
	\includegraphics[width=0.49\textwidth]{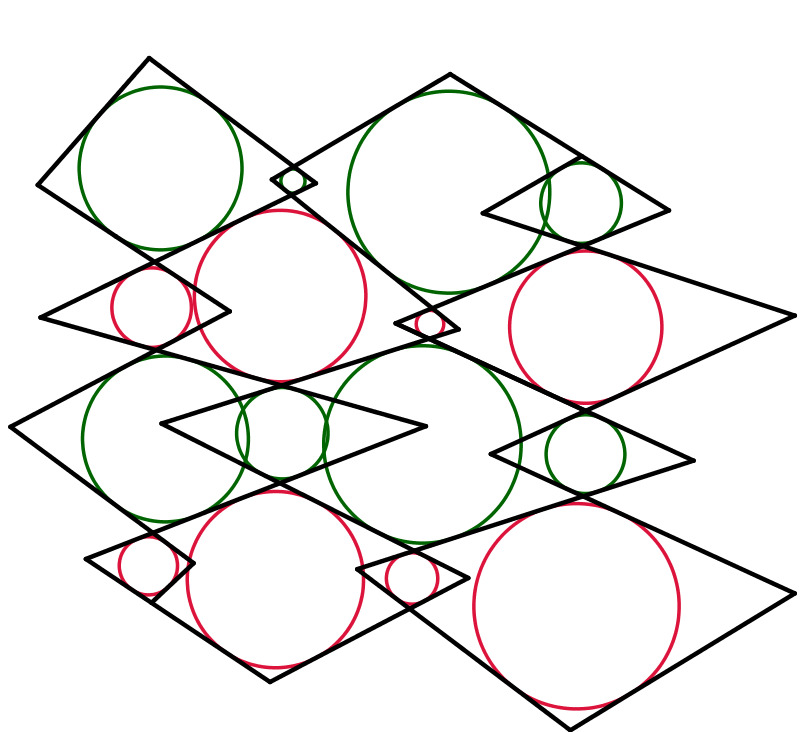}
	\includegraphics[width=0.49\textwidth]{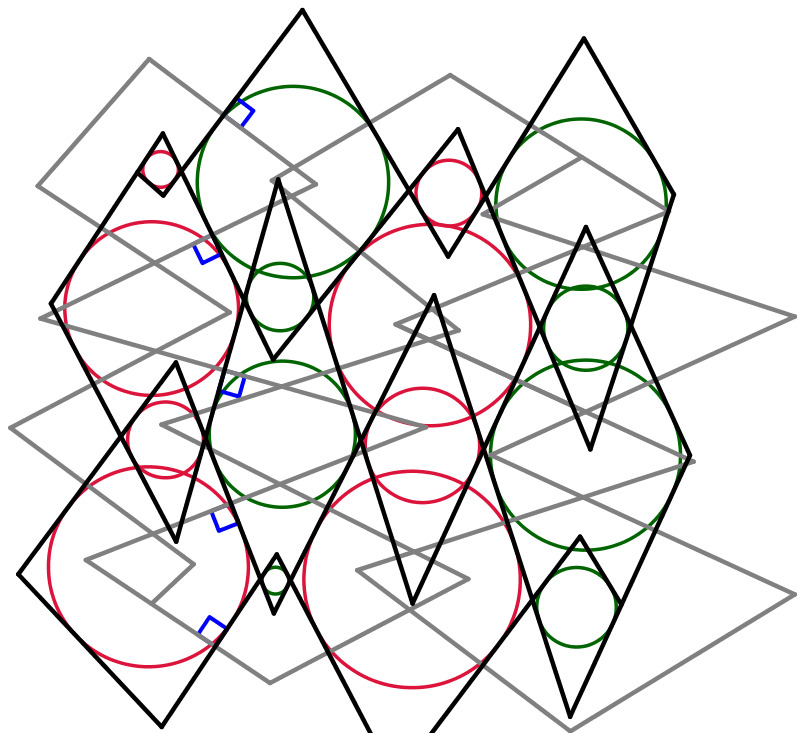}
	\caption{Left: a cycle pattern $\lcp$, green circles are oriented outwards, red circles inwards. Right: a cycle pattern $\xi$ that is concentric with and orthogonal to $\lcp$. For comparison $\lcp$ is drawn in gray and some right angles are highlighted in blue.}
	\label{fig:cyclepattern}
\end{figure}

\begin{definition}\label{def:laguerrecp}
  A \emph{cycle pattern} is a pair of maps 
  \begin{align}
    \lcp&: \Z^2 \rightarrow \osp(\eucl), & 
    \lcpf&: F(\Z^2) \rightarrow \oli(\eucl),
  \end{align}
such that $\lcp(v)$ is in contact with $\lcpf(f)$ whenever $v$ and $f$ are incident, see Figure~\ref{fig:cyclepattern}. The \emph{conical net} of a cycle pattern $\lcp$ is the map $\lcpm: \Z^2 \rightarrow \eucl$, where $\lcpm(v)$ is the center of $\lcp(v)$ for all $v\in \Z^2$.
\end{definition}

It is not difficult to see that the conical net of a cycle pattern also satisfies Lemma~\ref{lem:conicalreflectionid}. Hence, the conical net $\lcpm$ of a cycle pattern $\lcp$ is indeed a conical net in the sense of Definition~\ref{def:circlepattern}.

A cycle pattern is constructed from a conical net analogously to the circle pattern case. Begin by choosing an initial line $\lcpf(f_0)$. Now, if $f,f'$ are adjacent faces, then $\lcpf(f) = \rho(f,f') \circ  \lcpf(f')$, where $\rho(f,f')$ is the reflection defined by $\lcpm$, as introduced in Section~\ref{sec:conicalandcp}. Since this construction involves freely choosing an initial line, there is a 2-parameter family of cycle patterns that corresponds to the same conical net.

For an oriented timelike or null-sphere $S \in \osp(\lor)$, consider the plane $E$ parallel to $\eucl$ and containing $\centersOf{S}$. Let $\smallc(S) \in \oci(\lor)$ denote the spacelike circle $S \cap E$, which inherits its orientation from $S$ (and is the smallest Euclidean circle contained in $S$). Note that if $S$ is a null-sphere then $\smallc(S)$ has radius 0.

\begin{theorem}\label{th:cctolcp}
	Every contact congruence $\cc$ defines a cycle pattern $\lcp$ via the orthogonal projection $\pi_{\eucl}$ by
	\begin{align}
		\lcp(v) &= \pi_{\eucl} \circ \smallc(\cc(v)), & \lcpf(f) &= \pi_{\eucl} \circ \ccf(f).
	\end{align}
	The conical net $\lcpm$ is the orthogonal projection $\pi_\eucl(\ccm)$ of the center net $\ccm$.
\end{theorem}
\begin{proof}
  Let $S \in \osp$ be a timelike sphere. Since the tangent planes at $\smallc(S)$ are orthogonal to the plane $\eucl \subset \lor$, the orthogonal projection $\pi_\eucl(L)$ of any isotropic line $L \subset S$ is tangent to  $\pi_\eucl(\smallc(S))$.  The correct orientation of the contact is due to the arguments of Section~\ref{sec:laguerre}. If $S$ is a null-sphere then $\smallc(S)$ has radius 0. Therefore $\pi_\eucl(\smallc(S))$ is a point contained in $\pi_\eucl(L)$ and the statement follows.
\end{proof}

Consider a contact congruence $\lorlift{\cp}$ obtained as a Lorentz lift from a circle pattern $\cp$, and the projected cycle pattern $\lcp$. Then each circle $\cp(v)$ is concentric with $\lcp(v)$, and each oriented line $\lcp(f)$ contains the intersection point $\cp(f)$. Moreover, since $\lcp$ defines the same connection $\rho$, it is also possible to construct a Lorentz lift $\lorlift{\lcp}$ that agrees with $\lorlift{\cp}$.

\begin{remark}
  In Section~\ref{sec:laguerre} we explained how $\cyc \simeq \R^{2,2}$ is the cyclographic model of $\lor \simeq \R^{2,1}$. However, it is also possible to interpret $\R^{2,1}$ as the cyclographic model of $\eucl \simeq \R^2$. In that way, every point $x \in \lor$ defines an oriented circle $\xi(x)$ in $\eucl$, and the center of $\xi(x)$ is the orthogonal projection of $x$ onto $\eucl.$ Moreover, every isotropic plane $E$ in $\lor$ defines an oriented line $\xi(E)$ in $\eucl$. If $x\in E$, then $\xi(x)$ is in contact with $\xi(E)$. Now, consider a contact congruence $\cc$ and its center net $\ccm$. Since $\ccm$ is an isotropic conjugate net, $\xi(\ccm)$ is indeed a cycle pattern. Furthermore, if $\cc$ is the Lorentz lift of a circle pattern $\cp$ and a cycle pattern $\lcp$, then each line $\xi(\ccm)(f)$ contains $\cp(f)$ and is orthogonal to $\lcp(f)$. The orthogonality is due to the fact that $\lcp(f)$ is the projection of an isotropic line contained in the isotropic plane that is intersected with $\eucl$ to obtain $\xi(\ccm)(f)$. Combined, the two cycle patterns define an oriented orthogonal frame at each face. The Lie transformations of $\lor$ correspond to conformal oriented frame transformations in $\eucl$. We will investigate this relation in more depth in future research. 
\end{remark}



\section{Null congruences and incircular nets, s-embeddings} \label{sec:nullcongruences}

Recall from Definition~\ref{def:circlepacking} that we call a circle pattern $\cp$ a circle packing if adjacent circles in $\cpb$ are touching. Also recall that as a consequence, $\cpmb$ is an incircular net, and the centers of the incircles are given by $\cpmw$ (although the circles of $\cpw$ and the incircles do not coincide, see Figure~\ref{fig:circlepacking}). In the previous section we saw that there is a relation between circle patterns and cycle patterns, now we investigate how this relation specializes in the case of circle packings.

\begin{definition}\label{def:nulllaguerrecp}
  A \emph{Laguerre incircular net} is a cycle pattern $\lcp$ such that the radii of all circles $\lcpb$ are zero.
\end{definition}

Let $\cp$ be a circle packing. Then we define a Laguerre incircular net $\lcp$ by setting $\lcpm = \cpm$ and by setting $\lcpb$ equal to circles of radii zero (with centers $\cpmb$). Consequently, each line $\lcpf(f)$ needs to be the line
\begin{align}
	\cpmb(b)\vee \cpmb(b') = \lcpmb(b)\vee \lcpmb(b'),
\end{align}
where $b,b'$ are the two black vertices incident to $f$. Furthermore, the circles $\lcpw$ need to be the incircles of the incircular net $\cpmb$. The orientations of the lines and circles of $\lcp$ are not canonically given. However, if we choose the orientation of one line $\lcp(f_0)$, then all orientations are fixed because of the contact conditions on $\lcp$.  One may also obtain the orientations iteratively from the reflections $\rho$ assigned to edges. Since $\rho$ comes from $\cp$ and is flat, the orientations of $\lcp$ are well-defined.

Conversely, every Laguerre incircular net $\lcp$ defines a unique incircular net $\cpb$ just by forgetting the orientations.

\begin{remark}
  Let us add here a remark provided for the reader familiar with s-embeddings and the way s-embeddings are usually considered as a special case of t-embeddings, for example in \cite{klrr}, since we employ slightly different combinatorics. Given an incircular net $\cpmb: \Zb^2 \rightarrow \eucl$, the corresponding circle pattern $p$ consists of a circle per vertex of $\Z^2$ and a point per face of $\Z^2$. In contrast, in the setup of \cite{klrr} they consider the graph $G$ which is $\Z^2$ but with an extra edge $(b,b')$ for every face $f$ of $\Z^2$, where $b,b'$ are the two black vertices of $f$. Subsequently, given $\cpmb$ they define the circle pattern $\cp'$ on $G$ which satisfies that the circles of $\cp'$ coincide with the circles of $\cp$. However, since $G$ has two faces per face of $\Z^2$, $\cp'$ also has two intersection points per face of $\Z^2$, one of which coincides with the intersection point of $\cp$. The other point is the intersection point of $\cp$ after performing Miquel dynamics (see Section~\ref{sec:miquel}). In fact, by using the local Miquel dynamics rule (see~\cite{klrr}) at every other circle of $\cpw$, the graph $G$ becomes $\Z^2$. Therefore the approach in \cite{klrr} is related to our approach by a sequence of Miquel moves. The same is true if $\Zb^2$ is replaced by a quad graph with bipartite dual. If the dual is not bipartite, the \cite{klrr} approach still works, but our approach is not possible (without using a branched cover). However, since our goal are S-isothermic nets, for which the dual needs to be bipartite anyway, we do not consider this a loss for our present work. Instead, we prefer to work with only one intersection point per face of $\Z^2$, since in this manner we can treat null congruences as a special case of contact congruences without a change in combinatorics.
\end{remark}

In Section~\ref{sec:lorlift}, we showed how to geometrically construct a Lorentz lift $\lorlift\cp$ of a conical net $\cpm$, by choosing an initial isotropic line $\lorlift\cpf(f_0)$ associated with a face $f_0$ of $\Z^2$. The isotropic lines $\lorlift\cpf$ of all other faces are then well defined by the reflections $\lorlift\rho$ about the vertical planes $\lorlift\ell$ containing the edges of $\cpm$. These isotropic lines are generators of the Lorentz spheres $\lorlift\cp$ living at the vertices, in the sense that if $v\in\Z^2$ is incident to $f\in F(\Z^2)$ then the Lorentz sphere $\lorlift\cp(v)$ is obtained by rotating $\lorlift\cpf(f)$ around the vertical line containing $\cpm(v)$.

For an incircular net $\cpb$ there is a preferred choice for the initial isotropic line  $\lorlift\cpf(f_0)$ in the construction of the Lorentz lift. We choose the initial isotropic line $\lorlift\cpf(f_0)$ in the vertical plane that contains $\cpmb(b_0) \vee \cpmb(b_1)$, where $b_0, b_1$ are the two black vertices incident to $f_0$. Let $f_1$ be a face adjacent to $f_0$ such that $f_0$ and $f_1$ share $b_1$ and a white vertex $w$, and let $b_2$ denote the second black vertex of $f_1$. Then $\cpmw(w) \vee \cpmb(b_1)$ is an angle bisector of the two lines 
\begin{align}
	l_0 &= \cpmb(b_0) \vee \cpmb(b_1), & l_1 &= \cpmb(b_1) \vee \cpmb(b_2).
\end{align}
As a result, the reflection $\rho(w,b_1)$ maps $l_0$ to $l_1$ and vice versa. Hence, $\lorlift\cpf(f_1)$ is in the vertical plane above $l_1$, and the two isotropic lines $\lorlift\cpf(f_0)$, $\lorlift\cpf(f_1)$ intersect in the vertical axis through $\cpmb(b_1)$. Therefore all isotropic lines are contained in the vertical planes above the edges of the incircular net. Moreover, all isotropic lines  $\lorlift\cpf(f)$ of faces $f$ adjacent to a fixed black vertex $b$ intersect in a point above $\cpmb(b)$. Consequently, all spheres $\lorlift{\cpb}$ are actually null-spheres and thus $\lorlift{\cp}$ is a null congruence as defined in Definition~\ref{def:nullcongruence}.

We are now ready to prove the converse of the above which is the statement of Theorem~\ref{th:nullcctopacking}: every projection of a null congruence yields a circle packing.

\begin{proof}[Proof of Theorem~\ref{th:nullcctopacking}]
  Let $\cc$ be a null congruence and let $\cp$ be the circle pattern defined by Theorem~\ref{th:cctocp}, that is, the intersections of the spheres of $\cc$ with $\eucl$. By definition, the spheres $\ccb$ are null-spheres. As a direct result, if $b,b'$ are adjacent vertices in $\Zb^2$, the circles $\cpb(b)$, $\cpb(b')$ are touching.
\end{proof}

Of course, it follows immediately that $\cpmb$ is an incircular net if $\cc$ is a null congruence, since $\cp$ is a circle packing. Moreover, it also follows that the cycle pattern $\lcp$ obtained from a null congruence $\cc$ via Theorem~\ref{th:cctolcp} is a Laguerre incircular net as in Definition~\ref{def:nulllaguerrecp}.

\begin{remark}
	Let $\cc$ be a null congruence and let $b_1,b_2,b_3,b_4$ be four vertices adjacent to a white vertex $w$ in cyclic order. Then the line $\ccmb(b_1) \vee \ccmb(b_3)$ is Lorentz orthogonal to the line $\ccmb(b_2) \vee \ccmb(b_4)$. As such, $\ccmb$ is the \emph{control net} of an \emph{orthogonal checkerboard pattern}, see \cite{dellingercheckerboard}. These orthogonal checkerboard patterns are also explored as \emph{orthogonal binets} \cite{atbinets}. A special case of orthogonal binets are \emph{principal binets}, which satisfy additional local coplanarity constraints. It would be interesting to understand the set of incircular nets and null congruences that corresponds to principal binets. Principal binets also come with a Lie lift, which appears to be different from (but may be related to) the Lie lift discussed in this paper.
\end{remark}

Let us mention a characterization of null congruences that will specialize neatly to isothermic congruences in Lemma~\ref{lem:twopointcharacterization}.

\begin{lemma}
	Let $\ccw: \Zw^2 \rightarrow \osp_-$ be such that for all adjacent $w,w'\in \Zw^2$ the spheres $\ccw(w)$, $\ccw(w')$ are in oriented contact. Then there exists a null congruence $\cc$ such that $\ccw$ is its restriction to $\Zw^2$ if and only if the four spheres of $\ccw$ around each white face $f$ intersect in a point $P_f$.
\end{lemma}
\begin{proof}
  If an oriented sphere contains $P_f$ then it is also in oriented contact with the null-sphere centered at $P_f$.
\end{proof}



\section{Miquel dynamics}\label{sec:miquel}

\begin{figure}
	\centering
	\begin{minipage}{.44\textwidth}
		\includegraphics[width=\textwidth]{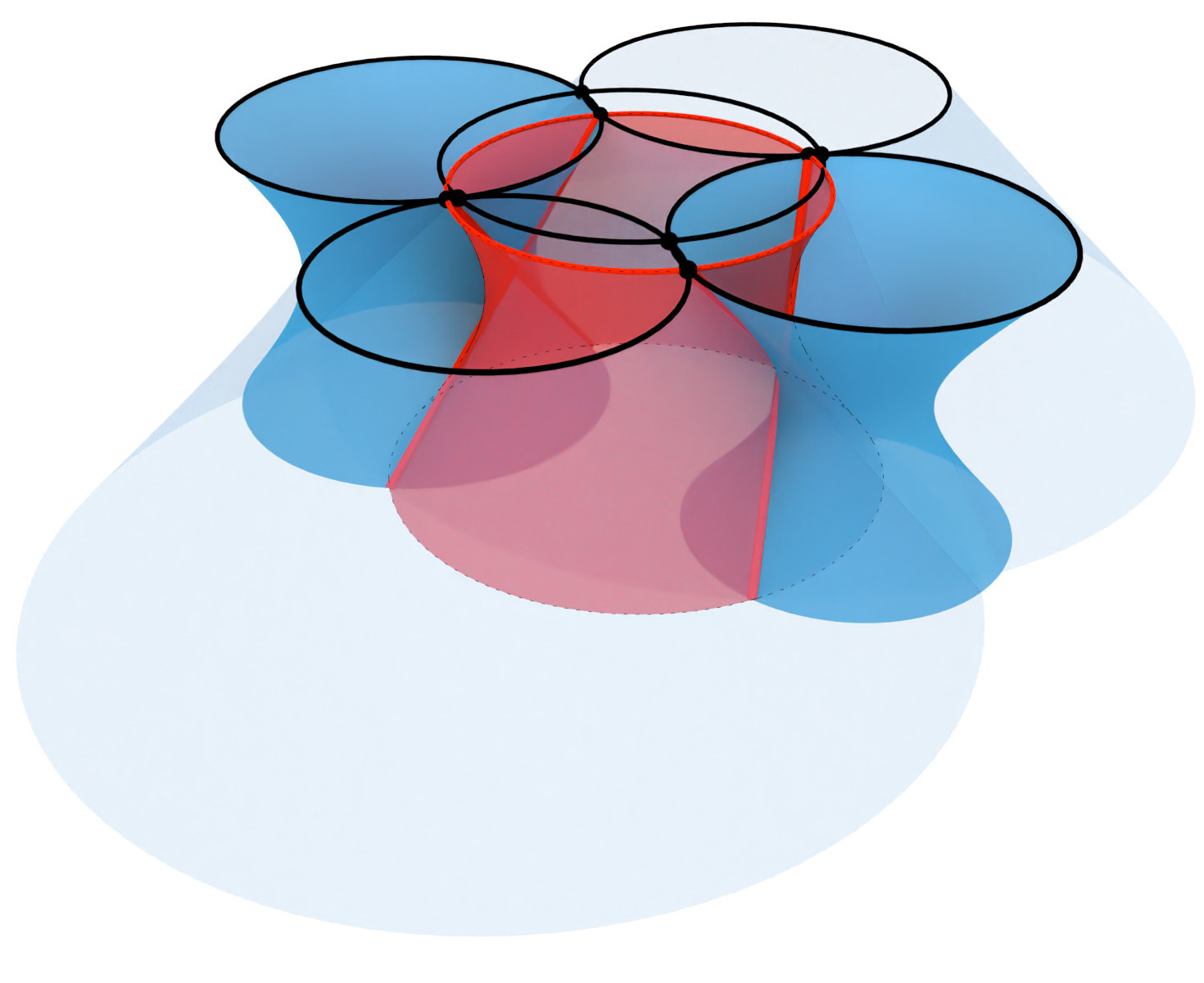}
	\end{minipage}
\hspace{.8cm}
	\begin{minipage}{.44\textwidth}
		\includegraphics[width=\textwidth]{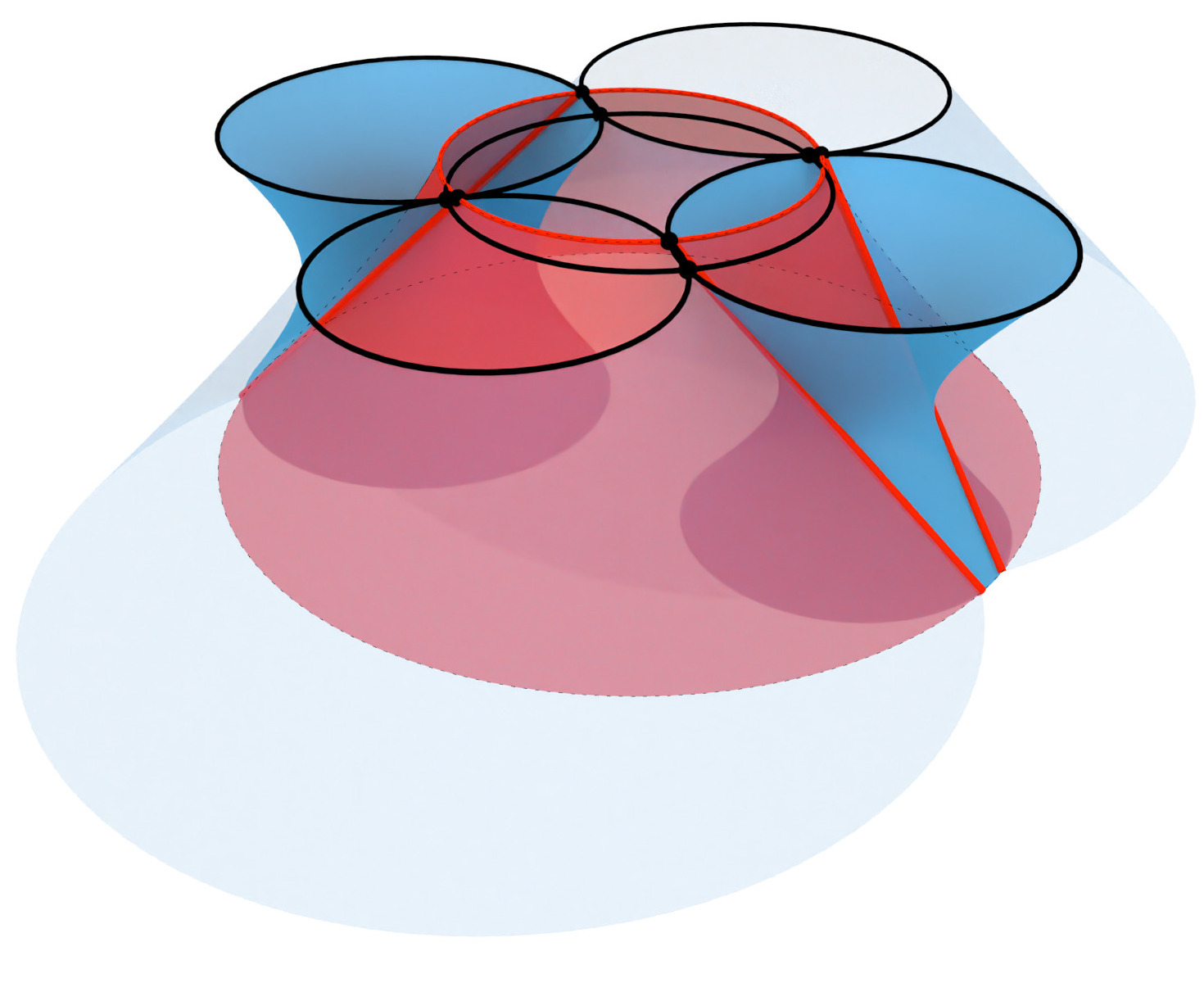}
	\end{minipage}

	\caption{Miquel dynamics in a contact congruence, and in the corresponding circle pattern.}
	\label{fig:miquel}
\end{figure}

Miquel dynamics were introduced by Ramassamy \cite{ramassamymiquel} following an idea of Kenyon as a discrete time dynamics for circle patterns. We translate the basic definition to the contact congruence setup, since we need Miquel dynamics and the generalization to contact congruences in Section~\ref{sec:isothermiccongruences} and Section~\ref{sec:subvariety}.

\begin{theorem}\label{th:ccmiquel}
  Let $S_1,S_2,S_3,S_4 \in \osp_-(\lor)$ be four timelike Lorentz spheres that are cyclically in oriented contact, and such that there is a sphere $S$ in oriented contact with $S_1,S_2,S_3,S_4$. Then there is a (generically) unique second sphere $\miq{S}$ that is also in oriented contact with $S_1,S_2,S_3,S_4$, see Figure~\ref{fig:miquel}. We say $\miq{S}$ replaces $S$ when applying \emph{Miquel dynamics} to $S$.
\end{theorem}
\begin{proof}
	The spheres $S, S_i$ correspond to points $P, P_i$ on the Lie quadric in $\RP^5$. Every sphere in contact with $S_1, S_2, S_3, S_4$ corresponds to a point in the polar complement of $P_1 \vee P_2 \vee P_3 \vee P_4$, which is a line $L \subset \RP^5$. Since $P$ is in the Lie quadric and on $L$, the line $L$ intersects the Lie quadric, and generically there is a second intersection point $\miq{P}$ which corresponds to the second sphere $\miq{S}$. 
\end{proof}

\begin{remark}
  Note that if the line $L$ in the proof of Theorem~\ref{th:ccmiquel} is a tangent to the Lie quadric (with only one contact point) then we define $\miq{S}=S$. If $L$ happens to be an isotropic line, then this is a very singular case in which all the spheres belong to a contact element. In this case $\miq{P}$ may be defined as the point on $L$ that has multi-ratio $-1$ with the other five points. However, we do not need this since we assume the generic case.
\end{remark}

We may apply Theorem~\ref{th:ccmiquel} to every vertex-star of a contact congruence $\cc$. Consider a vertex $v \in \Z^2$ and its four neighbours $v + e_1$, $v + e_2$, $v - e_1$, $v - e_2$. Then the spheres
\begin{align}
	\cc(v), \quad \cc(v + e_1), \quad \cc(v + e_2), \quad \cc(v - e_1), \quad \cc(v - e_2),
\end{align}
satisfy the assumptions of Theorem~\ref{th:ccmiquel}. Therefore the sphere $\miq \cc(v)$ exists as well. Moreover, applying Miquel dynamics to every black sphere $\ccb$ results in a new contact congruence which we denote by $\miq\cc$. For every face $f = (w_1, b_1, w_2, b_2)$ we denote by $\miq \ccf(f)$ the new isotropic line contained in 
\begin{align}
	\miq\cc(w_1) = \cc(w_1), \quad \miq\cc(b_1), \quad \miq\cc(w_2) = \cc(w_2), \quad \miq\cc(b_2) .
\end{align}
Note that since the white spheres $\cc(w_1)$, $\cc(w_2)$ remain unchanged, the new isotropic line $\miq \ccf(f)$ is actually the unique other isotropic line besides $\ccf(f)$ contained in the common tangent plane of $\cc(w_1)$ and $\cc(w_2)$.

Let us briefly explain how Miquel dynamics on contact congruences correspond to Miquel dynamics on circle patterns. By intersecting the configuration with $\eucl$ as in Theorem~\ref{th:cctocp}, it is not hard to see that Theorem~\ref{th:ccmiquel} implies \emph{Miquel's theorem} \cite{miquel}.

\begin{theorem}[Miquel's theorem]\label{th:cpmiquel}
	Let $C,C_1,C_2,C_3,C_4 \in \ci(\eucl)$ be five circles such that $C_i, C_{i+1}, C$ intersect in a common point for all four $i \in {1,2,3,4}$. Then there is a second circle  $\miq{C}$ such that $C_i, C_{i+1}, \miq C$ intersect in a common point for all four $i \in {1,2,3,4}$.
\end{theorem}

Therefore, given a circle pattern $\cp$ we denote by $\miq{\cp}$ the new circle pattern that is the intersection of $\miq{\cc}$ with $\eucl$. It turns out that $\miq{\cp}$ is the circle pattern obtained from $\cp$ via Miquel dynamics as defined in \cite{ramassamymiquel}.



\section{Isothermic congruences}  \label{sec:isothermiccongruences}

Recall that we defined an isothermic congruence (Definition~\ref{def:isothermiccongruence}) as a contact congruence $\cc$ such that $\ccmw$ is a conjugate net (Definition~\ref{def:conjugatenet}).

Consider a black vertex $b\in \Zb^2$ and the four adjacent white vertices $w_1,w_2,w_3,w_4 \in \Zw^2$. By the definition of an isothermic congruence, the four sphere centers $\ccm(w_1)$, $\ccm(w_2)$, $\ccm(w_3)$, $\ccm(w_4)$ are contained in a plane, and since the four spheres $\cc(w_1), \cc(w_2), \cc(w_3), \cc(w_4)$ are touching, there is also the circle $C(b)$ that passes through the four points of contact. The latter observation is also known as the \emph{touching coins lemma} \cite{bhsminimal}.

Lifting this configuration to (Lorentz) Möbius geometry (see Section~\ref{sec:conformallorentzgeometry}), there is a line $L(b) \subset \RP^4$ of signature $\si{+-}$ that is polar to the plane representing the lifted circle $C(b)$. Each point on $L(b)$ corresponds to a sphere that contains $C(b)$. One of the intersection points of $L(b)$ with the Lorentz Möbius quadric corresponds to the null-sphere $\cc(b)$, the other intersection point corresponds to another null-sphere $\cc'(b)$, which also contains $C(b)$.

Let $f_1$, $f_2$, $f_3$, $f_4$ be the four faces adjacent to $b$. Because $\ccb(b)$ is a null-sphere, all four isotropic lines $\ccf(f_1)$, $\ccf(f_2)$, $\ccf(f_3)$, $\ccf(f_4)$ are contained in $\ccb(b)$ and pass through $\ccm(b)$. Moreover, the four common tangent planes of $\ccmw(w_i)$ with $\ccmw(w_{i+1})$ contain the axis of the circle $C(b)$. Hence, the four other isotropic lines $\miq{\ccf}(f_i)$ as defined in Section~\ref{sec:miquel} also intersect in the axis of $C(b)$. Therefore they are contained in $\ccb'(b)$, and the null-sphere $\ccb'(b)$ is actually the null-sphere $\miq{\ccb}(b)$ obtained by Miquel dynamics, see Theorem~\ref{th:ccmiquel}. Consequently, $\ccb'(b) = \miq{\cc}(b)$ is also in contact with the four spheres $\ccw(w_i)$. These observations lead to the following characterization of isotropic congruences.

\begin{lemma} \label{lem:twopointcharacterization}
  Let $\ccw: \Zw^2 \rightarrow \osp_-$ be such that for all adjacent $w,w'\in \Zw^2$ the spheres $\ccw(w)$, $\ccw(w')$ are in contact. Then there exists an isotropic congruence $\cc$ such that $\ccw$ is its restriction to $\Zw^2$ if and only if the four spheres of $\ccw$ around each white face $f$ intersect in two points $P_f$ and $\tilde P_f$.
\end{lemma}
\begin{proof}
	If an oriented sphere contains $P_f$ (resp.~$\tilde P_f$) then it is also in oriented contact with the null-sphere centered at $P_f$ ($\tilde P_f$).
\end{proof}

Of course, as discussed above if there exists one isotropic congruence $\cc$ that restricts to $\ccw$ then there is a second isotropic congruence $\miq{\cc}$ that also restricts to $\ccw$.

Consider an isothermic congruence $\cc$ and the cycle pattern $\lcp$ that is the projection of $\cc$. Since $\cc$ is the special case of a null congruence, $\lcp$ will be the special case of an incircular cycle pattern.  As before, consider a black vertex $b$, the four adjacent white vertices $w_1,w_2,w_3,w_4$ and let $f_{i,i+1}$ be the face adjacent to $w_i$ and $w_{i+1}$. The oriented line $\lcpf(f_{12})$ is a tangent to both $\lcp(w_1)$ and $\lcp(w_2)$, and is also the projection of $\ccf(f_{12})$. Let $\miq{\lcp}$ be the projection of $\miq{\cc}$. Then $\miq{\lcpf}(f_{12})$ is also a common tangent of $\lcp(w_1)$ and $\lcp(w_2)$, in fact it is \emph{the} other (oriented) tangent of the two oriented circles. And since the four isotropic lines $\miq\cc(f_{12})$, $\miq\cc(f_{23})$, $\miq\cc(f_{34})$ and $\miq\cc(f_{41})$ intersect in a point, so do the four other tangents $\miq\lcp(f_{12})$, $\miq\lcp(f_{23})$, $\miq\lcp(f_{34})$ and $\miq\lcp(f_{41})$. As a result, $\miq\lcp$ is also an incircular cycle pattern! Let us formalize this.

\begin{definition}
	An \emph{isothermic cycle pattern} is an incircular cycle pattern $\lcp$ such that $\miq\lcp$ is also an incircular cycle pattern. 
\end{definition}

As a consequence of the discussion above, we obtain the following corollary.

\begin{corollary}
	Every isothermic cycle pattern lifts to an isothermic congruence, and every isothermic congruence projects to an isothermic cycle pattern.
\end{corollary}

As discussed in Section~\ref{sec:laguerrecp}, incircular cycle patterns correspond to incircular nets, which in turn correspond to circle packings. Analogous to the preceding discussion, we consider a circle packing $\cp$ to be an isothermic circle packing if $\miq{\cp}$ is also a circle packing (where we defined $\miq{\cp}$ as replacing all circles at black vertices by Miquel dynamics). 

Note that in an incircular net $\cpmb$, two adjacent incircles have a common tangent. This tangent may be of ``outer type'', if it does not intersect the segment joining the centers of the two incircles, or otherwise of ``inner type''. An incircular net corresponds to an isothermic cycle pattern if and only if the other tangents of the same type also form an incircular net. It is an advantage of the formulation in terms of cycle patterns that there is no need to discuss the different types of tangents, since this is automatically encoded in the orientations of the lines and oriented circles.

\begin{remark}
	If $\cp$ is a circle packing, then we defined $\miq{\cp}$ as Miquel dynamics applied to all circles at black vertices. On the other hand, if we apply Miquel dynamics to all circles at white vertices, the new circle pattern is equal to $\cp$ since the black circles are touching (see also \cite{klrr}). If $\cp$ is also an isothermic circle packing, then $\miq{\cp}$ is also a circle packing. In general, when iterating Miquel dynamics one alternates between replacing black and white circles. As a consequence, the Miquel dynamics sequence of an isothermic circle pattern $\cp$ is
	\begin{align}
		\dots \rightarrow \cp  \rightarrow \cp \rightarrow \miq{\cp} \rightarrow \miq{\cp} \rightarrow \cp \rightarrow \cp  \rightarrow \miq{\cp} \rightarrow  \miq{\cp} \rightarrow \dots \label{eq:miqsequence}
	\end{align}
	Therefore the sequence is a special case of a four periodic Miquel dynamics sequence.

  Moreover, it is not hard to see that Miquel dynamics induces the same type of sequence on $\lorlift{\cp}$, that is on an isothermic congruence. Additionally, note that it is only $\ccb$ that is changing, while $\ccw$ is fixed. Therefore, what we consider to be the actual discrete surface $\ccmw$, is \emph{constant} under Miquel dynamics. An interesting question is therefore: does Miquel dynamics discretize a continuous surface flow that fixes isothermic surfaces?
\end{remark}

\section{S-isothermic nets} \label{sec:isothermicnets}

Recall that in Definition~\ref{def:sisothermic} (following \cite{bhsminimal}) we defined an S-isothermic net $\iso$ as a collection of maps $\isow$, $\isob$, $\isof$ such that
\begin{enumerate}
	\item $\isow: \Zw^2 \rightarrow \osp_-$ maps to oriented timelike spheres,
	\item $\isob: \Zb^2 \rightarrow \ci_+$ maps to spacelike circles,
	\item $\isof: F(\Z^2) \rightarrow \lor$ maps to points,
\end{enumerate}
such that for each face the corresponding spheres of $\isow$ are in oriented contact and intersect the incident circles of $\isob$ orthogonally in the points of $\isof$.

\begin{remark}
	There are generalizations of S-isothermic nets \cite{bsisothermicmoutard,bsddgbook} to non-touching cases. In the touching cases, the way the orientations of the spheres are introduced depends on the reference. In \cite{bsisothermicmoutard,bsddgbook}, the orientations are chosen such that a specific T-net lift exists. On the other hand in \cite{bhsminimal}, orientations are not mentioned. However, \cite{bhsminimal} require that the touching points around each circle are in counterclockwise order, which implies that an appropriate orientation as in Definition~\ref{def:sisothermic} exists.
\end{remark}

Let us prove the correspondence between S-isothermic nets and isothermic congruences stated in Theorem~\ref{th:isocongruencetoisonet}. 

\begin{proof}[Proof of Theorem~\ref{th:isocongruencetoisonet}]
	Let $\cc$ be an isothermic congruence. In Section~\ref{sec:isothermiccongruences}, we introduced a circle $C(b)$ for each black vertex $b\in \Zb^2$. The circle $C(b)$ intersects each adjacent sphere orthogonally, therefore setting $\isob(b) = C(b)$ and $\isow(w) = \cc(w)$ for every black vertex $b\in \Zb^2$ and white vertex $w\in \Zw^2$ yields an S-isothermic net. 
	
  Conversely, let $\iso$ be an S-isothermic net and let us construct an isothermic congruence $\cc$. We set $\ccw = \isow$. Now, consider the situation in a face $f$, where we know the two oriented timelike spheres $\ccw(w)$, $\ccw(w')$ and the two circles $\isob(b)$, $\isob(b')$ as well as the touching point $\isof(f)$. The two axes of the circles are contained in the common tangent plane of the two spheres, because they intersect the two timelike spheres orthogonally in the touching point. Next, we choose $\ccf(f)$ as one of the two oriented isotropic lines in the common tangent plane of the spheres. This line intersects each of the axes of the two circles in a point, and we choose these points to be $\ccmb(b)$ and $\ccmb(b')$, which also defines the null-spheres $\ccb(b)$, $\ccb(b')$. 
  By construction, these null-spheres contain $\ccf(f)$ (and $\isob(b)$), as required. 
  
  Thus, we can construct $\cc$ in each face up to the choice of one of the two isotropic lines. It remains to show we can make these choices consistently. To see this, note that on each circle axis there are only two points that define a null-sphere which contains $\isob(b)$. These two points correspond to the choice of isotropic line. As a result, if we choose the isotropic line in some initial face $f_0$, the choice of isotropic line in each adjacent face is determined. Moreover, around each white vertex $w$ the resulting four isotropic lines are of alternating type in $\ccw(w)$, which shows that the choices for $\ccf$ and therefore for $\ccb$ are automatically consistent.
%
%
\end{proof}

Note that the other choice of $\ccf(f_0)$ just leads us to the construction of $\miq{\cc}$ instead of $\cc$.

Let us now turn towards dualization of S-isothermic nets. The K{\oe}nigs dual of a planar quad $P_1,P_2,P_3,P_4 \in \R^n$ is, up to scaling and translation, the quad $P^*_1,P^*_2,P^*_3,P^*_4 \in \R^n$ such that the four sides are parallel and such that the diagonals satisfy
\begin{align}
	&P^*_1 \vee P^*_3 \parallel P_2 \vee P_4, & &P^*_2 \vee P^*_4 \parallel P_1 \vee P_3,
\end{align}
that is, opposite diagonals are parallel.

For a general conjugate net $a$ (which consists of planar quads, Definition~\ref{def:conjugatenet}) it is in general not possible to find a so-called K{\oe}nigs dual conjugate net $a^*$, that is a net $a^*$ such that each quad in $a^*$ is K{\oe}nigs dual to the corresponding quad in $a$. 

\begin{definition}
	A conjugate net $a$ is called a \emph{K{\oe}nigs net} if it admits a K{\oe}nigs dual $a^*$ \cite{bskoenigs}.
\end{definition}

Note that if $a^*$ is a K{\oe}nigs dual of $a$, then $a$ is also a K{\oe}nigs of $a^*$. Therefore if $a$ is a K{\oe}nigs net so is $a^*$.

The following was shown in \cite{bhsminimal}.

\begin{lemma}\label{lem:koenigs}
	If $\iso$ is an S-isothermic net, then $\isomw$ is a K{\oe}nigs net.
\end{lemma}
\begin{proof}
	A conjugate net that admits a family of pairwise touching inscribed conics is a K{\oe}nigs net \cite{bfkoenigs}.
\end{proof}

Due to the identification of Theorem~\ref{th:isocongruencetoisonet}, if $\cc$ is an isothermic congruence then $\ccmw$ is a K{\oe}nigs net as well.

\begin{remark}
	The K{\oe}nigs net property is invariant under parallel projections, since dualization is preserved. As a consequence, if $\cp$ is an isothermic circle packing, then $\cpmw$ is a K{\oe}nigs net, since $\cpmw$ is the projection of the center net $\ccmw$ of an isothermic congruence.
\end{remark}

The center net $\ccmw$ of an isothermic congruence has a K{\oe}nigs dual $\ccmw^*$, however it is not immediately clear that $\ccmw^*$ belongs to an isothermic congruence as well. To show this, let us borrow another technique from \cite{bhsminimal}.

We call a planar quad $P_1,P_2,P_3,P_4\in \lor$ a \emph{harmonic quad} if the four points are on a circle in non-intersecting cyclic order and if
\begin{align}
	|P_1-P_2||P_3-P_4| = |P_2-P_3||P_4-P_1|.
\end{align}
We call a conjugate net a \emph{harmonic quad net} if every quad is a harmonic quad. Harmonic quad nets are generally known as \emph{isothermic nets} \cite{bpisosurf,bsddgbook}, but we avoid this name here because there is the risk of confusion with S-isothermic nets and isothermic congruences.

Let $a: \Z^2 \rightarrow \lor$ be a harmonic quad net. We use $\d a$ to denote the \emph{(discrete) differential} of $a$, that is,
\begin{align}
	\d a(v,v') = a(v') - a(v),
\end{align}
for all adjacent vertices $v,v'\in \Z^2$. Consider the following \emph{dual differential}
\begin{align}
	\d a^*(v,v') = \pm \frac{\d a(v,v')}{|\d a(v,v')|^2},
\end{align}
where the sign is $+$ if the edge is horizontal in $\Z^2$ and $-$ otherwise.

It turns out that $\d a^*$ is closed \cite{bpisosurf}, that is, $a^*$ is well-defined (up to translation). The map $a^*$ is a harmonic quad net as well and is called the \emph{Christoffel dual} of $a$ (see Figure~\ref{fig:christoffeldual}). Moreover, $a^*$ is also the K{\oe}nigs dual of $a$, therefore $a$ and $a^*$ are also K{\oe}nigs nets.

Let $\iso$ be an S-isothermic net, and consider the combined map $\isoc$ (see Equation~\eqref{eq:combinediso}) that is the combination of $\isom$ and $\isof$. Then $\isoc$ is a harmonic quad net. This is not hard to see since the quads of $\isoc$ have two right angles and side lengths $R,R,r,r$, where $R$, $r$ are the radii of a sphere of $\isow$ and an adjacent circle of $\isob$ respectively. Therefore every S-isothermic net $\iso$ comes with a Christoffel dual S-isothermic net $\iso^*$.

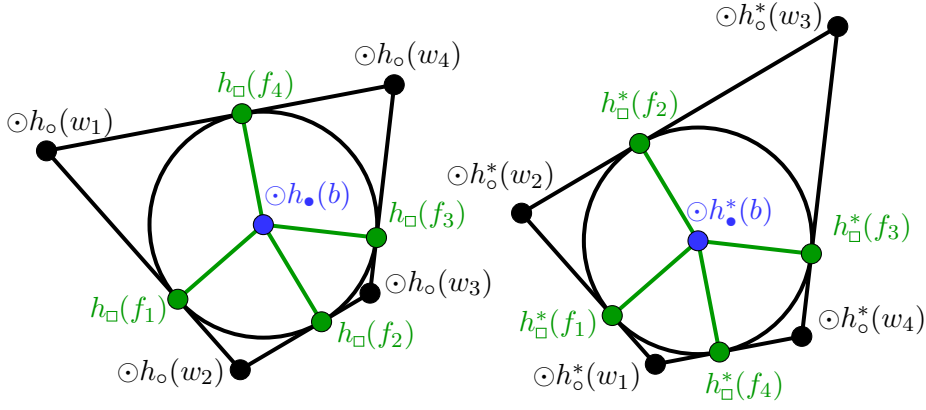
\begin{figure}
	\centering
	\begin{tikzpicture}[line cap=round,line join=round,>=triangle 45,x=1.0cm,y=1.0cm,scale=1.5]
		\definecolor{qqzzqq}{rgb}{0.,0.6,0.}
		\definecolor{ttttff}{rgb}{0.2,0.2,1.}
		\clip(-6.49007331469685,-1.6167843497060181) rectangle (1.6,2.0010235341768725);
		\draw [line width=1.5pt] (-4.24,0.02) circle (0.9987414173774751cm);
		\draw [line width=1.5pt] (-4.44273807014855,-1.2625517400125128)-- (-3.3044103136325047,-0.5856276320078151);
		\draw [line width=1.5pt] (-3.3044103136325047,-0.5856276320078151)-- (-3.0921841347984063,1.2569796075931026);
		\draw [line width=1.5pt] (-3.0921841347984063,1.2569796075931026)-- (-6.142269373424083,0.6722414186572291);
		\draw [line width=1.5pt] (-6.142269373424083,0.6722414186572291)-- (-4.44273807014855,-1.2625517400125128);
		\draw [line width=1.5pt] (-0.42746692907739503,-0.12274082086700801) circle (1.001260168648888cm);
		\draw [line width=1.5pt] (-1.9750422636041638,0.12189058615006641)-- (-0.8016207773583474,-1.213964693154792);
		\draw [line width=1.5pt] (-0.8016207773583474,-1.213964693154792)-- (0.4830962232332596,-0.967668923217702);
		\draw [line width=1.5pt] (0.4830962232332596,-0.967668923217702)-- (0.7985574922498841,1.771254339204683);
		\draw [line width=1.5pt] (0.7985574922498841,1.771254339204683)-- (-1.9750422636041638,0.12189058615006641);
		\draw [line width=1.5pt,color=qqzzqq] (-4.428046295555406,1.0008786925573707)-- (-4.24,0.02);
		\draw [line width=1.5pt,color=qqzzqq] (-4.24,0.02)-- (-4.990362066814596,-0.6391215270880599);
		\draw [line width=1.5pt,color=qqzzqq] (-3.7295229457086774,-0.8384273969458466)-- (-4.24,0.02);
		\draw [line width=1.5pt,color=qqzzqq] (-4.24,0.02)-- (-3.2478179286600612,-0.09427666470787437);
		\draw [line width=1.5pt,color=qqzzqq] (-1.1797213529895585,-0.7835246032250336)-- (-0.42746692907739503,-0.12274082086700801);
		\draw [line width=1.5pt,color=qqzzqq] (-0.42746692907739503,-0.12274082086700801)-- (-0.9392313683786432,0.7378514658689811);
		\draw [line width=1.5pt,color=qqzzqq] (-0.42746692907739503,-0.12274082086700801)-- (0.5672173513530837,-0.23730568278941422);
		\draw [line width=1.5pt,color=qqzzqq] (-0.42746692907739503,-0.12274082086700801)-- (-0.23894639480736937,-1.1060932162377335);
		\begin{normalsize}
		\draw [fill=ttttff] (-4.24,0.02) circle (2.5pt);
		\draw[color=ttttff] (-3.85,0.2785518624988409) node {$\isomb(b)$};
		\draw [fill=qqzzqq] (-4.428046295555406,1.0008786925573707) circle (2.5pt);
		\draw[color=qqzzqq] (-4.304571623750529,1.2540017697573678) node {$\isof(f_4)$};
		\draw [fill=qqzzqq] (-4.990362066814596,-0.6391215270880599) circle (2.5pt);
		\draw[color=qqzzqq] (-5.418284195106618,-0.733940446301149) node {$\isof(f_1)$};
		\draw [fill=qqzzqq] (-3.7295229457086774,-0.8384273969458466) circle (2.5pt);
		\draw[color=qqzzqq] (-3.2513761257407803,-0.9561948555499271) node {$\isof(f_2)$};
		\draw [fill=qqzzqq] (-3.2478179286600612,-0.09427666470787437) circle (2.5pt);
		\draw[color=qqzzqq] (-2.8,0.10568732197201337) node {$\isof(f_3)$};
		\draw [fill=black] (-6.142269373424083,0.6722414186572291) circle (2.5pt);
		\draw[color=black] (-6.0208695618383175,0.9206201558842003) node {$\isomw(w_1)$};
		\draw [fill=black] (-3.0921841347984063,1.2569796075931026) circle (2.5pt);
		\draw[color=black] (-2.971045168257859,1.5132985805476091) node {$\isomw(w_4)$};
		\draw [fill=black] (-4.44273807014855,-1.2625517400125128) circle (2.5pt);
		\draw[color=black] (-5.05,-1.2895764694230947) node {$\isomw(w_2)$};
		\draw [fill=black] (-3.3044103136325047,-0.5856276320078151) circle (2.5pt);
		\draw[color=black] (-2.7,-0.5) node {$\isomw(w_3)$};
		\draw [fill=ttttff] (-0.42746692907739503,-0.12274082086700801) circle (2.5pt);
		\draw[color=ttttff] (-0.15,0.13038225633298875) node {$\isomb^*(b)$};
		\draw [fill=qqzzqq] (0.5672173513530837,-0.23730568278941422) circle (2.5pt);
		\draw[color=qqzzqq] (1.1,-0.005439882652375758) node {$\isof^*(f_3)$};
		\draw [fill=qqzzqq] (-0.23894639480736937,-1.1060932162377335) circle (2.5pt);
		\draw[color=qqzzqq] (-0.0200005121213013,-1.4) node {$\isof^*(f_4)$};
		\draw [fill=qqzzqq] (-0.9392313683786432,0.7378514658689811) circle (2.5pt);
		\draw[color=qqzzqq] (-0.9337130834773901,1.0811372292305401) node {$\isof^*(f_2)$};
		\draw [fill=qqzzqq] (-1.1797213529895585,-0.7835246032250336) circle (2.5pt);
		\draw[color=qqzzqq] (-1.65,-0.8574151181060258) node {$\isof^*(f_1)$};
		\draw [fill=black] (-1.9750422636041638,0.12189058615006641) circle (2.5pt);
		\draw[color=black] (-2.1561123343456714,0.45) node {$\isomw^*(w_2)$};
		\draw [fill=black] (-0.8016207773583474,-1.213964693154792) circle (2.5pt);
		\draw[color=black] (-1.4,-1.3266188709645577) node {$\isomw^*(w_1)$};
		\draw [fill=black] (0.4830962232332596,-0.967668923217702) circle (2.5pt);
		\draw[color=black] (1.1,-0.8203727165645627) node {$\isomw^*(w_4)$};
		\draw [fill=black] (0.7985574922498841,1.771254339204683) circle (2.5pt);
		\draw[color=black] (0.2,1.8590276616012642) node {$\isomw^*(w_3)$};
		\end{normalsize}
	\end{tikzpicture}
	\caption{Local Christoffel dual $\iso^*$ of an S-isothermic net $\iso$. All corresponding edges (black and green) are parallel.}
	\label{fig:christoffeldual}
\end{figure}

Moreover, because corresponding edges of $\isoc$ and $\isoc^*$ are parallel, it follows that for any two white vertices $w,w'$ adjacent to a common face holds
\begin{align}
	\d \isomw^*(w,w') = \pm (R^{-1}(w) + R^{-1}(w')) \frac{\d \isomw(w,w')}{|\d \isomw(w,w')|}.
\end{align}
This provides us with a direct way to compute the Christoffel dual $\isomw^*$ without using $\isoc$.

\begin{remark}
	Note that for an S-isothermic net $\iso$ the K{\oe}nigs dual of $\isomw$ coincides with $\isomw^*$ obtained as the restriction of the Christoffel dual $\isoc^*$. 
\end{remark}

Due to the identification of Theorem~\ref{th:isocongruencetoisonet}, we also get a Christoffel dual $\cc^*$ of every isothermic
 congruence $\cc$. Since that identification is 2:1 from S-isothermic nets to contact congruences, we make the construction of the Christoffel dual $\cc^*$ more specific.

Let $\cc$ be an isothermic congruence and consider the combined map
\begin{align}
	\ccc = \ccm \cup \isof: \Z^2 \simeq V(\Z^2) \cup F(\Z^2) \rightarrow \lor, \label{eq:combinedcc}
\end{align}	
that is the combination of $\ccm$ and $\isof$. In comparison with $\isoc$ we substituted $\isomb$ with $\ccmb$, therefore $\ccc$ is not a conjugate net. However, for a face $f$ and an adjacent white vertex $w$ or black vertex $b$ it is possible to integrate 
\begin{align}
	\d \ccc^*(f,w) &= \pm \frac{\d \ccc(f,w)}{R(w)^2}, &
	\d \ccc^*(f,b) &= \pm \frac{\d \ccc(f,b)}{r(b)^2}, \label{eq:nullspheredual}
\end{align}
where $r(b)$ is the radius of the incircle of the corresponding quad of $\ccmw$ at $b$. The difference to the formula for $\d \isoc^*$ is that we cannot divide by $|\d \ccc(f,b)|$, since that length is zero. Instead, this formula works since the triangle $\isof(f),\isomb(b),\ccmb(b)$ is a right-angled triangle and 
\begin{align}
	|\isof(f) - \isomb(b)|^2 = - |\isomb(b) - \ccmb(b)|^2.
\end{align}
The discrete differential $\d \ccc^*$ is consistent in the sense that the integrated net $\ccc^*$ is indeed the combination of $\ccm^*$ and $\isof^*$.



\section{$X$-variables}  \label{sec:xvariables}

In this section we study the $X$-variables, that is, the quantities that can be read off conical nets that define the corresponding dimer model in statistical mechanics, see \cite{amiquel,klrr,clrdimer}.

\subsection{Definitions} \label{sec:xdef}

\begin{definition}\label{def:xvar}
	Let us identify $\R^2 \simeq \C$ and consider a conical net $\cpm: \Z^2 \rightarrow \C$. The $X$-variables $X: \Z^2\rightarrow \R$ are defined as
	\begin{align}	
		X(v) = -\frac{(\cpm(v + e_1)-\cpm(v))(\cpm(v -e_1)-\cpm(v))}{(\cpm(v + e_2)-\cpm(v))(\cpm(v -e_2)-\cpm(v))},
	\end{align}	
	for all $v\in \Z^2$.
\end{definition}

\begin{remark}
  It is not hard to see that the $X$-variables are indeed real-valued. In a conical net, opposing angles at a vertex $v$ add up to an odd or even multiple of $\pi$. Hence, $X(v)$ is positive or negative respectively.
\end{remark}

The $X$-variables may be expressed in the cyclographic lift $\cyclift{\cpm}$ as well, as discovered by Chelkak \cite{chelkakliftxvars} and expressed in the next lemma.
\begin{lemma} \label{th:xcyclo} \label{lem:xcyclo}
	Let $\cyclift{\cpm}$ be the cyclographic lift of a conical net $\cpm$. The $X$-variables of $\cpm$ satisfy
	\begin{align}\label{eq:xvarLaguerre}
	  X(v) = \frac{\sca{\cyclift{\cpm}(v+e_1)-\cyclift{\cpm}(v -e_1),\cyclift{\cpm}(v + e_1)-\cyclift{\cpm}(v - e_1)}_\cyc}{\sca{\cyclift{\cpm}(v+e_2)-\cyclift{\cpm}(v -e_2),\cyclift{\cpm}(v + e_2)-\cyclift{\cpm}(v - e_2)}_\cyc},
	\end{align}
	for all $v\in \Z^2$.
\end{lemma}

\begin{proof}
	Without loss of generality we assume that $\cpm(v) = 0$. For this proof, we use the map $h: \Z^2 \rightarrow \C^{1,1} \simeq \R^{2,2}$ as defined in Section~\ref{sec:cyclift}, which coincides with $\cyclift p$. By choosing the base-face $f_0$ of the origami map adjacent to $v$,$v+e_1$ and $v+e_2$, we obtain that
	\begin{align}
		o(v+e_1) &= \cpm(v+e_1), & o(v+e_2) &= \cpm(v+e_2),  \\
		o(v-e_1) &= \overline \cpm(v-e_1) \frac{\cpm(v+e_2)}{\overline \cpm(v+e_2)}, &  o(v-e_2) &= \overline \cpm(v-e_2) \frac{\cpm(v+e_1)}{\overline \cpm(v+e_1)}.
	\end{align}
	A short calculation shows that the right-hand side in Equation~\eqref{eq:xvarLaguerre} is equal to
	\begin{align}
		-\frac{|\cpm(v+e_1)|^2}{|\cpm(v+e_2)|^2}\frac{\mathrm{Im} (\cpm(v+e_2) \overline \cpm(v-e_1))  }{\mathrm{Im} (\overline \cpm(v-e_2)\cpm(v+e_1))},
	\end{align}
	Let $\alpha$ be the angle from $\cpm(v+e_2)$ to $\cpm(v-e_1)$, and $\beta$ be the angle from $\cpm(v-e_2)$ to $\cpm(v+e_1)$, then the formula above reads
	\begin{align}
	\frac{|\cpm(v+e_1)||\cpm(v-e_1)|}{|\cpm(v+e_2)||\cpm(v-e_2)|}\frac{\sin(\alpha) }{\sin(\beta)}.
	\end{align}
	If $X(v) > 0$, then $\beta =  \pi - \alpha$ and the claim follows. If $X(v) < 0$, then $\beta = - \alpha$ and the claim follows as well.
\end{proof}

In Section~\ref{sec:laguerre} we explained how (squared) distances in $\cyc$ correspond to (squared) tangential distances in $\lor$. Therefore it is also possible to express the $X$-variables via tangential distances of oriented spheres, as the squared tangential distance of $\lorlift{\cp}(v+e_1)$ to $\lorlift{\cp}(v-e_1)$ divided by the squared tangential distance of $\lorlift{\cp}(v+e_2)$ to $\lorlift{\cp}(v-e_2)$.

\subsection{Ising subvariety}\label{sec:isingsub}

Let us consider the special case that $\cp$ is a circle packing, therefore $\cpmb$ is an incircular net and $\lorlift{\cp}$ is a null congruence.
The tangential distance of two null-spheres is just the Lorentz distance of their centers. As a result, in this case the $X$-variable at a white verticex $w$ may be calculated from the Lorentz lift as
\begin{align}
	X(w) = \frac{|\lorlift{\cpm}(w + e_1) - \lorlift{\cpm}(w - e_1)|^2_\lor}{|\lorlift{\cpm}(w + e_2) - \lorlift{\cpm}(w - e_2)|^2_\lor}. \label{eq:xslorlift}
\end{align}

Moreover, for the conical net of a circle packing, it is known \cite{klrr} that the $X$-variables are in the so-called \emph{Ising subvariety} \cite{kenyonpemantle}. More explicitly, this means that for each black vertex $b$ we obtain
\begin{align}
	X^2(b) &= \frac{(1 + X(b + e_2))(1 + X(b - e_2))}{(1 + X^{-1}(b + e_1))(1 + X^{-1}(b - e_1))}. \label{eq:isingsub}	
\end{align}
Note that the quantities on the right hand side are $X$-variables of white vertices, therefore all black $X$-variables are determined by the white $X$-variables and thus by Equation~\eqref{eq:xslorlift}.

\subsection{Miquel dynamics}

In Section~\ref{sec:miquel}, we explained how Miquel dynamics applied to all black circles of a circle pattern $\cp$ produces a new circle pattern $\miq{\cp}$. Of course, it is also possible to apply Miquel dynamics to all white circles. As we need to distinguish the two possibilities in the following, let us denote the former by $\miqb \cp$ and the latter by $\miqw \cp$.

 Let us denote by $\miqb{X}$ the $X$-variables of $\miqb{\cp}$, and by $\miqw{X}$ the $X$-variables of $\miqw{\cp}$. In \cite{amiquel,klrr} it was shown how $\miqb{X}$ (and $\miqw{X}$) is determined by $X$, albeit with slightly different conventions. More explicitly, for every black vertex $b$ we have 
\begin{align}
	\miqb{X}(b) = X(b). 
\end{align}
Additionally, for every white vertex $w$ we have
\begin{align}
	\miqb{X}(w) = X^{-1}(w) \frac{(1 + X(w + e_2))(1 + X(w - e_2))}{(1 + X^{-1}(w + e_1))(1 + X^{-1}(w - e_1))}.  \label{eq:miqbw}
\end{align}
The formulas for Miquel dynamics applied to white circles are almost the same, they are
\begin{align}
	\miqw{X}(w) &= X(w),\\ 
	\miqw{X}(b) &= X^{-1}(b) \frac{(1 + X(b + e_2))(1 + X(b - e_2))}{(1 + X^{-1}(b + e_1))(1 + X^{-1}(b - e_1))}. \label{eq:miqwb}
\end{align}
Since Miquel dynamics applied to the white circles of a circle packing changes nothing, it is not surprising that Equation~\eqref{eq:isingsub} is the stationary constraint on Equation~\eqref{eq:miqwb}. This observation was also used by George in \cite{georgespectralising}.

\subsection{Isothermic subvariety}\label{sec:subvariety}

Consider an isothermic congruence $\cc$ (Section~\ref{sec:isothermiccongruences}) and the corresponding circle packing $\cp$ via Theorem~\ref{th:nullcctopacking}. Since $\cc$ is also a null congruence, the $X$-variables satisfy the constraint of Equation~\eqref{eq:isingsub}. In particular, this means that the statistics are defined by the restriction of the $X$-variables to the white vertices.

Since statistics are only defined if $X(v) > 0$ for all $v \in \Z^2$, we make this assumption in the following. In particular, this means we can take unique positive square-roots of $X$-variables.

Applying Miquel dynamics to the black circles, we obtain another isothermic congruence $\miqb{\cc}$ and circle packing $\miqb{\cp}$, and the new $X$-variables $\miqb{\cc}$ at white vertices are given by Equation~\eqref{eq:miqbw}, while the variables at black vertices remain unchanged. 

Moreover, since $\miqb{\cc}$ is itself a null-congruence, the variables $\miqb{X}$ also satisfy Equation~\eqref{eq:isingsub}, that is,
\begin{align}
	X^2(b) &= \frac{(1 + \miqb{X}(b + e_2))(1 + \miqb{X}(b - e_2))}{(1 + \miqb{X}^{-1}(b + e_1))(1 + \miqb{X}^{-1}(b - e_1))}. \label{eq:isingsubtwo}	
\end{align}
Note that $\miqb{X}(b + e_1)$ involves the four variables
\begin{align}
	X(b + 2 e_1), \quad X(b), \quad X(b + e_1 + e_2), \quad X(b + e_1 - e_2),
\end{align}
which are all defined on black vertices. For example, using Equation~\eqref{eq:isingsub} the variable $X(b + 2 e_1)$ may be expressed by the four variables
\begin{align}
	X(b + 3 e_1), \quad X(b + e_1), \quad X(b + 2e_1 + e_2), \quad X(b + 2e_1 - e_2),
\end{align}
which are all defined on white vertices. The left hand side of Equation~\eqref{eq:isingsubtwo} can be substituted by the right hand side of Equation~\eqref{eq:isingsub}. Combining all, we obtain an equation involving the 16 variables
\begin{align}
	\{X(b + y) \ | \ y_1 + y_2 \in 2 \Z \mbox{ and } |y_1| + |y_2| \leq 3 \},
\end{align}
all defined at white vertices, see also Figure~\ref{fig:isoconstraint}. We do not give the explicit equation since it is large, unwieldy, and we do not use it.

Let us phrase the previous observations as a theorem.

\begin{figure}
	\centering
	\includegraphics[scale=0.8]{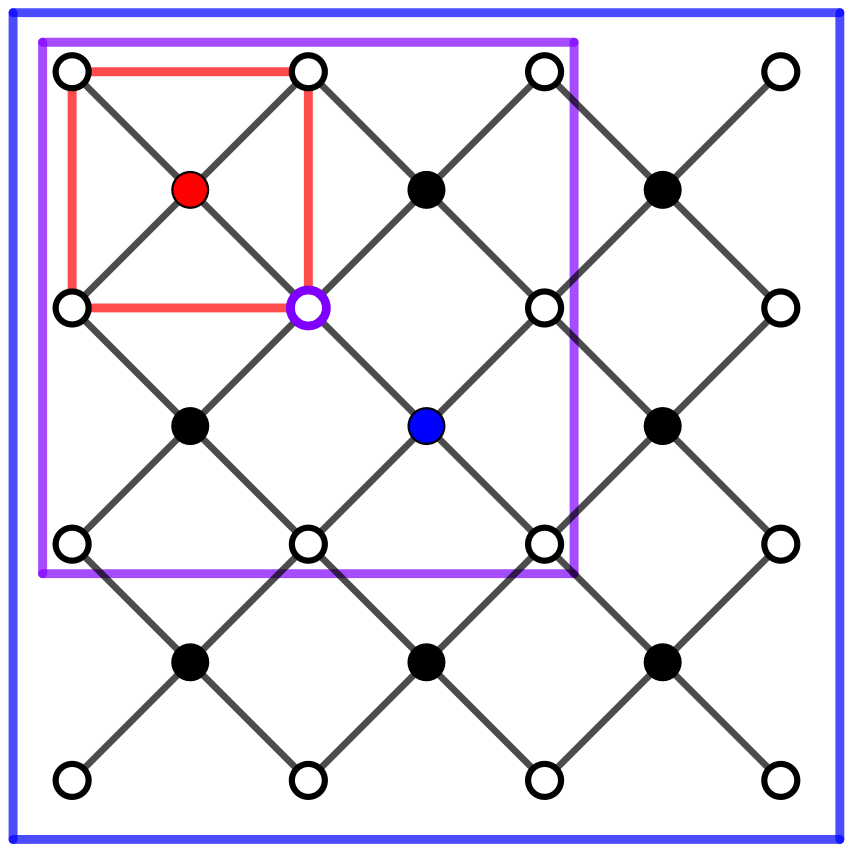}
  \caption{A (diagonal) $4\times4$ patch of $\Z^2$. The red $X_\bullet$ variable is determined by all $X_\circ$ variables in the red box. The violet $\miqb{X_\circ}$ variable is determined by all $X_\circ$ variables in the violet box. All $X_\circ$ variables in the blue box satisfy a constraint, since the blue $X_\bullet$ variable is twice determined.}
	\label{fig:isoconstraint}
\end{figure}

\begin{theorem}
	The $X$-variables of an isothermic congruence are in the subvariety defined by Equations~\eqref{eq:isingsub}, \eqref{eq:miqbw} and \eqref{eq:isingsubtwo}.
\end{theorem}

It is tempting to call this subvariety the \emph{isothermic subvariety}, but it is not completely clear if this is appropriate, due to the following arguments.
While every circle packing has $X$-variables in the Ising subvariety, not every circle pattern with $X$-variables in the Ising subvariety is a circle packing. Therefore, it is not a priori clear whether every null congruence with $X$-variables in the isothermic subvariety is an isothermic congruence.

\begin{remark}
	Note that the constraints relate $X$-variables at white vertices of a (diagonal) $4 \times 4$ patch, see Figure~\ref{fig:isoconstraint}. In particular, this shows that the Cauchy data is one-dimensional, in the sense that we may only prescribe the $X$-variables for ``thickened'' coordinate axes, and then the whole of the $X$-variables is determined by the subvariety equations. This is in contrast to the Ising subvariety, where the Cauchy data is still two-dimensional. Note that isothermic surfaces, both discrete and smooth, are also determined by one-dimensional Cauchy data, so this is not an unexpected phenomenon.
\end{remark}

It would also be very interesting to find some sort of parametrization of the isothermic subvariety, be it with rational or other functions.

\section{Transformation behaviour} \label{sec:transformations}

\subsection{Invariant objects}

We have defined contact congruences as a pair of maps of oriented spheres $\cc$ and oriented isotropic lines $\ccf$ (Definition~\ref{def:contactcongruence}). A consequence is that $\cc$ consists of oriented spheres such that adjacent spheres are in contact. In fact, generically this property is already characterizing: if the oriented spheres of a map $\cc:\Z^2 \rightarrow \osp_-$ are in contact per face, then the isotropic lines $\ccf$ exist and are uniquely determined. Since this characterization is manifestly \emph{Lie invariant} (see Section~\ref{sec:lie}), any Lie transform of a contact congruence is again a contact congruence. Therefore, contact congruences are natural objects of Lie geometry. Note that the Lie transformations form a \emph{larger} group than the isometries of $\cyc \simeq \R^{2,2}$, indeed they are the conformal transformations of $\R^{2,2}$.

Null congruences are the special case of contact congruences where all the black spheres have radius zero. Since Lie transformations generally do not preserve null-spheres, null congruences are \emph{not} invariant under Lie transformations. However, Möbius transformations also preserve sphere contact and map null-spheres to null-spheres. Hence, null congruences are naturally \emph{Möbius invariant}. 

Isothermic congruences correspond to S-isothermic nets, and in \cite{bhsminimal} they already showed that these are also Möbius invariant. This is due to the added constraint on null congruences that the four timelike spheres of white vertices adjacent to a fixed black vertex are orthogonal to a circle, which is a Möbius invariant condition. This is in agreement with the smooth theory, where isothermic surfaces are also Möbius invariant.

\begin{remark}
	Since we have shown that circle patterns $\cp$ and contact congruences $\lorlift{\cp}$ are closely related, one may also ask how the above mentioned transformations act on $\cp$, and conversely how the actions of the various sphere geometries on $\eucl$ are captured by the transformations of $\lor$. For example, every Möbius transformation of $\eucl$ may be (non-uniquely) extended to a Möbius transformation of $\lor$. But in general, a Möbius transformation of $\lor$ does not induce a Möbius transformation on the circle pattern in $\eucl$. Instead, we expect that it acts in some manner as a Lie transformation in $\eucl$. More generally, one may also ask how Lie transformations of $\lor$ act on the circle pattern in $\eucl$. Similar questions may also be asked in the special case of null congruences. Answering these questions is part of ongoing research by the authors and will be treated in a future publication.
\end{remark}

\subsection{Conformal $X$-variables}\label{sec:confxvariables}

Before we discuss the invariance of the $X$-variables in the different sphere geometries, let us relate recent work in \cite{amiquelmobius} to contact congruences. We keep this discussion brief, since it mainly serves to motivate future research and open questions that we discuss in the next subsection.

Let $\cc$ be a contact congruence and recall that $\miq{\cc}$ is the image of $\cc$ under Miquel dynamics, see Section~\ref{sec:miquel}. The spheres $\ccw$ are the same in $\cc$ and $\miq{\cc}$, but the spheres $\ccb$ and the oriented lines $\ccf$ are all different after applying Miquel dynamics. Consider a white vertex $w$ and in cyclic order the four adjacent faces $f_1$, $f_2$, $f_3$, $f_4$. The four isotropic lines $\ccf(f_1)$, $\miq{\ccf}(f_2)$, $\ccf(f_3)$ and $\miq{\ccf}(f_4)$ are contained in $\ccw(w)$ and are all four of the same type. Intersecting these four lines with $\eucl$ yields the four  points used in \cite{amiquelmobius} to define the $X$-variables that are invariant under Möbius transformations of $\eucl$, which we call \emph{conformal $X$-variables} in the following.

Note that the cross-ratio of four pairwise skew lines contained in a one-sheeted hyperboloid in $\RP^3$ is a projective invariant. It may be defined as the cross-ratio in any conic section of the hyperboloid. Consequently, it is also invariant under Lorentz Möbius transformations in the special case that the hyperboloid is a timelike Lorentz sphere. Moreover, it coincides with the conformal $X$-variables. Hence the conformal $X$-variables are Möbius invariant quantities associated to contact congruences.

\subsection{Invariant variables}

As the $X$-variables (see Section~\ref{sec:xvariables}) can be expressed as certain ratios of tangential distances, they are naturally Laguerre invariant. On the other hand, the conformal $X$-variables that we briefly discussed in Section~\ref{sec:confxvariables} are Möbius invariant. This leads to odd phenomena. For example, we may apply a Laguerre transformation to a null congruence, and the result will not be a null congruence anymore, but the $X$-variables will still be the same and therefore still be in the Ising subvariety. On the other hand, we may apply a Möbius transformation to an isothermic (or null) congruence, and the result will be an isothermic (or null) congruence, but the $X$-variables change.

It is possible that this is the result of a mismatch of theories. Both in the smooth and the discrete setup, there is also a theory of L-isothermic (meaning Laguerre-isothermic) surfaces \cite{szereszewskilisotherm, bsisothermicmoutard, ppyweierstrass, mnlaguerrevar, mlisothermiceq}. It is possible that the more reasonable approach would be to study the conformal $X$-variables together with isothermic congruences and the (standard) $X$-variables together with L-isothermic congruences. Unfortunately, a discrete S-L-isothermic theory does not exist yet. We should also mention that the conformal $X$-variables of null congruences are \emph{not} in the Ising subvariety, but they do still define a dimer model. The conformal $X$-variables of an isothermic congruence will still be in a subvariety due to the periodicity of Miquel dynamics, but we do not discuss the details here.

Of course, neither the $X$-variables nor the conformal $X$-variables are invariant under Lie transformations, despite contact congruences having that property. Hence one may also ask whether there is a sensible definition of Lie-$X$-variables? Natural requirements would be Lie invariance, positivity and the correct change under Miquel dynamics.


\bibliographystyle{alpha}
\bibliography{references}

\end{document}